\documentclass[twoside, reqno]{amsart}

\usepackage[left=2.5cm, right=2.5cm, top=3cm, bottom=2.5cm]{geometry}

\usepackage{hyperref}
\hypersetup{colorlinks=true, pdfstartview=FitV, linkcolor=blue,citecolor=blue, urlcolor=blue}

\usepackage[usenames]{xcolor}
\definecolor{labelkey}{rgb}{0,0,1}
\definecolor{Red}{rgb}{0.7,0,0.1}
\definecolor{Green}{rgb}{0,0.7,0}

\usepackage{amsfonts, amssymb, amsmath, amsthm, mathrsfs, bbm, cjhebrew, gensymb, textcomp, mathtools, dsfont,tikz}

\usepackage[normalem]{ulem}

\usepackage{commath, setspace, subcaption, parcolumns, multirow, multicol, accents, comment, marginnote, verbatim, empheq, enumerate, stackrel, enumitem}

\usepackage[capitalize,nameinlink,noabbrev]{cleveref}
\usepackage{graphicx}
\usepackage{epsfig}
\usepackage{psfrag}
\usepackage{float}
\usepackage[active]{srcltx}

\newtheorem{theorem}{Theorem}[section]

\newtheorem{Cor}{Corollary}[section]
\newtheorem{Thm}{Theorem}[section]

\newtheorem{Lem}[theorem]{Lemma}
\newtheorem{Rmk}{Remark}[section]

\numberwithin{equation}{section}

\newcommand{\al}{\alpha}
\newcommand{\be}{\beta}
\newcommand{\de}{\delta} 
\newcommand{\De}{\Delta}
\newcommand{\eps}{\epsilon}
\newcommand{\veps}{\varepsilon}
\newcommand{\ph}{\varphi}
\newcommand{\gam}{\gamma}

\newcommand{\kap}{\kappa}
\newcommand{\lam}{\lambda}
\newcommand{\Lam}{\Lambda}
\newcommand{\si}{\sigma}
\newcommand{\xe}{\xi-\eta}
\newcommand{\tiltht}{\widetilde{\theta}}
\newcommand{\tht}{\theta}

\newcommand{\om}{\omega}

\newcommand{\ze}{\zeta}

\newcommand{\lpj}{\triangle_j}
\newcommand{\lpk}{\triangle_k}

\newcommand{\tlpk}{\tilde{\triangle}_k}

\newcommand{\ZZ}{\mathbb{Z}}
\newcommand{\RR}{\mathbb{R}}
\newcommand{\NN}{\mathbb{N}}

\newcommand{\lb}{\langle}
\newcommand{\rb}{\rangle}
\newcommand{\goesto}{\rightarrow}

\newcommand{\Sob}[2]{\lVert#1\rVert_{#2}}
\newcommand{\nrm}[1]{\lVert#1\rVert}

\newcommand{\bdy}{\partial}

\newcommand{\til}[1]{\widetilde{#1}}

\newcommand{\Ae}{\abs{\eta}}
\newcommand{\Ax}{\abs{\xi}}
\newcommand{\Axe}{\abs{\xi-\eta}}

\newcommand{\Hdot}{\dot{H}}

\newcommand{\Acal}{\mathcal{A}}
\newcommand{\Bcal}{\mathcal{B}}

\newcommand{\Gdot}{\dot{G}}

\newcommand{\Ft}{\mathcal{F}}
\newcommand{\ld}{(\ln(I-\Delta))}

\newcommand{\ldmu}{(\ln(I-\Delta))^{\mu}}

\newcommand{\lnx}{(\ln(1+\abs{\xi}^{2}))}
\newcommand{\lnxe}{(\ln(1+\abs{\xi-\eta}^{2}))}

\newcommand{\Ba}{\mathcal{B}(2^{k-3})}

\newcommand{\An}{\mathcal{A}(2^{k-1},2^{k+1})}

\DeclareMathOperator{\supp}{supp}
\DeclareMathOperator*{\esssup}{ess\,sup}
\DeclareMathOperator*{\Div}{div}

\title[Existence, uniqueness, and smoothing for the GSQG equation in critical Sobolev spaces]{On the existence, uniqueness, and smoothing of solutions to the generalized SQG equations in critical Sobolev spaces}
\author{Michael S. Jolly, Anuj Kumar, Vincent R. Martinez}
\date{January 18, 2021}
\thanks{The research of M.S.J. and A.K. was supported in part by the NSF grant DMS-1818754. The research of V.R.M. was supported in part by the PSC-CUNY grant 62239-00 50.}
\begin{document}

\maketitle
\begin{abstract}
	This paper studies the dissipative generalized surface quasi-geostrophic equations in a supercritical regime where the order of the dissipation is small relative to order of the velocity, and the velocities are less regular than the advected scalar by up to one order of derivative. We also consider a non-degenerate modification of the endpoint case in which the velocity is less smooth than the advected scalar by slightly more than one order. The existence and uniqueness theory of these equations in the borderline Sobolev spaces is addressed, as well as the instantaneous smoothing effect of their corresponding solutions. In particular, it is shown that solutions emanating from initial data belonging to these Sobolev classes {{immediately enter a Gevrey class}}. Such results appear to be the first of its kind for a quasilinear parabolic equation whose coefficients are of higher order than its linear term;  they rely on an approximation scheme which modifies the flux in such a way that preserves the underlying commutator structure lost by having to work in the critical space setting, as well as delicate adaptations of well-known commutator estimates to Gevrey classes.
\end{abstract}

{\noindent \small {\it {\bf Keywords: surface quasi-geostrophic (SQG)
      equation, generalized SQG equation, critical space, existence, uniqueness, Gevrey regularity, smoothing effect, commutator estimates, quasilinear parabolic equation}
  }  }


\section{Introduction}
    {The main equation of interest in this paper is the two-dimensional (2D) dissipative generalized surface quasi-geostrophic (gSQG) equation given by}
	\begin{align}\label{E:dissipative-beta}
	\begin{split}
	{\partial_{t}\theta +{\gam} \Lam^{\kap}\tht+ u \cdot \nabla \theta=0,\quad 
	u=\nabla^{\perp}\psi:=(-\partial_{x_2}\psi,\partial_{x_1}\psi) ,\quad \Delta \psi=\Lambda^{\beta}\theta, \quad 0\leq\be<2}.
	\end{split}
	\end{align}
	Here,  $\theta$ represents the evolving scalar and $\psi$ its corresponding streamfunction. The operator $\Lambda$ denotes the fractional laplacian operator, $\left(-\Delta\right)^{\frac{1}{2}}$. The parameters $\gam,\kap$ are non-negative with $\kap\in(0,2]$. We assume the domain is the plane, $\RR^2$, and consider the initial value problem \eqref{E:dissipative-beta} such that $\tht(0,x)=\tht_0(x)$, where $\tht_0$ is given.  This model was first introduced in \cite{ChaeConstantinWu2012}, while its inviscid  ($\gam=0$) counterpart was studied in \cite{ChaeConstantinCordobaGancedoWu2012}. The family of equations in \eqref{E:dissipative-beta} parametrized by $\be\in[0,2)$ interpolates between the 2D incompressible Euler equation ($\be=0$) and the SQG equation ($\be=1$), and extrapolates beyond the SQG equation, $\be\in(1,2)$, to a family of active scalar equations with increasingly singular velocity. The $\be=2$ endpoint can also be considered by slightly modifying the equation for the streamfunction in \eqref{E:dissipative-beta}. The modification proposed in \cite{ChaeConstantinWu2012, ChaeConstantinCordobaGancedoWu2012} is given by 
	    \begin{align}\label{E:dissipative-log}
	        \psi=-(\ln(I-\De))^\mu\tht,\quad\mu>0.
	    \end{align}
	 We will ultimately study \eqref{E:dissipative-beta} when $\gam>0$, $\kap\in(0,1)$, and $1<\be\leq2$, where the endpoint case, $\be=2$, is modified as \eqref{E:dissipative-log}. When $\be\in(1,2)$, we establish existence and uniqueness of solutions for arbitrary initial data in $H^{\be+1-\kap}$, global existence when the corresponding homogeneous norm of the initial data is sufficiently small, and establish Gevrey regularity for the unique solution (see \cref{thm:main:beta}) with exponent arbitrarily close to optimal, that is, to a Gevrey class that is arbitrarily close to the one that the solution to the linear, parabolic part of \eqref{E:dissipative-beta}, \eqref{E:dissipative-log} naturally belongs to; the analogous results for the modified endpoint case are also developed in $H^{\si}$, for $\si>3-\kap$ (see \cref{thm:main:log}). 
	
	The case of the SQG equation ($\be=1$) models the temperature or buoyancy of a strongly stratified fluid in a rapidly rotating regime and is a fundamental equation in geophysics and meteorology (cf. \cite{Pedlosky1986}). It has received considerable attention in the last three decades especially due to the presence of mechanisms strongly analogous to those for vortex-stretching in the 3D Euler equation (cf. \cite{ConstantinMajdaTabak1994, Cordoba1998}). As a 2D hydrodynamic model, the SQG also exhibits features of 2D turbulence analogous to those exhibited by the 2D Euler equation (cf. \cite{MajdaTabak1996}). When $\gam,\kap>0$ and $\be=1$ \eqref{E:dissipative-beta} becomes the {dissipative} SQG equation. Here, one distinguishes between the subcritical ($1<\kap\leq2$), critical ($\kap=1$), and supercritical $(\kap<1)$ regimes. Global regularity has been established in the subcritical (cf. \cite{Resnick1995, ConstantinWu1999}), and critical regimes (cf. \cite{KiselevNazarovVolberg2007, CaffarelliVasseur2010, KiselevNazarov2009, ConstantinVicol2012, ConstantinTarfuleaVicol2015}). Global regularity in the supercritical regime remains an outstanding open problem, though conditional regularity (cf. \cite{ConstantinWu2009, Coti-ZelatiVicol2016}) or eventual regularity results (cf. \cite{Dabkowski2011}) are available. Nevertheless, {local well-posedness} for large data and {global well-posedness} for small data have been established in several functional spaces, including a wide-range of scaling-critical or borderline spaces, in the supercritical regime (cf. \cite{ChaeLee2003, Wu2004, ChenMiaoZhang2007, HmidiKeraani2007}), as well as the corresponding parabolic smoothing effect (cf. \cite{Dong2010, DongLi2010, Biswas2014, BiswasMartinezSilva2015}).

	    In the regime $1<\be<2$, in \cite{ChaeConstantinCordobaGancedoWu2012}, the Cauchy problem for the inviscid case ($\kap=0$) of \eqref{E:dissipative-beta} was shown to be {locally well-posed} in $H^{4}$. {This result was sharpened in \cite{HuKukavicaZiane2015}, where {local well-posedness} was established in  $H^{\be+1+\epsilon}$}, for any $\epsilon>0$. For blow-up of a closely related non-local transport equation, we refer the reader to \cite{Dong2014}. Local well-posedness in the borderline space $H^{\be+1}$ remains an outstanding open problem for these models, especially in light of recent ill-posedness results for the Euler equation (cf. \cite{BourgainLi2015, ElgindiMasmoudi2020}) and complementary results on the impossibility of uniform continuity of the solution operator (cf. \cite{BourgainLi2019, HimonasMisiolek2010, Inci2015, MisiolekYoneda2016} for Euler) and (cf. \cite{Inci2018}) for the inviscid SQG). Positive results are, however, available for rather mild regularizations up to a threshold. Such a threshold was identified in  \cite{ChaeWu2012} for the inviscid
	    generalized SQG equations in the regime $\be\in[0,1)$ and established to be sharp in \cite{Kwon2020} for the particular case of the 2D incompressible Euler equation.  In an upcoming paper by the authors (cf. \cite{JollyKumarMartinez2020b}), alternative mechanisms for recovering well-posedness are studied in the spirit of \cite{ChaeWu2012} for the full range $\be\in(1,2]$, with the $\be=2$ endpoint modified accordingly. In contrast, continuity of the solution operator in borderline spaces has been shown to hold for the   Navier-Stokes equations, (cf. most recently \cite{FarwigGigaHsu2019}). To the best of our knowledge, analogous results for similar hydrodynamic models, particularly for \eqref{E:dissipative-beta} in the regime of parameters treated here, are not known and remain an interesting unresolved issue.

    In this paper, we address the problem of existence, uniqueness of solutions, and the smoothing effect for the corresponding Cauchy problem of (\ref{E:dissipative-beta}) when $\kap\in(0,1)$ and $\be\in(1,2]$, with $\be=2$ modified as described in \eqref{E:dissipative-log}, particularly for arbitrarily large initial data belonging to the borderline Sobolev spaces, $H^{\be+1-\kap}$. These spaces are identified by imposing norm invariance under the scaling symmetry of the equations. This scaling is defined by
        \begin{align}\label{eq:scaling}
            \tht_\lam(t,x)=\lam^{\kap-\be}\tht(\lam^\kap t,\lam x).
        \end{align}
    In specific, if $\tht$ is a solution of \eqref{E:dissipative-beta} with initial data $\tht_0$, then $\tht_\lam$ is also a solution of \eqref{E:dissipative-beta} with initial data $(\tht_0)_\lam$. The homogeneous Sobolev norm, $\Sob{\cdotp}{\dot{H}^{\be+1-\kap}}$, of a solution remains invariant under \eqref{eq:scaling}. Consequently, $\dot{H}^{\be+1-\kap}$ is referred to as a scaling-critical space for \eqref{E:dissipative-beta}. Although the modification of the $\be=2$ endpoint breaks this scaling symmetry, due to the slightly supercritical nature of the velocity, we nevertheless consider by analogy the space $H^{3-\kap+\eps}$ as the borderline space corresponding to this case. 
    
    The existence and uniqueness of solutions for large initial data in critical Sobolev spaces for the critical and supercritical SQG equation, i.e., $\be=1$ case, was established by Miura in \cite{Miura2006}, while the instantaneous smoothing effect was later established in \cite{Dong2010, Biswas2014}, using different approaches. For the subcritical SQG equation, using a mild solution approach, analyticity was established in \cite{DongLi2008} in the critical Lebesgue spaces and in \cite{Biswas2012} in Besov spaces. It was observed in \cite{Miura2006} that the main difficulty when working in the critical space setting for the supercritically dissipative SQG equation is in obtaining a suitable continuity estimate for the bilinear term. Indeed, the classical Fujita-Kato mild solution approach cannot be carried out in this setting as the low degree of dissipation cannot, alone, accommodate the loss of derivatives from the nonlinearity. One can nevertheless establish such a continuity estimate by exploiting cancellation through the underlying commutator structure in the equation. Insofar as existence is concerned, one must therefore identify a suitable approximation procedure that respects this commutator structure. In the case of the supercritically dissipative  gSQG equation, the difficulty in obtaining the desired continuity estimate is compounded by the increasingly singular velocities that characterize the family. Indeed, a direct adaptation of the analysis in \cite{Miura2006} to the gSQG family breaks down in the range $\kap\leq\be-1$ without a more delicate treatment of the nonlinearity.

\begin{figure}[ht]\label{fig:map}
\psfrag{b}{$\beta$}
\psfrag{k}{$\kappa$}
\psfrag{1}{1}
\psfrag{2}{2}
\psfrag{mSQG}{\tiny mSQG}
\psfrag{CIW}{\tiny GR \cite{ConstantinIyerWu2008}} 
\psfrag{blt1}{\tiny GWP \cite{ChaeWu2012}}
\psfrag{endpt}{$\exists$ \& ! $H^{3-\kappa+\epsilon}$ Gevrey smoothing}
\psfrag{log}{$\psi=-(\ln(I-\De))^\mu\tht$}
\psfrag{brace}{$\big\}$}
\psfrag{sqg}{$\big\}$SQG}
\psfrag{sub}{\tiny GR \cite{Resnick1995, ConstantinWu1999} An. \cite{DongLi2008} Gev.\cite{Biswas2012}}
\psfrag{sup}{\tiny CR \cite{ConstantinWu2009, Coti-ZelatiVicol2016} ER \cite{Dabkowski2011}
Gev. \cite{Biswas2014,BiswasMartinezSilva2015}}
\psfrag{bk1}{\tiny $\beta=\kappa+1$}
\psfrag{bk2}{\tiny $\beta=\kappa/2+1$}
\psfrag{critical}{\tiny GR \cite{KiselevNazarovVolberg2007, CaffarelliVasseur2010, KiselevNazarov2009, ConstantinVicol2012, ConstantinTarfuleaVicol2015}}
\psfrag{LX}{\tiny ER \cite{LazarXue2019}}
\psfrag{result1}{$\exists$ \& ! $H^{\beta+1-\kappa}$}
\psfrag{CCW}{\tiny GR \cite{ChaeConstantinWu2012}}
\psfrag{diag}{\tiny GR \cite{MiaoXue2012}}
\psfrag{LWP}{\tiny LWP \cite{MiaoXue2011}}
\psfrag{GWP}{\tiny GWP  \cite{MiaoXue2011} }
\psfrag{Gevrey}{Gevrey smoothing}
\psfrag{gamma}{\tiny ($\gamma=0$)}
\psfrag{Euler}{Euler}
\psfrag{NSE}{NSE}
\psfrag{bgtr1}{\tiny loc $\exists$ \& ! $H^4$ \cite{ChaeConstantinCordobaGancedoWu2012}}
\psfrag{HKZ}{\tiny $H^{\beta+1+\epsilon}$  \cite{HuKukavicaZiane2015}}
 \centerline{\includegraphics[scale=.75]{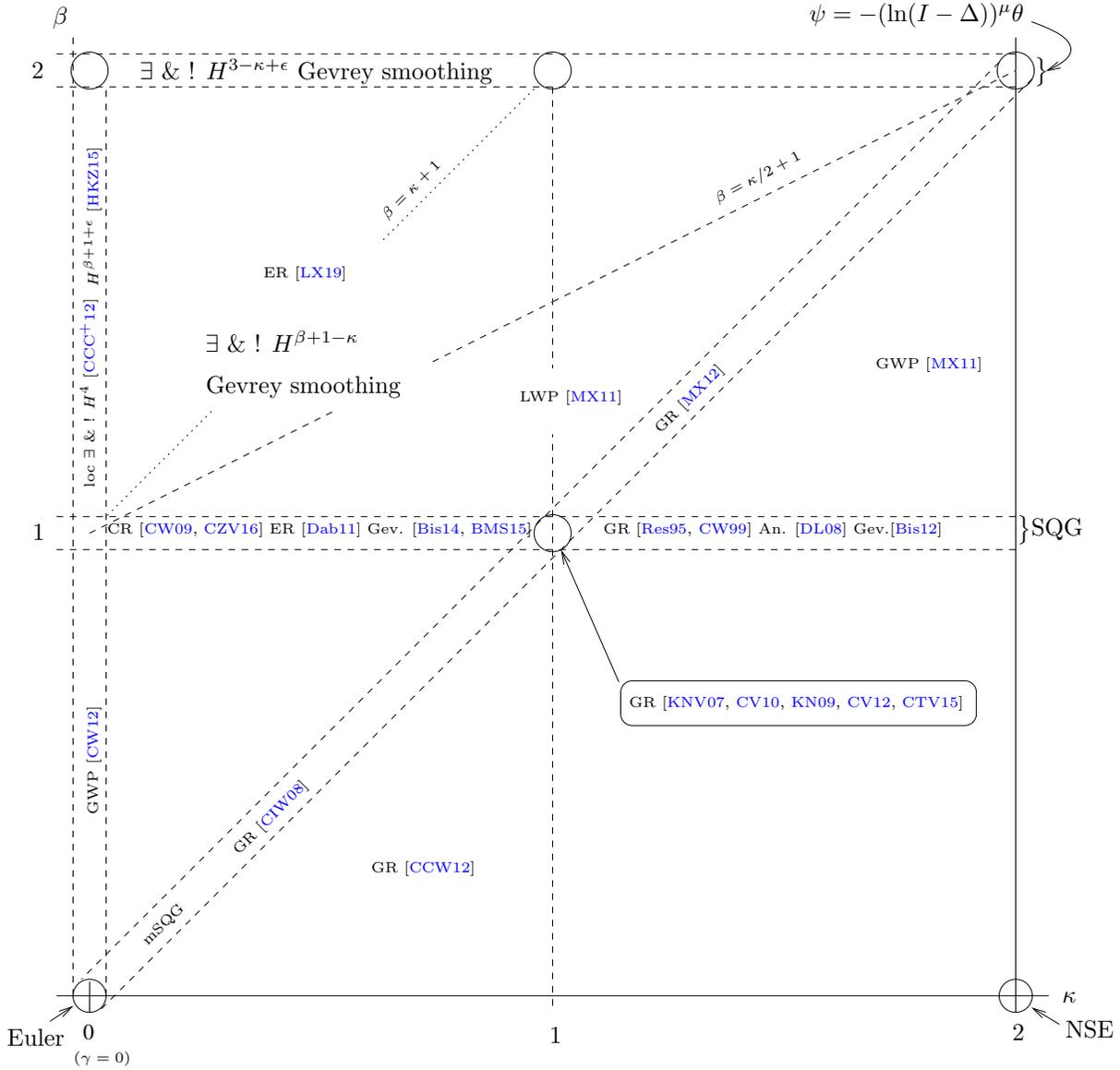} } \vskip .25 truein
\caption{LWP=local well-posedness, GWP=global well-posedness, GR=global regularity, CR=conditional regularity, ER=eventual regularity, An.=Analytic smoothing, Gev.= Gevrey smoothing}
\label{chaosfig}
\end{figure}

	When $\be\in(1,2)$, the criticality regimes are more nuanced; one identifies an additional ``subcritical" region, $\kap>\be-1$, ``critical" line, $\kap=\be-1$, and ``supercritical" regime $\kap<\be-1$. The supercritical regime represents the case where the equation becomes ``fully" quasilinear, in the particular sense that the coefficients of the nonlinearity depend on derivatives of the solution of an order that exceeds that of the linear term. Indeed, we see that in these additional critical and supercritical regimes, $\kap$ necessarily restricts to $\kap\in(0,1)$. Although we observe that one can take advantage of the additional commutator structure identified by Chae, Constantin, C\'ordoba, Gancedo, and Wu in \cite{ChaeConstantinCordobaGancedoWu2012} to successfully carry out the approaches in {\cite{Miura2006} and \cite{Biswas2014}}, a direct adaptation of the analysis there is limited to the subcritical regime $\kap>\be-1$. To overcome this limitation, one must make use of the more subtle commutator structure identified by Hu, Kukavica, and Ziane in \cite{HuKukavicaZiane2015}. In the critical space setting, however, the approximation procedure proposed by Miura in \cite{Miura2006} \textit{cannot} accommodate these additional commutators. In this paper, we propose a new approximation scheme in which one modifies the flux in such a way that ultimately preserves the underlying commutator structure (see \cref{sect:mod:flux}). Through this approximation, we are then able to obtain the desired continuity properties for the critical and supercritical regimes represented by $\kap\leq \be-1$. In either regime, we must carry out a more delicate analysis at the level of the Littlewood-Paley operators to accommodate the commutator estimates within the critical space setting, as well as extend these estimates appropriately to the Gevrey classes (see \cref{sect:commutator}). Analogous results for the $\be=2$ endpoint case are also established. As we remarked earlier, since the equation corresponding to \eqref{E:dissipative-log} does not possess a scaling symmetry, we instead work arbitrarily close to what would formally be the scaling critical space, that is, in $H^{3-\kap+\eps}$. With the appropriate commutator estimates and approximation scheme in hand, we prove our main results, \cref{thm:main:beta} and \cref{thm:main:log} in \cref{sect:proofs:main:beta} and \cref{sect:proofs:main:log}, respectively. As an immediate consequence of establishing Gevrey regularity, we obtain asymptotic decay of all derivatives with respect to the critical Sobolev topology (cf. \cref{cor:decay:beta}).
	
	Previous works on the existence theory for various regimes in the parameter space $(\kap,\be)\in(0,2]\times(0,2]$ of \eqref{E:dissipative-beta} have been carried out in \cite{ConstantinIyerWu2008,MiaoXue2012} for the diagonal case $\kap=\be$, i.e., the so-called ``modified SQG" equation, and in \cite{MiaoXue2011} for the regime $1<\be<2$, $\be<\kap<2$, where local well-posedness was studied, and in \cite{MiaoXue2012}, for the regime $1<\be<2$, $2\be-2<\kap<\be$, where global well-posedness was studied. A closely related generalization of \eqref{E:dissipative-beta} was also considered in \cite{LazarXue2019}, where global existence of weak solutions, global regularity for a slightly supercritical regularization, and eventual regularity of solutions were established. We point out that the smoothing effect in this paper is proved for an equation that belongs outside of the general class of systems treated in \cite{OliverTiti2001,BaeBiswas2015} as well as the general semilinear parabolic equation that includes the incompressible Navier-Stokes equations as a special case, treated in \cite{CheminGallagherZhang2020}. Indeed, this paper establishes a Gevrey regularity smoothing effect for a {quasilinear} parabolic equation of the form $\bdy_t\tht+\gam\Lam^\kap\tht=b(\Lam^\al\tht,{\nabla}\tht)$, where $\al,\kap\in(0,1)$, particularly, for one in which $\al>\kap$.  Some context for these results (displayed in larger font) in the $(\kap,\beta)$--plane is given in Figure 1.


	\section{Mathematical Preliminaries}\label{section:notation:preliminaries}
	Denote by $\mathscr{S}$ the space of Schwartz class functions on $\RR^2$ and by $\mathscr{S}'(\mathbb{R}^2)$ the space of tempered distributions. For $p\in[1,\infty]$, we let $L^p$ and $L^p_{loc}$ denote the spaces of Lebesgue integrable and locally Lebesgue integrable functions of order $p$ on $\RR^2$, respectively. The norm on $L^p$ is defined as
	    \begin{align}\notag
	       \Sob{f}{L^p}&:= \begin{cases}
	        \left(\int_{\RR^2}|f(x)|^pdx\right)^{1/p},& p\in[1,\infty),\\
	        \esssup_{x\in\RR^2}|f(x)|,& p=\infty.
	        \end{cases}
	    \end{align}    
	We recall that $L^p$ is a Banach space with this norm and that $L^p_{loc}\subset\mathscr{S}'$, for all $p\in [1,\infty]$. In the particular case $p=2$, $L^2$ can be endowed with the following inner product:
	    \begin{align*}
	        \lb f, g\rb:=\int_{\RR^2}f(x)\overline{g(x)}dx,
	    \end{align*}
	so that $L^2$ becomes a Hilbert space. Given $f\in\mathscr{S}'$, let $\hat{f}$ denote the Fourier transform of $f$; we will also use the notation, $\Ft$, to denote the Fourier transform. We recall that $\mathcal{F}$ is an isometry on $L^2$ and in particular that
	    \begin{align}\notag
	        \lb f,g\rb=\lb \hat{f},\hat{g}\rb.
	    \end{align}
	We define the fractional laplacian operator, $\Lam=(-\De)^{1/2}$, and its powers, $\Lam^\si$, $\si\in\RR$, by
	    \begin{align}\notag
	        \mathcal{F}(\Lam^\si f)(\xi)=|\xi|^\si\mathcal{F}(f).
	    \end{align} 
    For $\si\in\RR$, the homogeneous and the inhomogeneous Sobolev spaces on $\mathbb{R}^{2}$ are defined as
	    \begin{align}
	        &\Hdot^\si:=\left\{f\in \mathscr{S}':\hat{f}\in L^2_{loc}, \Sob{f}{\Hdot^\si}:=\Sob{\Lam^\si f}{L^2}<\infty\right\},\label{def:hom:Sob:norm}\\
	        &H^\si:=\left\{f\in \mathscr{S}':\hat{f}\in L^2_{loc}, \Sob{f}{H^\si}:=\Sob{(I-\De)^{\si/2} f}{L^2}<\infty\right\}.\label{def:inhom:Sob:norm}
	    \end{align}

	The spaces \eqref{def:hom:Sob:norm}, \eqref{def:inhom:Sob:norm} can be endowed with inner products given by
	    \begin{align}
	        \lb f,g\rb_{\Hdot^\si}&:=\lb \Lam^\si f,\Lam^\si g\rb,\notag\\
	        \lb f,g\rb_{H^\si}&:=\lb (I-\De)^\si f,(I-\De)^\si g\rb\notag.
	    \end{align}
	With this in hand, $H^\si$ is a Hilbert space for all $\si\in\RR$, whereas, in dimension-two,  $\Hdot^\si$ is a Hilbert space if and only if $\si<1$.
	The inhomogeneous spaces are nested $H^{\si'}\subset H^\si$, whenever $\si'>\si$, and moreover the embedding is compact over any compact set $K\subset\RR^2$. On the other hand, the homogeneous spaces are, in general, only related by the following interpolation inequality: For $\si_1\leq\si\leq \si_2$ 
	\begin{align}\label{E:interpolation}
	    \nrm{f}_{\Hdot^{\si}}\le C\Sob{f}{\Hdot^{\si_1}}^{\frac{\si_2-\si}{\si_2-\si_1}}\Sob{f}{\Hdot^{\si_2}}^{\frac{\si-\si_1}{\si_2-\si_1}},
	\end{align}
    for some constant depending only on $\si,\si_1,\si_2$. A related interpolation inequality that will be also be useful is stated as follows: for $\si\in\RR$ and $\si_1>-1>-\si_2$, there exists a constant, $C$, depending on $\si,\si_1,\si_2$ such that
        \begin{align}\label{E:interpolation:L1}
            \nrm{|\cdotp|^\si\hat{f}}_{L^1}\leq C\Sob{f}{\dot{H}^{\si+\si_2}}^{\frac{\si_1+1}{\si_1+\si_2}}\Sob{f}{\dot{H}^{\si-\si_1}}^{\frac{\si_2-1}{\si_1+\si_2}}.
        \end{align}
	
	\subsection{Littlewood-Paley Decomposition}
	We will make use of the characterization of Sobolev spaces in terms of the Littlewood-Paley decomposition. We review this decomposition now and refer the reader  \cite{BahouriCheminDanchinBook2011, CheminBook1995} for additional details.

	First, let us define
	\begin{align*}
	\mathscr{Q}:=\left\{f\in \mathscr{S}: \int_{\mathbb{R}^2}f(x)x^{\tau}\, dx=0, \quad \abs{\tau}=0,1,2,\cdots \right\}.
	\end{align*}
	We denote by $\mathscr{Q}'$ the topological dual space of $\mathscr{Q}$. Note that $\mathscr{Q}'$ can be identified with the space of tempered distributions modulo polynomials, that is, as
	\begin{align*}
	\mathscr{Q}'\cong\mathscr{S}'/\mathscr{P}.
	\end{align*}
    {where the vector space of polynomials on $\mathbb{R}^2$ is denoted by  $\mathscr{P}$}.
	
	Let us denote by ${\Bcal}(r)$, the open ball of radius $r$ centered at the origin and by ${\Acal}(r_{1},r_{2})$ , the open annulus with inner and outer radii $r_{1}$ and $r_{2}$, and centered at the origin. There exist two non-negative radial functions $\chi,\phi\in\mathscr{S}$ with $\supp\chi\subset{\Bcal}(1)$ and $\supp\phi\subset{\Acal}(2^{-1},2)$ such that for $\chi_j(\xi):=\chi(2^{-j}\xi)$ and $\phi_j(\xi):=\phi(2^{-j}\xi)$, one has $\sum_{j\in\ZZ}\phi_j(\xi)=1$, whenever $\xi\in\RR^2\setminus\{\mathbf{0}\}$, $\chi+\sum_{j\geq0}\phi_j\equiv 1$, and one has the following almost-orthogonality conditions:
	    \begin{align}\notag
	        &\supp\phi_i\cap\supp\phi_j=\varnothing,\quad |i-j|\geq2,\quad\text{and}\quad\supp\phi_i\cap\supp\chi =\varnothing,\quad i\geq1.
	    \end{align}
	We will make use the shorthand
        \begin{align}\notag
            {\Acal}_{j}= {\Acal}(2^{j-1},2^{j+1}),\quad{\Acal}_{\ell,k}= {\Acal}(2^{\ell},2^{k}),\quad  {\Bcal}_j={\Bcal}(2^j),
        \end{align}
   so that, in particular, ${\Acal}_j={\Acal}_{j-1,j+1}$. With this notation, observe that
	    \begin{align}\label{eq:rewrite:supp}
	        \supp\phi_j\subset{\Acal}_j,\quad \supp\chi_j\subset{\Bcal}_j.
	    \end{align}
	We denote the (homogeneous) Littlewood-Paley dyadic blocks by ${\lpj}$ and $S_{j}$, which are both defined in terms of its Fourier transform by
	\begin{align}\notag
	\mathcal{F}({\lpj}f)=\phi_{j}\mathcal{F}(f), \quad \mathcal{F}(S_{j}f)=\chi_{j}\mathcal{F}(f). 
	\end{align}
    The following localization properties of $\{\lpj\}_{j\in\ZZ}$ and $\{S_j\}_{j\in\ZZ}$ are a direct consequence of \eqref{eq:rewrite:supp}:
	\begin{align*}
	&\mathcal{F}({\lpj}f)|_{{\Acal}_j^c}=0,\quad
	\mathcal{F}(S_{j}f)|_{{\Bcal}_j^{c}}=0,
	\end{align*}
    Observe that by definition
	\begin{align*}
	S_{j}=S_i+\sum_{ i\leq k \le j-1}{\lpk},\quad i<j,
	\end{align*}
    and, in particular that
    \begin{align*}
         f&=S_if+\sum_{j\geq i}\lpj f,\quad i\in\ZZ,\quad f\in\mathscr{S}'
    \end{align*}
	On the other hand, when restricted to $\mathscr{Q}'$ one has
	    \begin{align*}
	        S_j=\sum_{k\leq j-1}\lpk,
	    \end{align*}
	and, in particular that
	    \begin{align*}
	        f&=\sum_{j\in\ZZ}\lpj f,\quad f\in\mathscr{Q}'.
	    \end{align*}
	In light of the Littlewood-Paley decomposition, one can see that the homogeneous Sobolev spaces, $\Hdot^{\si}$, can be identified with the homogeneous Besov spaces, $\dot{B}^{\si}_{2,2}$, whose norm is given by
	    \begin{align*}
	        \nrm{f}_{\dot{B}^{\si}_{2,2}}:=\left(\sum_{j\in \mathbb{Z}}\left(2^{j\si}\nrm{\lpj f}_{L^2}\right)^{2}\right)^{\frac{1}{2}}.
	    \end{align*}
	In particular, we have
	    \begin{align*}
	        C^{-1}\Sob{f}{\dot{B}^\sigma_{2,2}}\leq \Sob{f}{\dot{H}^\si}\leq C\Sob{f}{\dot{B}^\sigma_{2,2}},
	    \end{align*}
	for some constant $C$, independent of $\si$. The relation between the Littlewood-Paley blocks and the fractional laplacian is conveniently captured by the following Bernstein-type inequalities, which we will make copious use of throughout the paper.
	\begin{Lem}[Bernstein inequalities]\label{T:Bernstein}
		Let $\si\in\RR$ and $1\le p \le q\le \infty$. Then
		\begin{align*}
		c'2^{\si j}\nrm{{\lpj}f}_{L^q}\le \nrm{\Lam^{\si}{\lpj}f}_{L^q}\le c2^{\si j+2j(\frac{1}{p}-\frac{1}{q})}\nrm{{\lpj}f}_{L^p},
		\end{align*}
		where $c',c$ are constants that depends on $p,q$ and $\si$.
	\end{Lem}
	
	\subsection{Gevrey classes} We introduce the Gevrey classes in this section. These spaces identify a scale of subspaces between the analytic class of functions, $C^\om$, and the class of smooth functions, $C^\infty$. We will consider a Littlewood-Paley characterization of the Gevrey classes inspired by the spectral characterization of the Gevrey norm introduced by Foias and Temam in their seminal paper \cite{FoiasTemam1989}. We remark that the Littlewood-Paley characterization was also adopted in \cite{Biswas2014}; for an extension to $L^p$-based Besov spaces, see \cite{BiswasMartinezSilva2015}. 
	
	Let $\al\in(0,1]$ and $\lam>0$. Then we define the Gevrey operator, ${G}_\al^\lam$, of order $\al$ and radius $\lam$ by
    \begin{align*}
        \mathcal{F}({G}_\al^\lam f)(\xi)=e^{\lam|\xi|^{\al}}\hat{f}(\xi).
    \end{align*}
    We will also make use of the notation   
        \begin{align}\label{def:alt:Gev:op}
            {G}_\al^\lam=e^{\lam \Lambda^\al}.
        \end{align}
    We define the $\Hdot^{\si}$--based Gevrey norm by
	\begin{align*}
	\Sob{f}{\dot{G}^{\lam}_{{\al},\si}}:=\Sob{{G}_\al^\lam f}{\Hdot^{\si}}.
	\end{align*}
	Then the homogeneous $(\al,\lam,\si)$--Gevrey classes are defined as
	    \begin{align}\label{def:Gev:class}
	        \dot{G}^\lam_{\al,\si}:=\{f\in \dot{H}^\si:\Sob{f}{\dot{G}^{\lam}_{\al,\si}}<\infty\}.
	    \end{align}
	Finiteness of $\nrm{f}_{\dot{G}^{\lam}_{\al,\si}}$, for some $\si\in\RR$ and $\al,\lam>0$, automatically yields estimates on higher-order derivatives. Indeed, one has
	\begin{align}\label{E:decay-estimate}
	\Sob{\bdy^\be f}{\Hdot^{\si}}\le \left(\frac{\be!}{(\lam\al)^{|\be|}
	}\right)^{1/\al}\nrm{f}_{G^{\lam}_{\al,\si}},
	\end{align}
    for all multi-indices, $\be\in\NN_0^{2}$, where $\NN_0:=\NN\cup\{0\}$. Thus, with the Sobolev embedding theorem, \eqref{E:decay-estimate} implies uniform bounds on all orders of derivatives. In the particular case $\al=1$, if $f\in L^2$ satisfies (\ref{E:decay-estimate}) for all $\be\in\NN_0^2$, for some $\lam>0$, then $f$ is real analytic at each $x\in\RR^2$ with analyticity radius $\lam$. On the other hand if (\ref{E:decay-estimate}) is satisfied with $\al<1$, we say that $f$ belongs to a subanalytic Gevrey class, which is a subclass of smooth functions. For additional properties of Gevrey classes and applications to a wide-class of equations, the reader is referred to \cite{LevermoreOliver1997, OliverTiti2000, OliverTiti2001, BaeBiswas2015, MartinezZhao2017}.
	
	The main scenario of interest in this paper is when $f$ is a time-dependent function, globally defined in time. In this case, if $f$ satisfies $f(t)\in\dot{G}^{\lam(t)}_{\al,\si}$, for all $t>0$, for some monotonically increasing function $\lam=\lam(t)$, then (\ref{E:decay-estimate}) yields temporal decay of all higher-order derivatives of $f$; this  is one of the main motivations for using the Gevrey norm.
	
	We will distinguish between the $(\al,\lam)$--Gevrey operators, ${G}_\al^\lam$, and the related fractional heat propagator, $\{{H}_\kap^\gam(t)\}_{t\geq0}$, which, for $\kap\in(0,2]$, ${\gam}>0$, we define as
    \begin{align*}
        {H}_\kap^\gam(t)=\exp\left(-t{\gam}\Lam^\kap\right).
    \end{align*}
In particular, given $\tht_0\in L^2$, one has that $\tht(t;\tht_0):={H}_\kap(t)\tht_0$ satisfies
    \begin{align*}
        \bdy_t\tht+{\gam}\Lam^\kap\tht=0,\quad \tht(0;\tht_0)=\tht_0.
    \end{align*}
Observe that the $(\kap,\lam)$--Gevrey operators can be rewritten as fractional heat propagators appropriately re-scaled:
    \begin{align*}
       {G}_\kap^\lam={H}_\kap^\gam\left(-\frac{\lam}{\gam}\right)={H}^1_\kap(\lam).
    \end{align*}
One may thus alternatively view the $(\al,\lam,\si)$--Gevrey classes as the space of functions for which
    \begin{align}\notag
        \Sob{{H}_\kap^1(-\lam)f}{\dot{H}^\si}<\infty,
    \end{align}
that is, for which the inverse of the heat propagator belongs to the $\dot{H}^\si$. It can be viewed as a parabolic analog of the so-called $X^{s,b}$--spaces in the dispersive PDE literature, which are defined in terms of inverses of dispersive operators such as the Airy kernel or Schr\"odinger propagator (cf. \cite{Bourgain1993a, Bourgain1993b, TaoBook2006}).

	Lastly, let us recall the following interpolation-type inequality for Gevrey operators that was originally proved in \cite{OliverTiti2000}, but stated here in a slightly more generalized form.
	\begin{Lem}\label{lem:interpolate:gevrey}
		{Let $\al,\lam>0$ and $s_1\leq s_2$. Suppose $f\in \Hdot^{s_1}$ such that $G_\al^{\lam}f\in \Hdot^{s_2}$. Then}
		\begin{align*}
		{\Sob{{G}_\al^{\lam}f}{\dot{H}^{s_1}}^2\leq e\Sob{{G}_\al^{(1-\rho)\lam}f}{\dot{H}^{s_1}}^2+\left({2\rho\lam}\right)^{\frac{2(s_2-s_1)}{\al}}\Sob{{G}_\al^{\lam}f}{\dot{H}^{s_2}}^2},
		\end{align*}
	{for all $\rho\in(0,1]$}.
	\end{Lem}
	\begin{proof}
	 Let $R>0$, to be chosen later, and fix $\rho\in(0,1]$. By Plancherel's theorem, we have
	    \begin{align}
	        \Sob{{G}^{\lam}_\al f}{\dot{H}^{s_1}}^2=\left(\int_{|\xi|\leq R}+\int_{|\xi|>R}\right) e^{2\lam|\xi|^{\al}}|\xi|^{2s_1}|\hat{f}(\xi)|^2d\xi=I+II.\notag
	    \end{align}
	We estimate $I$ as
	    \begin{align}
	        |I|\leq e^{2\rho\lam R^\al}\Sob{{G}^{(1-\rho)\lam}_\al f}{\dot{H}^{s_1}}^2.\notag
	    \end{align}
	We estimate $II$ as
	    \begin{align}
	        |II|&\leq R^{-2(s_2-s_1)}\Sob{{G}_\al^{\lam}f}{\dot{H}^{s_2}}^2\notag.\notag
	    \end{align}
	Now choose $R=(2\rho\lam)^{-1/\al}$, so that
	    \begin{align}\notag
	        |I|+|II|\leq e\Sob{{G}^{(1-\rho)\lam}_\al f}{\dot{H}^{s_1}}^2+\left({2\rho\lam}\right)^{\frac{2(s_2-s_1)}{\al}}\Sob{{G}_\al^{\lam}f}{\dot{H}^{s_2}}^2,
	    \end{align}
	 as desired.
	\end{proof}

\begin{Rmk}
For the rest of the paper, we adopt the convention that $C$ denotes a positive constant whose magnitude may change from line-to-line. Dependencies on other parameters will typically be suppressed in performing estimates, but may be specified in statements of lemmas, propositions, or theorems when they are relevant or for clarity.
\end{Rmk}

\section{Statements of Main Results}
    Our main result for \eqref{E:dissipative-beta} when $\be\in(1,2)$ in the regime of supercritical dissipation is stated in the following theorem. 
	\begin{Thm}\label{thm:main:beta}
		Let $\be\in(1,2)$, $\kap\in(0,1)$, and $\si_c:=1+\be-\kap$. For each $\tht_0\in H^{\si_c}$, there exists $T>0$ and a unique solution, $\tht$, of \eqref{E:dissipative-beta} such that 
		    \begin{align*}
		        \tht\in C([0,T);H^{\si_c})\cap L^2(0,T;\Hdot^{\si_c+\kap/2}).
		    \end{align*}
		Moreover, for any $0<\al<\kap$ and $\de>0$ sufficiently small, there exists an increasing function $\lam:[0,\infty)\goesto[0,\infty)$ with $\lam(0)=0$ such that
		\begin{align}\label{E:ET}
		\nrm{\tht(t)}_{\Gdot^{\lam(t)}_{\al, \si_c+{\de}}}\le C\frac{\nrm{\tht_{0}}_{\Hdot^{\si_c}}}{(\gam t)^{\de/\kap}},
		\end{align}
		for all $0<t<T$, for some constant $C>0$, independent of $T$. Lastly, if $\nrm{\tht_0}_{\Hdot^{\si_c}}$ is small enough, then $T=\infty$ is allowed.
	\end{Thm}
	
	 We remark that in the assertion of local existence above, the standard critical space phenomenon where $T$ depends on $\tht_{0}$ in a manner beyond exclusively through its critical norm $\Sob{\tht_0}{\Hdot^{1+\be-\kap}}$ is observed. On the other hand, in the small data, global existence setting, we observe that \eqref{E:ET} along with the Sobolev embedding theorem implies temporal decay of all higher-order derivatives of the corresponding solution in the critical norm. As we will see in the proof of \cref{thm:main:beta}, $\lam$ can be chosen as
	    \begin{align}\label{def:lam:initial}
	        \lam(t)=\veps \gam^{\al/\kap}t^{\al/\kap},
	    \end{align}
	 for $\veps>0$ chosen sufficiently small. With this in mind, we have the following immediate corollary.
	
	\begin{Cor}\label{cor:decay:beta}
	Let $\be\in(1,2)$, $\kap\in(0,1)$, and $\si_c:=1+\be-\kap$. For $\de>0$ and $\nrm{\tht_0}_{\Hdot^{\si_c}}$ sufficiently small as in Theorem \ref{thm:main:beta}, we have for each integer $k>0$
	    \begin{align*}
	        \Sob{D^k\tht(t)}{\dot{H}^{\si_c+\de}}=C_k\frac{\Sob{\tht_0}{\dot{H}^{\si_c}}}{(\gam t)^{\frac{k+\de}{\kap}}},
	    \end{align*}
	for all $t>0$, where $\lam(t)$ is given by \eqref{def:lam:initial}, and $C_k$ depends on $k$.
	\end{Cor}

	We lastly observe that \eqref{E:ET} is \textit{nearly} optimal in the sense that $\tht(t)\in \dot{G}_{\al,\si_c+\de}^{\lam(t)}$, for all $\al\in(0,\kap)$, where the optimal result is represented by the endpoint $\al=\kap$. This is consistent with the results obtained in \cite{Biswas2014} for the supercritical SQG equation in the critical Sobolev spaces. The analysis in \cite{Biswas2014} was subsequently extended to the $L^p$--based critical Besov spaces in \cite{BiswasMartinezSilva2015}. In light of these results, it would be interesting to extend \cref{thm:main:beta} to the Besov space setting as well. 
	
	Our second main result establishes the analogous statement for the modified $\be=2$ endpoint case defined by (\ref{E:dissipative-log}).

	\begin{Thm}\label{thm:main:log}
		Let $\kap\in(0,1)$ and suppose $\tht_{0}\in H^{\si}$, where $\si>3-\kap$. There exists $T=T(\nrm{\tht_0}_{H^{\si}})>0$ and a unique solution $\tht(x,t)$ to \eqref{E:dissipative-beta} with streamfunction given by (\ref{E:dissipative-log}) such that 
		    \begin{align*}
		        \tht\in C([0,T);H^{\si})\cap L^2(0,T;\Hdot^{\si+\kap/2}).
		    \end{align*}
		Moreover, for any $0<\al<\kap$ and $\lambda>0$, we have 
		\begin{align*}
	\sup_{0\le t\leq T}\nrm{\tht(t)}_{\Gdot^{\lam t}_{\al,\si}}\le C(1+\nrm{\tht_{0}}_{H^{\si}}),
		\end{align*}
	for some constant $C>0$ independent of $\tht_{0}$.
	\end{Thm}
	
	\begin{Rmk}\label{rmk:mSQG}
	The global regularity problem for the regime $0<\kap<\be$, (cf. \cref{fig:map}) remains an outstanding open problem. This issue was resolved for the case $\kap=\be$, for all $\be\in(0,1)$ i.e., modified SQG, in \cite{ConstantinIyerWu2008, MiaoXue2012}. We point out that our analysis can as well be extended to the case $\be\in(0,1)$ without any difficulty. In particular, \cref{thm:main:beta} additionally improves on the work \cite{MiaoXue2011}, where local well-posedness was established in $H^\si$, for $\si>2$, provided that $\be<\kap/2+1$.
	\end{Rmk}

	\section{Commutator estimates}\label{sect:commutator}
	In \eqref{E:dissipative-beta}, the expression for $\tht(x,t)$ in terms of $u(x,t)$ is given by a singular integral; {the strength of the singularity of the operator is quantified by the parameter $\beta\in(1,2)$. The parameter, $\be$, belonging to this range precludes one from obtaining a suitable continuity-type estimate on the bilinear term.} To overcome this difficulty, {we exploit observations made in \cite{ChaeConstantinCordobaGancedoWu2012} and \cite{HuKukavicaZiane2015} in which additional commutators are identified that allow one to allocate derivatives more effectively}. We will require the following lemma, the proof of which is provided in \cref{sect:app}. It is essentially a classical product estimate, but we provide it in a frequency-restricted dualized form, as this is the natural form in which it appears in the apriori analysis below. It will be frequently deployed in proving the required commutator estimates.
	\begin{Lem}\label{lem:commutator1}
		For $\si \in (-1,1)$ and $f,g,h \in \mathscr{S}(\RR^2)$, define
		\begin{align*}
		\mathcal{L}_{\si}(f,g,h)&:=\iint\Ax^{\si}\hat{f}(\xi-\eta){ \hat{g}(\eta)\overline{\hat{h}(\xi)}} d\eta d\xi.
		\end{align*}
		Suppose that $\supp \hat{h}\subset\Acal_j$, for some $j\in\mathbb{Z}$. Then for each $\si\in(-1,1)$ and $\epsilon \in (0,2)$ such that $\sigma>\epsilon-1$, there exists a constant $C>0$, depending only on $\si,\epsilon$, and a sequence $\{c_j\}\in\ell^2(\mathbb{Z})$ with $\Sob{\{c_j\}}{\ell^2}\leq1$ such that
		\begin{align}\notag
		    	|\mathcal{L}_{\si}(f,g,h)|\le  Cc_j2^{\epsilon j}\min\left\{\nrm{f}_{\Hdot^{1-\epsilon}}\nrm{g}_{\Hdot^{\si}},\nrm{g}_{\Hdot^{1-\epsilon}}\nrm{f}_{\Hdot^{\si}}\right\}\nrm{h}_{L^{2}}.
		\end{align}
	\end{Lem}

We observe that Lemma \ref{lem:commutator1} immediately imply the following corollary, which will be useful to have in-hand for proving the commutator estimates below.

\begin{Cor}\label{cor:commutator1}
	For $\si \in (-1,1)$ and $f,g,h \in \mathscr{S}(\RR^2)$, define
		\begin{align*}
		\mathcal{L}_{\si}'(f,g,h)&:=\iint\left(\Axe^{\si}+\abs{\eta}^{\si}\right)\hat{f}(\xi-\eta){ \hat{g}(\eta)\overline{\hat{h}(\xi)}} d\eta d\xi.
		\end{align*}
		Suppose that $\supp \hat{h}\subset\Acal_j$, for some $j\in\mathbb{Z}$. Then for each $\si\in[0,1)$ and $\epsilon \in(0,1)$, there exists a constant $C>0$, depending only on $\si,\epsilon$, and a sequence $\{c_j\}\in\ell^2(\mathbb{Z})$ with $\Sob{\{c_j\}}{\ell^2}\leq1$ such that
			\begin{align}\notag
	 |\mathcal{L}_{\si}'(f,g,h)|\le Cc_j2^{\epsilon j}\min\left\{\nrm{f}_{\Hdot^{1-\epsilon}}\nrm{g}_{\Hdot^{\si}},\nrm{g}_{\Hdot^{1-\epsilon}}\nrm{f}_{\Hdot^{\si}}\right\}\nrm{h}_{L^{2}}.
		\end{align}
\end{Cor}

\subsection{Commutator estimates in Sobolev classes}\label{sect:commutator:Sob}
We will require two commutator estimates, stated below in \cref{lem:commutator2a} and \cref{lem:commutator2}, for the trilinear interactions that appear naturally in the energy arguments carried out in \cref{sect:mod:flux}, \cref{sect:proofs:main:beta}, \cref{sect:proofs:main:log} below. \cref{lem:commutator2a} essentially arises as an intermediate step in establishing a classical version of the commutator estimate proved by H. Miura in \cite{Miura2006}. We include a proof in \cref{sect:app}. The trilinear form of \cref{lem:commutator2a} is crucial as some of the commutators we appeal to can ultimately only be formed between a triad interaction of functions. On the other hand, due to the expression of the velocity in terms of fractional laplacians and partial derivatives, we will require a variation of the commutator estimate appearing in \cite{Miura2006} to accommodate these types of operators; this is established in \cref{lem:commutator2}.

In what follows, we will denote the commutator of two operators, $S$ and $T$, by $[S,T]$, where
    \begin{align*}
        [S,T]:=ST-TS.
    \end{align*}
We adopt the convention that $[T,f]=[T,fI]$, where $f$ is a scalar function, and $I$ denotes the identity operator.

\begin{Lem}\label{lem:commutator2a}
Let ${\rho_1\in(0,2)}$ and ${\rho_2} \in (-1,1)$ such that $\rho_2>\rho_1-1$. Let $f,g,h \in L^2$ with $ \supp \hat{h}\subset{\Acal}_j$. Then there exists a sequence $\{c_j\}\in\ell^2(\ZZ)$ such that $\Sob{\{c_j\}}{\ell^2}\leq1$ and
    \begin{align}
       \abs{\lb [\triangle_{j},g]f,h\rb} \leq Cc_{j}2^{(\rho_1-\rho_2-1)j}\min\left\{\Sob{f}{\Hdot^{1-\rho_1}}\Sob{g}{\dot{H}^{1+\rho_2}},\Sob{f}{\Hdot^{\rho_2}}\Sob{g}{\dot{H}^{2-\rho_1}}\right\}\Sob{h}{L^2}, \notag
    \end{align}
for some constant $C$ depending only on $\rho_1,\rho_2$.
\end{Lem}

\begin{Rmk}
Note that the upper bound $Cc_{j}2^{(\rho_1-\rho_2-1)j}\Sob{ f}{\Hdot^{1-\rho_1}}\Sob{h}{L^2}\Sob{g}{\dot{H}^{1+\rho_2}}$ can also be established above by directly applying Proposition 2 in \cite{Miura2006}. Thus, the bound we provide allows for additional flexibility in the allocation of derivatives.
\end{Rmk}

\begin{Lem}\label{lem:commutator2}
Let $\be\in(1,2)$,  ${\rho_1\in (0,2)}$, ${\rho_2} \in (-1,1)$ be such that $\rho_2>\rho_{1}-1$. Let $f,g \in L^2$. Then there exists a constant $C>0$, depending only on $\be,{\rho_1,\rho_2}$, such that
    \begin{align}
        \Sob{[\Lam^{\be-2}\bdy_\ell,g]f}{\Hdot^{\rho_2-\rho_1}}\leq C\Sob{g}{\dot{H}^{\be-\rho_1}}\Sob{ f}{\Hdot^{\rho_2}},\quad \ell=1,2.\notag
    \end{align}
\end{Lem}

\begin{proof}
Let $h \in L^{2}$. It will be convenient to define the following functional:
    \begin{align}\label{def:L:functional:1}
        \mathcal{L}_{\be,\ell}(f,g,h):=\iint m_{\be,\ell}(\xi,\eta)\hat{f}(\xi-\eta)\hat{g}(\eta)\overline{\hat{h}(\xi)}d\eta d\xi,
    \end{align}
where
    \begin{align*}
        m_{\be,\ell}(\xi,\eta):=|\xi|^{\be-2}\xi_\ell-|\xi-\eta|^{\be-2}(\xi-\eta)_\ell.
    \end{align*}
Indeed, by the Plancherel theorem, observe that
    \[
        \lb\Lam^{\rho_2-\rho_1}\lpj[\Lam^{\be-2}\bdy_\ell,g]f ,h\rb=\mathcal{L}_{\beta,\ell}(f,g, \Lam^{\rho_2-\rho_1}\lpj h).
    \]

Let
    \begin{align}\label{def:A}
    \mathbf{A}(\tau,\xi,\eta):=\tau\xi+(1-\tau)(\xe).
    \end{align}
For convenience, we will suppress the dependence of $\mathbf{A}$ on $\xi,\eta$. Observe that
    \begin{align}\label{E:meanvalue1}
    \begin{split}
        |m_{\be,\ell}(\xi,\eta)|&=\abs{\int_{0}^{1}\frac{d}{{d\tau}}\left(\abs{\mathbf{A}(\tau)}^{\be-2}\mathbf{A}(\tau)_{\ell}\right)d\tau}\\
        &=\abs{\int_{0}^{1}\left(\abs{\mathbf{A}(\tau)}^{\be-2}\eta_{\ell}+(\be-2)\abs{\mathbf{A}(\tau)}^{\be-4}(\mathbf{A}(\tau)\cdot \eta)\mathbf{A}(\tau)_{\ell}\right)d\tau}\\
    &\le C\abs{\eta}\int_{0}^{1}\abs{\mathbf{A}(\tau)}^{\be-2}d\tau,
    \end{split}
    \end{align}
where the fact $\be\in(1,2)$ is invoked to obtain the last inequality. 
Let $\varphi:=\frac{\xe}{\Ae}$ and $\vartheta:=\frac{\eta}{\Ae}$.  We observe that for fixed $\xi$ and $\eta$, we have
\begin{align*}
\int_{0}^{1}\abs{\eta}\abs{ \mathbf{A}(\tau)}^{\be-2}\,d\tau=\Ae^{\be-1}\int_{0}^{1}\frac{1}{\abs{\varphi+\tau \vartheta}^{2-\be}}\,d\tau.
\end{align*}
We have the following inequality:
\begin{align*}
\abs{\varphi+\tau \vartheta}^{2}=\abs{\varphi}^{2}+\tau^{2}+2\tau\varphi\cdot\vartheta\ge\abs{\varphi}^{2}+\tau^{2}-2\tau\abs{\varphi}=\abs{\abs{\varphi}-\tau}^{2},
\end{align*}
giving us, since $\be<2$,
\begin{align*}
\int_{0}^{1}\frac{1}{\abs{\varphi+\tau \vartheta}^{2-\be}}\,d\tau \le\int_{0}^{1}\frac{1}{\abs{\abs{\varphi}-\tau}^{2-\be}}d\tau.
\end{align*}
If $\abs{\varphi}\le 1$, we have for $1<\be<2$
\begin{align*}
\begin{split}
\int_{0}^{1}\frac{1}{\abs{\abs{\varphi}-\tau}^{2-\be}}\,d\tau&=\int_{0}^{\abs{\varphi}}\frac{1}{(\abs{\varphi}-\tau)^{2-\be}}\,d\tau+\int_{\abs{\varphi}}^{1}\frac{1}{(\tau-\abs{\varphi})^{2-\be}}\,d\tau\\
&\le C\left(\abs{\varphi}^{\be-1}+(1-\abs{\varphi})^{\be-1}\right)\le C.
\end{split}
\end{align*}
If $\abs{\varphi}>1$, then by the mean value theorem, since $1<\be<2$,
\begin{align*}
\int_{0}^{1}\frac{1}{\abs{\abs{\varphi}-\tau}^{2-\be}}\,d\tau&=C(-(\abs{\varphi}-1)^{\be-1}+\abs{\varphi}^{\be-1})
\le C.
\end{align*}
Using the above inequalities in (\ref{def:L:functional:1}), we get
\begin{align*}
 \abs{\mathcal{L}_{\be,\ell}(f,g,\lpj\Lam^{\rho_2-\rho_1} h)} \le &C\int_{\RR^{2}}\int_{\RR^{2}}\abs{\xi}^{\rho_2-\rho_1}|\widehat{\Lam^{\be-1} g}(\eta)||\hat{f}(\xi-\eta)||\widehat{\lpj h}(\xi)|d\eta d\xi.
\end{align*}
Application of \cref{lem:commutator1} with $\si=\rho_2$, $\epsilon=\rho_1$ gives us
\begin{align}\notag
\abs{\lb  \Lam^{\rho_2-\rho_1}\lpj[\Lam^{\be-2}\bdy_\ell,g]f ,h\rb}=\abs{\mathcal{L}_{\be,\ell}(f,g,\lpj\Lam^{\rho_2-\rho_1} h)} \le Cc_{j}\Sob{g}{\dot{H}^{\be-\rho_1}}\Sob{ f}{\Hdot^{\rho_2}}\Sob{ h}{L^2},
\end{align}
where $c_{j}$ is independent of $h$ and $\sum_{j}c_{j}^{2}\le 1$. Since $L^{2}= \dot{B}^{0}_{2,2}$, this completes the proof.  
\end{proof}

\subsection{Commutator estimates in Gevrey classes}\label{sect:commutator:Gev}

Next we prove a commutator estimate for operators which can be expressed as a product of Fourier multiplier operators given by ${G}^{\lam}_{\al}, \Lam^{\si}, \bdy_{\ell}, \triangle_j$, where ${G}^{\lam}_{\al}$ is defined in \eqref{def:alt:Gev:op}. In what follows, it will be convenient to introduce the operator, ${E}_\al^\lam$, given by
    \begin{align}\label{def:gev:avg}
    {\Ft({E}^\lam_\al f)(\xi):=\left(\int_0^1e^{\lam\tau^\al|\xi|^\al}d \tau\right)\hat{f}(\xi)}.
    \end{align}
We will also let ${D}$ denote  
    \begin{align*}
        D=\Lam\ \text{or}\ \bdy_\ell,\quad \text{for}\ \ell=1,2.
    \end{align*}

\begin{Lem}\label{lem:commutator3}
Let $\lam\geq0$, $\si \in [0,1)$, $\al\in(0,1]$, $\ze\in[0,1)$, $\nu\in(0,1)$, and $\rho \in \mathbb{R}$. Suppose $f,g,h\in L^2$ such that $\supp\hat{h}\subset{\Acal}_j$.  Then there exists a sequence $\{c_j\}\in\ell^2(\ZZ)$ such that $\Sob{\{c_j\}}{\ell^2}\leq1$ and
    \begin{align}
        |\lb [{G}_\al^{\lam}\Lam^{\si+\rho}{D}\lpj,g]f,h  \rb|
        \leq & Cc_j2^{{\nu} j}\min\left\{\Sob{ f}{\Gdot^\lam_{\al,1-\nu}}\Sob{  g}{\Gdot^\lam_{\al,\si+1}},\Sob{ g}{\Gdot^\lam_{\al,2-\nu}}\Sob{ f}{\Gdot^\lam_{\al,\si}}\right\}\Sob{\Lam^{\rho}h}{L^2}\notag\\
      &+ C\lam2^{(\si+1+\al-\ze)j}\Sob{{E}_\al^{\lam} S_{j-3}g}{\dot{H}^{1+\ze}}\Sob{{G}_\al^\lam\lpj f}{L^2}\Sob{\Lam^{\rho}h}{L^2},\notag
    \end{align}
for some constant $C>0$, depending only on $\si,\al,\ze,\nu,\rho$. 
\end{Lem}

\begin{proof}
As we will see below, the proof will make use of the fact that the symbol of $D$ is dominated by that of $\Lam$, and so in order to avoid redundancy in the argument, it will suffice to treat the case $D=\bdy_{\ell}$. 

First, let us define
    \begin{align*}
        \mathcal{L}^{\lam,\al,\si}_{j,\ell}(f,g,h):=\iint\limits_{\xi\in{\Acal}_j} m_{\al,\si,j,\ell}^{\lam}(\xi,\eta)\hat{f}(\xi-\eta)\hat{g}(\eta)\overline{\hat{h}(\xi)}\, d\eta \, d\xi,
    \end{align*}
where
    \begin{align*}
        m_{\al,\si,j,\ell}^{\lam}(\xi,\eta):=e^{\lam|\xi|^\al} \phi_{j}(\xi)\abs{\xi}^{\si}\xi_{\ell}-e^{\lam|\xi-\eta|^\al}\phi_{j}(\xe)\abs{\xi-\eta}^{\si}(\xi-\eta)_{\ell}.
    \end{align*}
Then, using Plancherel's theorem, we see that
    \begin{align}\label{def:L:comm:relation}
        \mathcal{L}^{\lam,\al,\si}_{j,\ell}(f,g,h)=\lb[{G}_\al^\lam\Lam^\si\bdy_\ell\lpj,g]f,h\rb.
    \end{align}
By \eqref{def:L:comm:relation}, it is therefore equivalent to obtain bounds for $\mathcal{L}_{j,\ell}^{\lam,\al,\si+\rho}$. For convenience, we will now suppress the indices on $m$.

Observe that from the triangle inequality, we have
    \begin{align}
         |m(\xi,\eta)|\le& \abs{\Ax^{\si}\xi_{\ell}-\Axe^{\si}(\xe)_{\ell}}e^{\lam|\xi|^\al}\abs{\xe}^{\rho}\phi_{j}(\xe)+ \abs{\abs{\xi}^{\rho}\phi_{j}(\xi)-\abs{\xe}^{\rho}\phi_{j}(\xe)}e^{\lam|\xi|^\al} \Ax^{\si+1}\notag\\& +\abs{e^{\lam|\xi|^\al}-e^{\lam|\xi-\eta|^\al}}\Axe^{\si+\rho+1}\phi_{j}(\xe)\notag\\
         =&m_1(\xi,\eta)+m_2(\xi,\eta)+m_3(\xi,\eta)\notag.
    \end{align}
Note that $m_j\geq0$, for each $j=1,2,3$.

Let $\mathbf{A}(\tau,\xi,\eta)$ be given as in \eqref{def:A}. Then we estimate $m_1$ as in \eqref{E:meanvalue1} and the triangle inequality, making use of the facts that $\si\in[0,1)$ and $\al\in(0,1]$. Since $\phi_{j}(\xe)=0$ whenever $\xe \notin {\Acal}_{j}$, we have
    \begin{align}\label{est:comm3:m1}
        m_1(\xi,\eta)
        &\leq C 2^{\rho j} e^{\lam|\xi-\eta|^\al}e^{\lam|\eta|^\al}\phi_{j}(\xe)|\eta|\int_0^1|\mathbf{A}(\tau)|^{\si}d\tau \notag\\
        &\leq C 2^{\rho j} e^{\lam|\xi-\eta|^\al}e^{\lam|\eta|^\al}\phi_{j}(\xe)|\eta|\left(\Axe^{\si}+\Ae^{\si}\right).
    \end{align}
Similarly, for $m_2$, we additionally use the fact that $\xi\in{\Acal}_j$ to estimate
    \begin{align}\label{est:comm3:m2}
        m_2(\xi,\eta)&\leq C2^{-j+\rho j}\left(\Sob{\phi_0}{L^\infty}+\Sob{{\nabla}\phi_0}{L^\infty}\right)|\eta||\xi|^{\si+1}e^{\lam|\xi|^\al}\notag\\
        &\leq C 2^{\rho j}e^{\lam|\xi-\eta|^\al}e^{\lam|\eta|^\al}|\eta|\left(|\xi-\eta|^{\si}+|\eta|^{\si}\right).
    \end{align}
Finally, for $m_3$, let us first observe that $\xi,\xi-\eta\in{\Acal}_j$ implies $\eta\in{\Bcal}_{j+2}$. It then follows from the fact $\xi\in{\Acal}_j$ that
    \begin{align*}
        m_3(\xi,\eta)&=m_3(\xi,\eta)\mathbbm{1}_{{\Acal}_j}(\xi)\mathbbm{1}_{{\Acal}_j}(\xi-\eta)\left(\mathbbm{1}_{{\Bcal}_{j-3}}(\eta)+\mathbbm{1}_{{\Acal}_{j-3,j+2}}(\eta)\right)\notag\\
        &=m_{3}(\xi,\eta)\mathbbm{1}_{{\Acal}_j}(\xi)\mathbbm{1}_{{\Acal}_j}(\xi-\eta)\mathbbm{1}_{{\Bcal}_{j-3}}(\eta)+m_{3}(\xi,\eta)\mathbbm{1}_{{\Acal}_j}(\xi)\mathbbm{1}_{{\Acal}_j}(\xi-\eta)\mathbbm{1}_{{\Acal}_{j-3,j+2}}(\eta)\notag\\
        &=m_3^{(1)}(\xi,\eta)+m_3^{(2)}(\xi,\eta).
    \end{align*}
For $m_3^{(1)}$, observe that for $\xi-\eta\in{\Acal}_j$ and $\eta\in{{\Bcal}_{j-3}}(\eta)$, we have
    \begin{align}\label{est:A:ball}
        |{\bf{A}(\tau,\xi,\eta)}|^{\al-1}\leq C2^{(\al-1)j},
    \end{align}
It follows from \eqref{est:A:ball} that
    \begin{align}
       \abs{e^{\lam|\xi|^\al}-e^{\lam|\xi-\eta|^\al}}&=\left| \int_0^1\frac{d}{d\tau}\left(e^{\lam{|\bf{\Acal}}(\tau)|^\al}\right)d\tau\right|\notag\\
       &\leq \al\lam|\eta| \int_0^1e^{\lam|{\bf{\Acal}(\tau)}|^\al}|{\bf{\Acal}(\tau)}|^{\al-1}d\tau\leq C\lam|\eta| 2^{(\al-1)j}e^{\lam|\xi-\eta|^\al}\int_0^1e^{\lam\tau^\al|\eta|^\al}d\tau. \notag
    \end{align}
Thus
    \begin{align}\label{est:m3:exp1}
        m_3^{(1)}(\xi,\eta)\leq C\lam|\eta| 2^{(\al-1+\rho)j}e^{\lam|\xi-\eta|^\al}        \left(\int_0^1e^{\lam\tau^\al|\eta|^\al}d\tau\right)|\xi-\eta|^{\si+1}\phi_j(\xi-\eta)\mathbbm{1}_{{\Acal}_j}(\xi)\mathbbm{1}_{{\Bcal}_{j-3}}(\eta).
    \end{align}
On the other hand, for $m_3^{(2)}$, we have
    \begin{align}\label{est:m3:exp2}
        m_3^{(2)}(\xi,\eta)&\leq e^{\lam|\xi-\eta|^\al}\left(e^{\lam|\eta|^\al}-1\right)|\xi-\eta|^{\si+1+\rho}\phi_j(\xi-\eta)\mathbbm{1}_{{\Acal}_j}(\xi)\mathbbm{1}_{{\Acal}_{j-3,j+2}}(\eta)\notag\\
        &\leq C2^{\rho j}e^{\lam|\xi-\eta|^\al}e^{\lam|\eta|^\al}|\eta||\xi-\eta|^{\si}\phi_j(\xi-\eta)\mathbbm{1}_{{\Acal}_j}(\xi)\mathbbm{1}_{{\Acal}_{j-3,j+2}}(\eta).
    \end{align}
Hence, upon combining \eqref{est:m3:exp1} and \eqref{est:m3:exp2}, we have
    \begin{align}\label{est:comm3:m3}
        m_3(\xi,\eta)\leq& C2^{\rho j}\lam|\eta| 2^{(\al-1)j}e^{\lam|\xi-\eta|^\al}        \left(\int_0^1e^{\lam\tau^\al|\eta|^\al}d\tau\right)|\xi-\eta|^{\si+1}\phi_j(\xi-\eta)\mathbbm{1}_{{\Acal}_j}(\xi)\mathbbm{1}_{{\Bcal}_{j-3}}(\eta)\notag\\
        &+C2^{\rho j}e^{\lam|\xi-\eta|^\al}e^{\lam|\eta|^\al}|\eta||\xi-\eta|^{\si}\phi_j(\xi-\eta)\mathbbm{1}_{{\Acal}_j}(\xi)\mathbbm{1}_{{\Acal}_{j-3,j+2}}(\eta).
    \end{align}

Upon returning to \eqref{def:L:comm:relation}, and applying \eqref{est:comm3:m1}, \eqref{est:comm3:m2}, and \eqref{est:comm3:m3}, then using the notation in \eqref{def:gev:avg}, we obtain
    \begin{align}
        |\mathcal{L}^{\lam,\al,\si+\rho}_{j,\ell}(f,g,h)|
        \leq& C\iint\limits_{\xi\in{\Acal}_j}\left(|\xi-\eta|^{\si}+|\eta|^{\si}\right)|\mathcal{F}({G}_\al^\lam f)(\xi-\eta)||\mathcal{F}({G}_\al^\lam \Lam g)(\eta)|\abs{\xi}^{\rho}|\hat{h}(\xi)|d\eta d\xi\notag\\
        &+C\lam2^{(\al-1)j}\iint\limits_{\xi\in{\Acal}_j}|\mathcal{F}({G}_\al^\lam\Lam^{\si+1}\lpj f)(\xi-\eta)||\mathcal{F}({E}_\al^{\lam}\Lam S_{j-3}g)(\eta)|\abs{\xi}^{\rho}|\hat{h}(\xi)|d\eta d\xi\notag\\
        \leq& L_1+L_2.\notag
    \end{align}
For $L_1$ we may use \cref{cor:commutator1}
with $s=\si$ and $\eps={\nu}$ to obtain 
    \begin{align}\label{est:L1:comm3:a}
        L_1\leq 
        C_j2^{{\nu} j}\min\left\{\Sob{{G}_\al^\lam f}{\dot{H}^{1-{\nu}}}\Sob{{G}_\al^\lam  g}{\dot{H}^{\si+1}},\Sob{{G}_\al^\lam g}{\dot{H}^{2-{\nu}}}\Sob{{G}_\al^\lam f}{\dot{H}^{\si}}\right\}\Sob{\Lam^{\rho}h}{L^2},
    \end{align}
for $\si \in [0,1)$, $\nu \in (0,1)$, for some $\{C_j\}_j\in\ell^2(\ZZ)$.

For $L_2$, we use the Cauchy-Schwarz inequality,  Young's convolution inequality, Plancherel's theorem, and Bernstein's inequality, to obtain for any $\ze\in[0,1)$
    \begin{align}\label{est:L2:comm3}
        L_2&\leq C\lam 2^{(\al-1)j}\Sob{{G}_\al^\lam\Lam^{\si+1}\lpj f}{L^2}\Sob{\Ft({E}_\al^{\lam}\Lam S_{j-3}g)}{L^1}\Sob{\Lam^{\rho}h}{L^2}\notag\\
        &\leq C\lam2^{(\al-\ze)j}\Sob{{G}_\al^\lam\Lam^{\si+1}\lpj f}{L^2}\Sob{{E}_\al^{\lam}\Lam^{1+\ze} S_{j-3}g}{L^2}\Sob{\Lam^{\rho}h}{L^2}\notag\\
        &\leq  C\lam2^{(\si+1+\al-\ze)j}\Sob{{G}_\al^\lam\lpj f}{L^2}\Sob{{E}_\al^{\lam}\Lam^{1+\ze} S_{j-3}g}{L^2}\Sob{\Lam^{\rho}h}{L^2}.
    \end{align}
Upon adding \eqref{est:L1:comm3:a}
and \eqref{est:L2:comm3} we obtain the desired result.
\end{proof}

\subsection{Commutator estimates with logarithmic multipliers}

\begin{Lem}\label{lem:commutator4}
Let $\mu,\rho >0$, $\eps\in(0,1)$, and $\de\in(0,2\mu)$. Suppose $f,g,h\in L^2$.
Then there exists a constant $C>0$, depending only on $\mu,{\epsilon}, \de$, such that
    \begin{align}
       | \langle[\ld^{\mu}\partial_{\ell},g]f,h  \rangle|\leq C\Sob{g}{\dot{H}^{2-\eps+\rho}}^{\frac{1}{1+\rho}}\Sob{g}{\Hdot^{1-\eps}}^{\frac{\rho}{1+\rho}}\left(\Sob{f}{\Hdot^{{\eps+\de}}}\Sob{h}{L^2}+\Sob{f}{L^2}\Sob{h}{\Hdot^{{\eps+\de}}}\right),\quad \text{for}\ \ell=1,2.\notag
    \end{align}
\end{Lem}

\begin{proof}
We consider the following functional:
    \begin{align}\label{def:L:functional:log}
        \mathcal{L}_{\mu,\ell}(f,g,h):=\iint m_{\mu,\ell}(\xi,\eta)\hat{f}(\xi-\eta)\hat{g}(\eta)\overline{\hat{h}(\xi)}d\eta d\xi,
    \end{align}
where
    \begin{align}\notag
        m_{\mu,\ell}(\xi,\eta):=\lnx^{\mu}\xi_\ell-\lnxe^{\mu}(\xi-\eta)_\ell.
    \end{align}
By the Plancherel theorem, observe that
    \[
        \lb [\ldmu \bdy_\ell,g]f ,h\rb=\mathcal{L}_{\mu,\ell}(f,g,h).
    \]
As before we set
    \begin{align}\notag
    \mathbf{A}(\tau,\xi,\eta):=\tau\xi+(1-\tau)(\xe).
    \end{align}
For convenience, we will suppress the dependence of $\mathbf{A}$ on $\xi,\eta$. Observe that by the elementary inequality $\frac{x}{(1+x)}\le \ln(1+x)$, we have
    \begin{align*}
    \begin{split}
        |m_{\mu,\ell}(\xi,\eta)|&=\abs{\int_{0}^{1}\frac{d}{{d\tau}}\left(\left( \ln \left(1+\abs{\mathbf{A}(\tau)}^{2}\right)\right)^{\mu}\mathbf{A}(\tau)_{\ell}\right)d\tau}\\
        &\le C\Ae\int_{0}^{1}\left(\ln\left(1+\abs{\mathbf{A}(\tau)}^{2}\right)\right)^{\mu-1}\left(\frac{\abs{\mathbf{A}(\tau)}^{2}}{1+\abs{\mathbf{A}(\tau)}^{2}}\right)+\left(\ln\left(1+\abs{\mathbf{A}(\tau)}^{2}\right)\right)^{\mu}\,d\tau\\
	&\le C\Ae\left(\ln\left(1+\text{max}\{\,\Ax^{2},\Axe^{2}\}\right)\right)^{\mu} \\
	&\le C\Ae^{1-{\eps}}\max\left\{\Axe^{{\eps}}\lnxe^{\mu},\,\Ax^{{\eps}}\lnx^{\mu}\right\}.
    \end{split}
    \end{align*}
Using the above inequality in (\ref{def:L:functional:log}), we have 
\begin{align*}
        |\mathcal{L}_{\mu,\ell}(f,g,h)|\le& \iint \Ae^{1-{\eps}} |\hat{g}(\eta)|\lnxe^{\mu}\Axe^{{\eps}}|\hat{f}(\xi-\eta)|\hat{h}(\xi)|d\eta d\xi \\
        &+C\iint \Ae^{1-{\eps}} |\hat{g}(\eta)||\hat{f}(\xi-\eta)\lnx^{\mu}\Ax^{{\eps}}|\hat{h}(\xi)|d\eta d\xi\\
        =&I+II.
    \end{align*}
We make use of the following elementary inequality: for any $\al\in(0,1)$
\begin{align}\label{E:elementary:log}
    \log(1+x)\leq C_\al x^\al.
\end{align}
In particular, we have $((\log(1+|\eta|^2))^\mu\leq C_{\mu,\de}|\eta|^{\de}$, whenever $\de\in(0,2\mu)$. We apply the Cauchy-Schwarz inequality,  Young's convolution inequality, \eqref{E:interpolation:L1} with $\si=1-\eps$, $\si_1=0$, $\si_2=1+\rho$, and Plancherel's theorem, to obtain
\begin{align*}
        |I| &\le C
            \Sob{{|\eta|^{1-{\eps}}\hat{g}}(\eta)}{L^1}\Sob{\Ae^{{\eps+\de}}\hat{f}(\eta)}{L^2}\Sob{\hat{h}(\eta)}{L^2}\\
            &\le C\Sob{g}{\dot{H}^{2-\eps+\rho}}^{\frac{1}{1+\rho}}\Sob{g}{\Hdot^{1-\eps}}^{\frac{\rho}{1+\rho}}\Sob{f}{\Hdot^{{\eps+\de}}}\Sob{h}{L^2}.
    \end{align*}
Similarly, we can show
\begin{align*}
     |II|\le C\Sob{g}{\dot{H}^{2-\eps+\rho}}^{\frac{1}{1+\rho}}\Sob{g}{\Hdot^{1-\eps}}^{\frac{\rho}{1+\rho}}\Sob{f}{L^2}\Sob{h}{\Hdot^{{\eps+\de}}},
\end{align*}
thus completing the proof.
\end{proof} 

\begin{Rmk}\label{remark:log}
Note that by using Plancherel's theorem and the inequality (\ref{E:elementary:log}), we can also deduce that for any $\eps, \mu,\de >0$ satisfying $\de \in (0,2\mu)$
\begin{align}\label{E:log:inequality}
    \Sob{\ldmu f}{\Hdot^{\eps}}\le C_{\mu, \eps, \de}\Sob{f}{\Hdot^{{\eps+\de}}}.
\end{align}
\end{Rmk}

\section{Dissipative perturbation of a linear conservation law with modified flux}\label{sect:mod:flux}

The proof of our first main result, \cref{thm:main:beta}, will rely on an approximating sequence that is determined by a linear scalar conservation law that is dissipatively perturbed by the appropriate power of the fractional laplacian. To be specific, given $q$ sufficiently smooth, we will consider the following initial value problem:
\begin{align}\label{E:mod:claw}
    {\partial_{t}\theta+  \Div F_q(\tht)=-{\gam} \Lam^{\kap}\tht},\quad\tht(0,x)=\tht_0(x).
\end{align}
where 
\begin{align}\label{def:mod:flux}
F_q(\tht)=\begin{cases} 
 (\nabla^{\perp}\Lam^{\be-2}q) \theta &\text{if $\be<1+\kap$}\\
 (\nabla^{\perp}\Lam^{\be-2}q) \theta+\Lam^{\be-2}(({\nabla}^\perp\tht)q)\quad &\text{if $\be\geq 1+\kap$}.
\end{cases}
\end{align}
Note that one formally has $\Div F_{-\tht}(\tht)=-({\nabla}^\perp\Lam^{\be-2}\tht)\cdotp{\nabla}\tht=u\cdotp{\nabla}\tht$. Hence, one recovers equation (\ref{E:dissipative-beta}) in the case $q=-\tht$.
The purpose of this particular modification to the flux is to accommodate additional commutators in the study of \eqref{E:dissipative-beta} that the ``standard" approximating sequence of linear transport equations cannot handle. We observe that when $\be<1+\kap$, that is, $\si_c<2$, where $\sigma_c=\be+1-\kap$ denotes the critical Sobolev exponent, no modification is required and one may simply use the standard approximating sequence by a linear transport equation, as indicated by \eqref{def:mod:flux}. However, a modification is crucial for treating the regime corresponding to $\be\geq 1+\kap$, i.e., $\si_c\geq2$, of the more singular velocities in \eqref{E:dissipative-beta}. In this regard, the proposed equation \eqref{E:mod:claw}  faithfully respects the more nuanced commutator structure of the generalized SQG equations required to treat the more singular regime of $\beta\geq1+\kap$. We will first establish existence and uniqueness of solutions to \eqref{E:mod:claw}
\begin{Thm}{\label{T:transport}}
Let $\be\in(1,2)$, $\kap\in(0,1)$, and $\si_c=1+\be-\kap$. Given $T>0$, suppose $q\in L^2(0,T;\dot{H}^{\si_c+\kap/2})$. Then for each $\tht_0\in H^{\si_c}$, (\ref{E:mod:claw}) has a unique solution $\theta\in C([0,T];H^{\si_c})\cap L^{2}(0,T;\Hdot^{\si_{c}+\frac{\kap}{2}}) $ satisfying
\begin{align*}
    \sup_{0\leq t\leq T}\nrm{\tht(t)}_{H^{\si_{c}}}\le \nrm{\tht_0}_{H^{\si_c}}\exp \left(C\Sob{q}{L^2_T\dot{H}^{\si_c+\kap/2}}^2\right).
\end{align*}
\end{Thm}
\cref{T:transport} can be proved by an artificial viscosity argument. A sketch of the proof is provided in \cref{sect:app:thm}. The reader is referred to \cite{KumarThesis2021} for additional details.
With this in hand, we will only develop apriori estimates for solutions to (\ref{E:mod:claw}). We ultimately find it expedient to perform these estimates in Gevrey classes and simply specialize them to the case where the exponential rate, $\lam$, in the Gevrey norm is identically zero to obtain estimates in the corresponding Sobolev spaces. 

\subsection{Apriori estimates}
Given $q$, let $v$ denote
    \begin{align}\label{def:vel:phi}
        v:=-{\nabla}^\perp\Lam^{\be-2}q.
    \end{align}
Then ${\nabla}\cdotp v=0$. We will begin by establishing $L^2$ estimates. Then we will proceed to establishing the claimed Sobolev space estimates. Lastly, we provide estimates in the Gevrey norm topology $\dot{G}^{\lam}_{\al,\si_c+\delta}$. 

\subsubsection{$L^2$ estimates} 
Since $v$ is divergence-free, integrating by parts, we have $\lb v\cdotp{\nabla} h,h\rb=0$, for any $h$ sufficiently smooth. Combined with the skew self-adjointness of the operator ${\nabla}\Lam^{\be-2}$, we deduce
\begin{align}\label{E:skew-adjoint}
    \langle {\nabla}\cdotp F_{q}(\tht),\theta\rangle=\lb\Lam^{\be-2}{\nabla}\cdotp(({\nabla}^\perp \tht)q),\tht\rb=-\lb\Lam^{\be-2}{\nabla}\cdotp(({\nabla}^\perp q)\tht),\tht\rb=-\frac{1}2\lb [{\nabla}\Lam^{\be-2},{\nabla}^\perp q]\tht,\tht\rb.
\end{align}
By \cref{lem:commutator2} with $\rho_1=\rho_2=\kap/2$, and Young's inequality, we have
\begin{align}
    |\langle {\nabla}\cdotp F_{q}(\tht),\theta\rangle|&=\frac{1}2|\lb [\Lam^{\be-2}{\nabla},{\nabla}^\perp {q}]\tht,\tht\rb|\notag\\
    &\le C 
    \|\theta\|_{\dot{H}^{\frac{\kap}{2}}}\|\theta\|_{L^2}\|{q}\|_{\dot{H}^{\si_{c}+\frac{\kap}{2}}}\notag \\
    &\le \frac{\gam}{100}\|\theta\|^{2}_{\dot{H}^{\frac{\kappa}{2}}}+C\|\theta\|^{2}_{L^2}\|{q}\|_{\dot{H}^{\si_{c}+\frac{\kap}{2}}}^{2}. \label{E:Young}
\end{align}
Taking the inner product in $L^2$ of (\ref{E:mod:claw}) with $\tht$ and using (\ref{E:Young}) yields
\begin{align*}
    \frac{d\|\theta\|^{2}_{L^2}}{dt}+\gam\Sob{\tht}{\Hdot^{\frac{\kap}{2}}}^{2}\le C\|\theta\|^{2}_{L^2}\|{q}\|_{\dot{H}^{\si_{c}+\frac{\kap}2}}^{2}.
\end{align*}
Integrating in time, we obtain
\begin{align}\label{est:L2:critical}
     \nrm{\tht}_{L^{\infty}_{T}L^2}^{2}\le\nrm{\tht_0}_{L^2}^{2}\exp(C\nrm{q}_{L^{2}_{T}\Hdot^{\si_{c}+\frac{\kap}{2}}}^{2}). 
\end{align}

\subsubsection{Preparatory estimates}
It will be economical at this point to introduce the Gevrey operator and derive the estimates with this operator included since the commutator estimates that we apply will reduce accordingly to the Sobolev setting upon setting the rate, $\lam$, in the Gevrey norm to be identically zero. Since we will have to make different choices for various parameters in each setting, we will then specialize to the Sobolev setting first, and then return to the Gevrey setting again afterwards.

For the remainder of \cref{sect:mod:flux}, let us assume that 
    \begin{align*}
        0\leq{\de}<\kap,\quad\lam(t)\quad \text{is differentiable in $t$},\quad \lam(0)=0.
    \end{align*}
Observe that $\lam\equiv0$ is allowed. To help contain expressions, we will make use of the notations
    \begin{align*}
        \til{h}_j:={G}_\al^{\lam(t)}\lpj h,\quad \Lam^\si_j:=\Lam^\si\lpj,\quad \si\in\RR.
    \end{align*}
{Observe that for $\ph=\ph(t,x)$, sufficiently smooth in $t,x$, we have}
    \begin{align}\label{identity:dt:Gevrey}
        {\bdy_t({G}_\al^{\lam(t)}\ph)=\lam'(t){G}_\al^{\lam(t)}\Lam^\al\ph+{G}_\al^{\lam(t)}\bdy_t\ph.}
    \end{align}
{Upon applying the operator ${G}_\al^{\lam(t)}\Lam_j^{\si_{c}+\delta}$ to (\ref{E:mod:claw}) and invoking \eqref{identity:dt:Gevrey} with $\ph=\Lam_j^{\si_{c}+\de}\tht$, one has}
    \begin{align}\label{E:lambda-beta}
    	\begin{split}
	     { \bdy_t(\Lam^{\si_c+{\de}}\til{\tht}_j)-}&{\lam'(t)\Lam^{{\si_c+{\de}}+\al}\til{\tht}_j +{\gam} \Lam^{\si_c+{\de}+\kappa}}\til{\tht}_j +{G}^{\lam(t)}_\al\Lam^{\si_c+\de}_j{\nabla}\cdotp F_q(\theta)=0.
    	\end{split}
	\end{align}
{Then by \eqref{def:vel:phi}, taking the inner product in $L^2$ of \eqref{E:lambda-beta} with ${G}_\al^{\lam(t)}\Lam_j^{\sigma_{c}+\de}\tht$ yields}
	\begin{align}\label{balance:gevrey:basic}
	   { \frac{1}{2}\frac{d}{dt}\Sob{\Lam^{\si_c+{\de}}\til{\tht}_j}{L^2}^2+}{{\gam} \Sob{\Lam^{{\si_c+{\de}}+\kap/2}\til{\tht}_j}{L^2}^2}{=\lam'(t)\Sob{\Lam^{{\si_c+{\de}}+\al/2}\til{\tht}_j}{L^2}^2+\lb {G}_\al^{\lam(t)}\Lam_j^{\si_c+{\de}}{\nabla}\cdotp F_q(\tht),\Lam^{\si_c+{\de}}\til{\tht}_j\rb.}
	\end{align}
We will now treat the trilinear terms; it will be divided into two cases: ${\be<1+\kap}$ and ${\be\ge1+\kap}$.

\subsubsection*{\textbf{Case: $\be<1+\kap$}}. Observe that in this case ${1<\be<{\si_c}<2}$ and $F_q(\tht)=-v\tht$. Let  $\de \in [0,2-\si_c)$. We decompose the term involving the flux in \eqref{balance:gevrey:basic} as
	\begin{align*}
	I:=&\left \lb {G}_\al^{\lam(t)}\Lam_j^{\si_c+{\de}}{\nabla}\cdotp(v\tht),\Lam^{\si_c+{\de}}\tiltht_j\right\rb \\
	=&  \left(\left  \langle {G}_\al^{\lam(t)}\Lam^{\si_c+{\de}}_j({v} \cdot \nabla) \theta, \Lam^{\si_c+{\de}}\tiltht_j  \right \rangle -\left \langle ({v}\cdot\nabla) {G}_\al^{\lam(t)}\Lam_j^{\si_c+{\de}}\tht,\Lam^{\si_c+{\de}} \tiltht_j \right \rangle\right)+\left \langle ({v}\cdotp {\nabla}) \Lam^{\si_c+{\de}} \tiltht_j,\Lam^{\si_c+{\de}} \tiltht_j \right \rangle \notag\\
	=&I_1+I_2\notag
	\end{align*}
Since $I_{2}=0$, it suffices to obtain a bound on $I_{1}$. Observe that
     \begin{align}\notag
    I_{1}=\sum_{\ell=1,2}\left \langle [{G}^{\lam(t)}_\al \Lam^{{\si_c+{\de}}}_j,v^\ell]\bdy_{\ell}\tht,\Lam^{\si_c+{\de}}\tiltht_j  \right \rangle. 
    \end{align}
Since $\de<2-\si_c$, we may apply \cref{lem:commutator3} with { ${\si}=\be-\kap+{\de}$, $\rho=0$, and $f=\bdy_\ell\tht$, $g= v^\ell$, $h=\Lam^{\si_c+{\de}}\tiltht_j$}, so that, along with applying Bernstein's inequalities, we obtain
\begin{align}\label{est:I1}
    \abs{I_{1}}\leq& C\left(c_j2^{\nu j}\Sob{v}{\dot{G}_{\al,2-\nu}^{\lam(t)}}\Sob{{\nabla}\tht}{\dot{G}^{\lam(t)}_{\al,{\be-\kap+{\de}}}}+ \lam(t)2^{(\si_c+\de+\al-\ze)j}\Sob{{E}_\al^{\lam(t)} S_{j-3}v}{\dot{H}^{1+\ze}}\Sob{{\nabla}\tiltht_j}{L^2}\right)\Sob{\Lam^{\si_c+{\de}}\tiltht_j}{L^2}\notag\\
    \leq& C\left(c_j2^{\nu j}\Sob{q}{\dot{G}_{\al,1+\be-\nu}^{\lam(t)}}\Sob{\tht}{\dot{G}^{\lam(t)}_{\al,{\si_c+{\de}}}}+ \lam(t)2^{(1+\al-\ze)j}\Sob{{E}_\al^{\lam(t)} S_{j-3}q}{\dot{H}^{\be+\ze}}\Sob{\tiltht_j}{\Hdot^{\si_c+\de}}\right)\Sob{\Lam^{\si_c+{\de}}\tiltht_j}{L^2},
\end{align}
where $\ze\in[0,1)$, $\nu\in(0,1)$, and $\{c_j\}\in\ell^2(\ZZ)$ with $\Sob{\{c_j\}}{\ell^2}\leq1$.

\subsubsection*{\textbf{Case: ${\be\ge 1+\kap}$} }
Observe that in this case $\si_c\geq2$. We decompose the term involving the flux in \eqref{balance:gevrey:basic} as
	\begin{align*}
	J:&=\left\lb {G}_\al^{\lam(t)}\Lam_j^{\si_c+{\de}}{\nabla}\cdotp F_q(\tht),\Lam^{\si_c+{\de}}\til{\tht}_j\right\rb\notag\\
	&=\underbrace{\left  \lb {G}_\al^{\lam(t)}\Lam_j^{\si_c+{\de}} (\nabla^{\perp}\Lambda^{\be-2}{q} \cdot \nabla \tht),\Lam^{\si_c+{\de}}\til{\tht}_j\right \rb}_{J^{a}}-\underbrace{\left \lb {G}_\al^{\lam(t)}\Lam_j^{\si_c+{\de}+\be-2} (\nabla^{\perp}{q} \cdot \nabla \tht),\Lam^{\si_c+{\de}}\til{\tht}_j\right\rb\notag}_{J^b}\\
	&=J_{1}+J_{2}+J_{3}+J_{4}+J_{5},
\end{align*}where
\begin{align}
	J_{1}=& \underbrace{\left \langle (\nabla^{\perp}\Lambda^{\be-2}\Lam^{\si_c+{\de}}\til{{q}}_j \cdot \nabla) \tht,\Lam^{\si_c+{\de}} \tiltht_j \right \rangle}_{J_{1}^a} -\underbrace{\left \langle \nabla^{\perp}\Lambda^{\be-2}\cdot(\Lam^{\si_c+{\de}}\til{{q}}_j \nabla \tht),\Lam^{\si_c+{\de}} \tiltht_j \right \rangle\notag}_{J_{1}^b}\\
	J_{2}=&\left \langle (\nabla^{\perp}\Lambda^{\be-2}{q}\cdot {\nabla} \Lam^{\si_c+{\de}} \tiltht_j),\Lam^{\si_c+{\de}} \tiltht_j \right \rangle\notag\\
	J_{3}=&J^{a}-J_{1}^{a}-J_{2}\notag\\
	=& \left \{ \left  \langle {G}_\al^{\lam(t)}\Lam^{\si_c+{\de}}_j(\nabla^{\perp}\Lambda^{\be-2}{q} \cdot \nabla \theta),\Lam^{\si_c+{\de}}\tiltht_j  \right \rangle
	- \left \langle (G_\al^{\lam(t)}\Lam_j^{\si_c+{\de}}\nabla^{\perp}\Lambda^{\be-2}{{q}} \cdot \nabla) \tht,\Lam^{\si_c+{\de}} \tiltht_j \right \rangle \right.\notag \\ 
	&\left.\quad -\left \langle (\nabla^{\perp}\Lambda^{\be-2}{q}\cdot{\nabla} G_\al^{\lam(t)}\Lam_j^{\si_c+{\de}} \tht),\Lam^{\si_c+{\de}} \tiltht_j \right \rangle \right\}\notag\\
	J_{4}=&-\left \langle \Lam^{\frac{\be-2}2}(\nabla^{\perp}{q}\cdot {\nabla} \Lam^{\frac{\be-2}{2}}\Lam^{\si_c+{\de}} \tiltht_j),\Lam^{\si_c+{\de}} \tiltht_j \right \rangle\notag\\
	J_{5}=&-J^{b}+J_{1}^{b}-J_{4} \notag\\=&- \left \{ \left  \langle {G}_\al^{\lam(t)}\Lam^{\si_c+{\de}}_j\Lam^{\be-2}(\nabla^{\perp}{q} \cdot \nabla \theta),\Lam^{\si_c+{\de}}\tiltht_j  \right \rangle
	- \left \langle \nabla^{\perp}\Lambda^{\be-2}\cdotp((G_\al^{\lam(t)}\Lam_j^{\si_c+{\de}}{{q}}) \nabla \tht),\Lam^{\si_c+{\de}} \tiltht_j \right \rangle \right.\notag \\ 
	&\left.\quad -\left \langle \Lam^{\frac{\be-2}2}(\nabla^{\perp}{q}\cdot {\nabla} G_\al^{\lam(t)}\Lam_j^{\si_c+{\de}}\Lam^{\frac{\be-2}{2}} \tht), \Lam^{\si_c+{\de}}\tiltht_j \right \rangle \right\}\notag
\end{align}
Observe that by integrating by parts, we derive $J_{2},J_{4}=0$, so that it suffices to treat $J_1, J_3$ and $J_{5}$.

\subsubsection*{Bound for ${J_1}$\nopunct}:
	 Letting $\bdy^\perp=(-\bdy_2,\bdy_1)$, observe that we can write $J_1$ as
	\begin{align}
	J_{1}&= -\left \langle [\Lambda^{\beta-2}\bdy^\perp_\ell,\bdy_{\ell}\tht]\Lam^{\si_c+{\de}}\til{{q}}_j,\Lam^{\si_c+{\de}}\tiltht_j\right \rangle,\label{J1:commutator:rep}
	\end{align} 
where we sum over repeated indices. By \cref{lem:commutator2} with { $f=\Lam^{\si_c+{\de}}\til{{q}}_j$, $g=\bdy_\ell\tht$,  $h=\Lam^{\si_c+{\de}}\til{\tht}_j$, and ${\rho_1=\rho_2}=\kap-{\de}$}, and by Bernstein's inequality, it follows that
    \begin{align}\label{est:J1}
        |J_1|&\leq C\Sob{{\nabla}\tht}{\dot{H}^{\be-\kap+{\de}}}\Sob{\Lam^{\si_c+{\de}}\til{{q}}_j}{\dot{H}^{\kap-\de}}\Sob{\Lam^{\si_c+{\de}}\til{\tht}_j}{L^2}\notag\\
        &\leq Cc_j2^{\nu j}\Sob{\tht}{\dot{H}^{{\si_c+{\de}}}}\Sob{{q}}{\dot{G}_{\al,1+\be-\nu}^{\lam(t)}}\Sob{\Lam^{\si_c+{\de}}\tiltht_j}{L^2},
    \end{align}
for any $\nu\in\RR$, for some $\{c_j\}\in\ell^2(\ZZ)$ such that $\Sob{\{c_j\}}{\ell^2}\leq1$.

\subsubsection*{{Bound for} ${J_3}$\nopunct}:
    We will make use of the following notation
        \begin{align}\label{def:A:perp}
        A=\Lam^{\be-2}\bdy^\perp,\quad \bdy^\perp=(-\bdy_2,\bdy_1),
        \end{align}
    so that $A_\ell=\Lam^{\be-2}\bdy^\perp_\ell$, for $\ell=1,2$. Observe that we may then rewrite $J_3$ as
	\begin{align*}
	\begin{split}
	J_{3}=&\left \langle {G}_\al^{\lam(t)}\Lam_j^{\si_c+{\de}}(A_{\ell}{q}\,\bdy_{\ell} \theta),\Lam^{\si_c+{\de}}\tiltht_j  \right \rangle
	- \left \langle G_\al^{\lam(t)}(\Lam^{\si_c+{\de}} A_{\ell}{{q}}_j) \bdy_\ell \tht,\Lam^{\si_c+{\de}} \tiltht_j \right \rangle 
	-\left \langle A_{\ell}{q}\,\bdy_\ell \Lam^{\si_c+{\de}} G_\al^{\lam(t)}\tht_j,\Lam^{\si_c+{\de}}\tiltht_j \right \rangle.
	\end{split}
	\end{align*}
	We observe, as in \cite{HuKukavicaZiane2015}, that we may write $J_{3}$ as a double commutator. Indeed, for any $\widetilde{\si}>2$, we have
	\begin{align}\label{split:Lam:rewrite}
	\Lam^{\widetilde{\si}}f=\Lam^{\widetilde{\si}-2}(-\de)f=-(\Lam^{\widetilde{\si}-2}\bdy_{l})\bdy_{l}f.
	\end{align}
	Then by applying \eqref{split:Lam:rewrite} and the product rule, we have
	\begin{align*}
	\begin{split}
	J_{3}=&-\left \langle {G}_\al^{\lam(t)} \Lam_j^{{\si_c+{\de}}-2}\bdy_l(\bdy_{l}A_{\ell}{q}\,\bdy_{\ell}\theta),\Lam^{\si_c+{\de}} \tiltht_j  \right \rangle
	+ \left \langle  ({G}_\al^{\lam(t)}\Lam_j^{{\si_c+{\de}}-2}\bdy_l(\bdy_{l}A_{\ell}{q})) \bdy_{\ell} \tht,  \Lam^{\si_c+{\de}}\tiltht_j \right \rangle \\
    	&-\left \langle {G}_\al^{\lam(t)} \Lam_j^{{\si_c+{\de}}-2}\bdy_l(A_{\ell}{q}\,\bdy_{l}\bdy_{\ell} \theta),\Lam^{\si_c+{\de}} \tiltht_j  \right \rangle
    	+\left \langle A_{\ell}{q}  ({G}_\al^{\lam(t)}\Lam_j^{{\si_c+{\de}}-2}\bdy_l(\bdy_{l}\bdy_{\ell} \tht)),\Lam^{\si_c+{\de}} \tiltht_j \right \rangle  \notag\\
    	=&-\left\lb [{G}_\al^{\lam(t)} \Lam_j^{{\si_c+{\de}}-2}\bdy_l,\bdy_\ell\tht]\bdy_{l}A_{\ell}{q},\Lam^{\si_c+{\de}}\tiltht_j\right\rb
    	-\left\lb [{G}_\al^{\lam(t)} \Lam_j^{{\si_c+{\de}}-2}\bdy_l,A_{\ell}{q}]\bdy_l\bdy_\ell\tht,\Lam^{\si_c+{\de}}\tiltht_j\right\rb\notag\\
    	=&J_3^a+J_3^b,
	\end{split}
	\end{align*}
where we sum over repeated indices. By \cref{lem:commutator3} with  {$f=\bdy_{l}A_{\ell}{q}$, $g=\bdy_\ell\tht$, $h=\Lam^{\si_c+{\de}}\tiltht_j$} and {$\rho=0$, ${\si}=\si_c-2+\de$}, and by Bernstein's inequality, we have
    \begin{align}\label{est:J3a}
         |J_3^a|\leq& C\left(c_j2^{\nu j}\Sob{\Lam^\be {q}}{\dot{G}^{\lam(t)}_{\al,1-\nu}}\Sob{{\nabla} \tht}{\dot{G}^{\lam(t)}_{\al,{\si_c-1+\de}}}
      + \lam(t)2^{(\si_c-1+\de+\al-\ze)j}\Sob{{E}_\al^{\lam(t)} S_{j-3}{\nabla}\tht}{\dot{H}^{1+\ze}}\Sob{\Lam^\be\til{q}_j}{L^2}\right)\Sob{\Lam^{\si_c+{\de}}\tiltht_j}{L^2}\notag\\
      \leq& Cc_j\left(2^{\nu j}\Sob{ {q}}{\dot{G}^{\lam(t)}_{\al,1+\be-\nu}}\Sob{ \tht}{\dot{G}^{\lam(t)}_{\al,{\si_c+\de}}}
      + \lam(t)2^{(1+\al-\ze)j}\Sob{{E}_\al^{\lam(t)} S_{j-3}\tht}{\dot{H}^{\be+\ze}}\Sob{ {q}}{\dot{G}^{\lam(t)}_{\al,\si_c+\de}}\right)\Sob{\Lam^{\si_c+{\de}}\tiltht_j}{L^2},
    \end{align}
for some $\{c_j\}\in\ell^2(\ZZ)$ such that $\Sob{\{c_j\}}{\ell^2}\leq1$ and where $\zeta \in [0,1)$, $\nu\in(0,1)$. On the other hand, we apply \cref{lem:commutator3} with $f=\bdy_l\bdy_\ell\tht$, $g=A_{\ell}{q}$, $h=\Lam^{\si_c+{\de}}\tiltht_j$, and { ${\si}=\si_c-2+\de$, $\rho=0$}, and Bernstein's inequality to arrive at the same bound for $J_{3}^{b}$
    \begin{align}\label{est:J3b}
        |J_3^b|&\leq C\left(c_j2^{\nu j}\Sob{A {q}}{\dot{G}^{\lam(t)}_{\al,2-\nu}}\Sob{\De\tht}{\dot{G}^{\lam(t)}_{\al,{\be-1-\kap+\de}}}+\lam(t)2^{(\si_c-1+\de+\al-\ze)j}\Sob{{E}_\al^{\lam(t)} S_{j-3}A{q}}{\dot{H}^{1+\ze}}\Sob{\De\til{\tht}_j}{L^2}\right)\Sob{\Lam^{\si_c+{\de}}\tiltht_j}{L^2},\notag\\
        &\leq Cc_j\left(2^{\nu j}\Sob{ {q}}{\dot{G}^{\lam(t)}_{\al,1+\be-\nu}}\Sob{\tht}{\dot{G}^{\lam(t)}_{\al,{\si_c+\de}}}+\lam(t)2^{(1+\al-\ze)j}\Sob{{E}_\al^{\lam(t)} S_{j-3}{q}}{\dot{H}^{\be+\ze}}\Sob{{\tht}}{\Gdot^{\lam(t)}_{\al,\si_c+\de}}\right)\Sob{\Lam^{\si_c+{\de}}\tiltht_j}{L^2},
    \end{align}
for some $\{c_j\}\in\ell^2(\ZZ)$ such that $\Sob{\{c_j\}}{\ell^2}\leq1$ and where $\zeta \in [0,1)$, $\nu\in(0,1)$.

\subsubsection*{{Bound for} ${J_5}$\nopunct}:	We recall that  
	\begin{align}
	J_5=&- \left \{ \left  \langle {G}_\al^{\lam(t)}\Lam^{\si_c+{\de}}_j\Lam^{\be-2}(\nabla^{\perp}{q} \cdot \nabla \theta),\Lam^{\si_c+{\de}}\tiltht_j  \right \rangle
	- \left \langle \nabla^{\perp}\Lambda^{\be-2}\cdotp((G_\al^{\lam(t)}\Lam_j^{\si_c+{\de}}{{q}}) \nabla \tht),\Lam^{\si_c+{\de}} \tiltht_j \right \rangle \right.\notag \\ 
	&\left.\quad -\left \langle \Lam^{\frac{\be-2}2}(\nabla^{\perp}{q}\cdot {\nabla} G_\al^{\lam(t)}\Lam_j^{\si_c+{\de}}\Lam^{\frac{\be-2}{2}} \tht), \Lam^{\si_c+{\de}}\tiltht_j \right \rangle \right\}.\notag
	\end{align}
	Similar to $J_{3}$, we can re-write $J_{5}$ as a double commutator.  
	By applying \eqref{split:Lam:rewrite} and the product rule, and using the notation $\bdy_{1}^{\perp}=-\bdy_{2}$ and $\bdy_{2}^{\perp}=\bdy_{1}$, we have
	\begin{align*}
	\begin{split}
	J_{5}=&\left \langle {G}_\al^{\lam(t)} \Lam_j^{{\si_c+{\de}}-2}\bdy_l({\nabla}^{\perp}(\bdy_{l}{q})\cdotp{\nabla}\theta),\Lam^{\si_c+{\de}+\be-2} \tiltht_j  \right \rangle
	- \left \langle  ({G}_\al^{\lam(t)}\Lam_j^{{\si_c+{\de}}-2}{\nabla}^{\perp}\De{q})\cdotp{\nabla}\tht,  \Lam^{\si_c+{\de}+\be-2}\tiltht_j \right \rangle \\
    	&+\left \langle {G}_\al^{\lam(t)} \Lam_j^{{\si_c+{\de}}}\Lam^{\be/2-3}\bdy_l(({\nabla}^{\perp}{q}\cdotp{\nabla}(\bdy_{l} \theta)),\Lam^{\si_c+{\de}+\be/2-1} \tiltht_j  \right \rangle
    	-\left \langle ({\nabla}^{\perp}{q}\cdotp{\nabla}) {G}_\al^{\lam(t)}\Lam_j^{{\si_c+{\de}}+\be/2-3}\De \tht,\Lam^{\si_c+{\de}+\be/2-1} \tiltht_j \right \rangle  \notag\\
    	=&\left\lb [{G}_\al^{\lam(t)} \Lam_j^{{\si_c+{\de}}-2}\bdy_l,\bdy_\ell\tht]\bdy_{\ell}^\perp\bdy_{l}{q},\Lam^{\si_c+{\de}+\be-2}\tiltht_j\right\rb
    	+\left\lb [{G}_\al^{\lam(t)} \Lam_j^{{\si_c+{\de}}+\be/2-3}\bdy_l,\bdy^{\perp}_{\ell}{q}]\bdy_\ell\bdy_l\tht,\Lam^{\si_c+{\de}+\be/2-1}\tiltht_j\right\rb\notag\\
    	=&J_5^a+J_5^b.
	\end{split}
	\end{align*}
By \cref{lem:commutator3} with {${\si}=\si_c-2+{\de}$, $\rho=0$ and $f=\bdy^{\perp}_{\ell}\bdy_l{q}$, $g=\bdy_\ell\tht$, $h=\Lam^{\si_c+{\de}+\be-2}\tiltht_j$}, and Bernstein's inequality, we have
    \begin{align}\label{est:J5a}
         |J_5^a| \leq& C\left(c_j2^{\nu' j}\Sob{\De q}{\dot{G}^{\lam(t)}_{\al,1-\nu'}}\Sob{{\nabla}\tht}{\dot{G}^{\lam(t)}_{\al,{\si_c-1+\de}}}
      +\lam(t)2^{(\si_c-1+\de+\al-{\ze})j}\Sob{E_\al^{\lam(t)}S_{j-3}{\nabla}\tht}{\Hdot^{1+{\ze}}}\Sob{\De \til{q}_j}{L^2}\right)\Sob{\Lam^{\si_c+{\de}+\be-2}\tiltht_j}{L^2}\notag\\
    \leq& C2^{(\be-2)j}\left(c_j2^{\nu' j}\Sob{ q}{\dot{G}^{\lam(t)}_{\al,3-\nu'}}\Sob{\tht}{\dot{G}^{\lam(t)}_{\al,{\si_c+\de}}}
      + \lam(t)2^{(1+\al-{\ze})j}\Sob{E_\al^{\lam(t)}S_{j-3}\tht}{\Hdot^{2+{\ze}}}\Sob{ \til{q}_j}{\dot{H}^{\si_c+\de}}\right)\Sob{\Lam^{\si_c+{\de}}\tiltht_j}{L^2}\notag\\
     \leq& C2^{(\be-2)j}c_j\left(2^{\nu' j}\Sob{ q}{\dot{G}^{\lam(t)}_{\al,3-\nu'}}\Sob{\tht}{\dot{G}^{\lam(t)}_{\al,{\si_c+\de}}}
      + \lam(t)2^{(3-\be+\al-{\ze})j}\Sob{E_\al^{\lam(t)}S_{j-3}\tht}{\Hdot^{\be+{\ze}}}\Sob{ {q}}{\dot{G}^{\lam(t)}_{\si_c+\de}}\right)\Sob{\Lam^{\si_c+{\de}}\tiltht_j}{L^2},
    \end{align}
for some $\{c_j\}\in\ell^2(\ZZ)$ such that $\Sob{\{c_j\}}{\ell^2}\leq1$, where  $\zeta \in [0,1)$, $\nu'\in(0,1)$. Similarly, for { ${\si}=\si_c-2+\de$, $\rho=\be/2-1$ and $f=\bdy_\ell\bdy_l\tht$, $g=\bdy^{\perp}_{\ell}{q}$, $h=\Lam^{\si_c+{\de}+\be/2-1}\tiltht_j$}, we apply \cref{lem:commutator3} and Bernstein's inequality to obtain again
    \begin{align}\label{est:J5b}
        |J_5^b|\leq& C\left(c_j2^{\nu' j}\Sob{{\nabla}^\perp {q}}{\dot{G}^{\lam(t)}_{\al,2-\nu'}}\Sob{ \De\tht}{\dot{G}^{\lam(t)}_{\al,{\si_c-2+\de}}}+ \lam(t)2^{(\si_c-1+\de+\al-{\ze})j}\Sob{{E}_\al^{\lam(t)} S_{j-3}{\nabla}^\perp{q}}{\dot{H}^{1+{\ze}}}\Sob{\De \til{\tht}_j}{L^2}\right)\Sob{\Lam^{\si_c+{\de}+\be-2}\tiltht_j}{L^2}\notag\\
        \leq& C2^{(\be-2)j}c_j\left(2^{\nu' j}\Sob{ {q}}{\dot{G}^{\lam(t)}_{\al,3-\nu'}}\Sob{ \tht}{\dot{G}^{\lam(t)}_{\al,{\si_c+\de}}}+ \lam(t)2^{(3-\be+\al-{\ze})j}\Sob{{E}_\al^{\lam(t)} S_{j-3}{q}}{\dot{H}^{\be+{\ze}}}\Sob{{\tht}}{\dot{G}^{\lam(t)}_{\si_c+\de}}\right)\Sob{\Lam^{\si_c+{\de}}\tiltht_j}{L^2},
    \end{align}
for some $\{c_j\}_j\in\ell^2(\ZZ)$ such that $\Sob{\{c_j\}_j}{\ell^2}\leq1$, where  $\zeta \in [0,1)$, $\nu'\in(0,1)$.

\subsubsection*{Summary of preparatory bounds.}

For each case, $\be<1+\kap$ and $\be\ge 1+\kap$, let us now summarize our estimates.

\subsubsection*{\textbf{Case: $\be<1+\kap$}}

Upon returning to \eqref{balance:gevrey:basic} and applying \eqref{est:I1}, we have
    \begin{align}\label{est:I:summary}
    	   \frac{1}{2}\frac{d}{dt}&\Sob{\Lam^{\si_c+{\de}}\til{\tht}_j}{L^2}^2+{\gam} \Sob{\Lam^{{\si_c+{\de}}+\kap/2}\til{\tht}_j}{L^2}^2
    	   \leq\lam'(t)\Sob{\Lam^{{\si_c+{\de}}+\al/2}\til{\tht}_j}{L^2}^2\notag\\
    	    &+ Cc_j\left(2^{\nu j}\Sob{q}{\dot{G}_{\al,1+\be-\nu}^{\lam(t)}}\Sob{\tht}{\dot{G}^{\lam(t)}_{\al,{\si_c+{\de}}}}+ \lam(t)2^{(1+\al-\ze)j}\Sob{{E}_\al^{\lam(t)} S_{j-3}q}{\dot{H}^{\be+\ze}}\Sob{\tht}{\Gdot^{\lam(t)}_{\si_c+\de}}\right)\Sob{\Lam^{\si_c+{\de}}\tiltht_j}{L^2},
\end{align}
where $\ze\in[0,1)$, $\nu\in(0,1)$, and $\{c_j\}\in\ell^2(\ZZ)$ with $\Sob{\{c_j\}}{\ell^2}\leq1$.

\subsubsection*{\textbf{Case: $\be\ge1+\kap$}}

Applying \eqref{est:J1} and \eqref{est:J3a}-\eqref{est:J5b}, we have
    \begin{align}\label{est:J:summary}
    	   \frac{1}{2}&\frac{d}{dt}\Sob{\Lam^{\si_c+{\de}}\til{\tht}_j}{L^2}^2+{\gam} \Sob{\Lam^{{\si_c+{\de}}+\kap/2}\til{\tht}_j}{L^2}^2
    	   \leq\lam'(t)\Sob{\Lam^{{\si_c+{\de}}+\al/2}\til{\tht}_j}{L^2}^2\notag\\
    	    &+  Cc_j\left(2^{\nu j}\Sob{ {q}}{\dot{G}^{\lam(t)}_{\al,1+\be-\nu}}\Sob{\tht}{\dot{G}^{\lam(t)}_{\al,{\si_c+\de}}}+\lam(t)2^{(1+\al-{\ze})j}\Sob{{E}_\al^{\lam(t)} S_{j-3}{q}}{\dot{H}^{\be+{\ze}}}\Sob{{\tht}}{\Gdot^{\lam(t)}_{\al,\si_c+\de}}\right)\Sob{\Lam^{\si_c+{\de}}\tiltht_j}{L^2}\\
    	    &+C2^{(\be-2)j}c_j\left(2^{\nu' j}\Sob{ {q}}{\dot{G}^{\lam(t)}_{\al,3-\nu'}}\Sob{ \tht}{\dot{G}^{\lam(t)}_{\al,{\si_c+\de}}}+ \lam(t)2^{(3-\be+\al-{\ze})j}\Sob{{E}_\al^{\lam(t)} S_{j-3}{q}}{\dot{H}^{\be+{\ze}}}\Sob{{\tht}}{\dot{G}^{\lam(t)}_{\si_c+\de}}\right)\Sob{\Lam^{\si_c+{\de}}\tiltht_j}{L^2},\notag
\end{align}
where ${\ze}\in[0,1)$, $\nu,\nu'\in(0,1)$, and $\{c_j\}\in\ell^2(\ZZ)$ with $\Sob{\{c_j\}}{\ell^2}\leq1$.

\subsubsection{Sobolev space estimates.} We will now specialize the estimates from the previous section to the Sobolev setting by simply taking $\lam(t)\equiv0$. Upon particular choices of the parameters $\de,\nu,\nu'$, we will derive estimates in $L^\infty_T\dot{H}^{\si_c}\cap L^2_T\dot{H}^{\si_c+\frac{\kap}{2}}$. These estimates will ultimately be leveraged to establish existence and uniqueness. To this end, let us choose the parameter $\de=\de(\kap,\be)$ to be defined by
\begin{align}\label{defn:delta:sobolev}
\de=\begin{cases}
\frac{\kap+1-\be}{2},\quad &\text{if $\be<1+\kap$} \\
\frac{\kap}{3},\quad &\text{if $\be \ge 1+\kap$}.
\end{cases}
\end{align}
Observe that $0<\de<\kap/2$.

\subsubsection*{Intermediary $L^{\frac{\kap}\de}_T\dot{H}^{\si_c+\de}$--estimates.}
In order to close estimates in $L^\infty_T\dot{H}^{\si_c}$ and $L^2_T\dot{H}^{\si_c+\frac{\kap}{2}}$, we will first derive an intermediate set of estimates in $L^{\frac{\kap}\de}_T\dot{H}^{\si_c+\de}$.
 Let us choose     \begin{align*}
        \nu=\kap-\de,\quad \nu'=2-\be+\kap-\de.
    \end{align*}
With $\lam\equiv0$ and these choices for $\de,\nu,\nu'$, upon returning to \eqref{est:I:summary} and \eqref{est:J:summary}, then applying the Bernstein inequalities, we derive
       \begin{align*}
    	   \frac{1}{2}\frac{d}{dt}&\Sob{\Lam^{\si_c+{\de}}{\tht}_j}{L^2}^2+c'2^{\kap j}{\gam} \Sob{\Lam^{{\si_c+{\de}}}{\tht}_j}{L^2}^2
    	   \leq
    	    Cc_j2^{(\kap-\de) j}\Sob{q}{\Hdot^{\si_c+\de}}\Sob{\tht}{\Hdot^{{\si_c+{\de}}}}\Sob{\Lam^{\si_c+{\de}}\tht_j}{L^2}.
    \end{align*}
Upon dividing both sides by $\Sob{\Lam^{\si_c+{\de}}_j\tht}{L^2}$, then integrating in time, we obtain
\begin{align*}
    \Sob{\Lam^{\si_c+{\de}}_j\tht(t)}{L^2}&\le e^{-c' 2^{\kap j}\gam t}\Sob{\Lam^{\si_c+{\de}}_j\tht(0)}{L^2}
    +Cc_j2^{(\kap-\de) j}\int_{0}^{t}e^{-c'2^{\kap j}\gam (t-s)}\Sob{ {q}(s)}{\dot{H}^{\si_{c}+\de }}\Sob{\tht(s)}{\dot{H}^{{\si_c+{\de}}}}\,ds.
\end{align*}
Observe that
\begin{align}\label{E:elem1:HLS}
    \sup_{j}2^{(\kap-\de)j}e^{-c'2^{\kap j}\gam (t-s)}\leq C(\gam(t-s))^{-1+\de/\kap}.
\end{align}
With this in hand, we take the $\ell_{2}$--norm in $j$, followed by the $L^{\kap/\de}$--norm in time to obtain
\begin{align}\label{est:intermediate}
    \Sob{\tht}{L^{\frac{\kap}{\de}}_{T}\Hdot_x^{\si_{c}+\de}}\le \mathscr{T}_{1}(T)+C\mathscr{T}_{2}(T),
\end{align}
where 
\begin{align*}
    \mathscr{T}_{1}(T)=\nrm{\left(\sum_{j\in\mathbb{Z}}e^{-c'2^{\kap j+1}\gam t}\Sob{\Lam^{\si_c+\de}_j\tht(0)}{L^2}^{2}\right)^{\frac{1}{2}}}_{L^{\frac{\kap}{\de}}_{T}},
\end{align*}
\begin{align*}
    \mathscr{T}_{2}(T)&=\nrm{\int_{0}^{t}(\gam(t-s))^{-1+\de/\kap}\Sob{{q}(s)}{\Hdot^{\si_{c}+\de}}\Sob{\tht(s)}{\Hdot^{\si_{c}+\de}}\,ds}_{L^{\frac{\kap}{\de}}_{T}}.
\end{align*}
We treat $\mathscr{T}_1(T)$ by applying Minkowski's inequality and the Lebesgue dominated convergence theorem to deduce
\begin{align}\label{est:T1}
     \mathscr{T}_{1}(T)\le C\gam^{-\de/\kap}\nrm{\tht_{0}}_{\Hdot^{\si_c}}, \quad \lim_{T \to 0}\mathscr{T}_{1}(T)=0.
\end{align}
We treat $\mathscr{T}_2(T)$ by applying the Hardy-Littlewood-Sobolev's inequality followed by the Cauchy-Schwarz inequality, we obtain
\begin{align}\label{est:T2}
    \mathscr{T}_{2}(T)&\le C\gam^{-1+\de/\kap} \left\lVert\Sob{{q}(\cdotp)}{\Hdot^{\si_{c}+\de}}\Sob{\tht(\cdotp)}{\Hdot^{\si_{c}+\de}}\right\rVert_{L^{\frac{\kap}{2\de}}_{T}}\notag\\
    &\le C\gam^{-1+\de/\kap}\Sob{{q}}{L^{\frac{\kap}{\de}}_{T}\Hdot^{\si_{c}+\de}}\Sob{\tht}{L^{\frac{\kap}{\de}}_{
    T}\Hdot^{\si_{c}+\de}}.
\end{align}
Upon returning to \eqref{est:intermediate} and applying \eqref{est:T1} and \eqref{est:T2}, we obtain
    \begin{align}\label{est:intermediate:final}
        \Sob{\tht}{L^\frac{\kap}\de_T\dot{H}^{\si_c+\de}}\leq C\gam^{-\de/\kap}\Sob{\tht_0}{\dot{H}^{\si_c}}+C\gam^{-1+\de/\kap}\Sob{{q}}{L^{\frac{\kap}{\de}}_{T}\Hdot^{\si_{c}+\de}}\Sob{\tht}{L^{\frac{\kap}{\de}}_{
    T}\Hdot^{\si_{c}+\de}}.
    \end{align}

\subsubsection*{$L^2_T\dot{H}^{\si_c+\kap/2}$--estimates} For this case, we choose     
    \begin{align*}
        \nu=\frac{\kap}2,\quad \nu'=2-\be+\frac{\kap}2.
    \end{align*}
Referring back to \eqref{est:I:summary} and \eqref{est:J:summary} with this choice, we obtain
\begin{align*}
	   { \frac{1}{2}\frac{d}{dt}\Sob{\Lam^{\si_c+\de}_j{\tht}}{L^2}^2}+{{c'\gam}2^{\kap j} \Sob{\Lam^{\si_c+\de}_j{\tht}}{L^2}^2} \leq c_j2^{\frac{\kap}{2} j}\Sob{ {q}}{\dot{H}^{\si_{c}+\frac{\kap}{2}}}\Sob{\tht}{\dot{H}^{{\si_c+\de}}}\Sob{\Lam^{\si_c+\de}_j\tht}{L^2}.
	\end{align*}
We divide both sides by $2^{({\de}-\frac{\kap}{2})j}\Sob{\Lam^{\si_c+\de}_j\tht}{L^2}$, use Bernstein's inequality, then integrate in time to obtain
\begin{align}\notag
    \Sob{\Lam^{\si_c+{\frac{\kap}{2}}}_j\tht(t)}{L^2}&\le e^{-c'2^{\kap j}\gam t}\Sob{\Lam^{\si_c+{\frac{\kap}{2}}}_j\tht(0)}{L^2}
    +C2^{(\kap-{\de}) j}\int_{0}^{t}e^{-c'2^{\kap j}\gam (t-s)}c_{j}\Sob{ {q}(s)}{\dot{H}^{\si_{c}+\frac{\kap}{2}}}\Sob{\tht(s)}{\dot{H}^{{\si_c+{\de}}}}ds.
\end{align}
Using \eqref{E:elem1:HLS}, and taking the $\ell_{2}$-norm in $j$, followed by the $L^{2}$-norm in time, we have
\begin{align}\label{est:L2:endpoint}
    \Sob{\tht}{L^{2}_{T}\Hdot^{\si_{c}+\frac{\kap}{2}}}\le \mathscr{S}_{1}(T)+C\mathscr{S}_{2}(T),
\end{align}
where 
\begin{align*}
    \mathscr{S}_{1}(T)=\nrm{\left(\sum_{j\in\mathbb{Z}}e^{-c'2^{\kap j+1}\gam t}\Sob{\Lam^{\si_c+{\frac{\kap}{2}}}_j\tht(0)}{L^2}^{2}\right)^{\frac{1}{2}}}_{L^{2}_{T}},
\end{align*}
\begin{align*}
    \mathscr{S}_{2}(T)&=\nrm{\int_{0}^{t}(\gam(t-s))^{-1+\frac{\de}{\kap}}\Sob{{q}(s)}{\Hdot^{\si_{c}+\frac{\kap}{2}}}\Sob{\tht(s)}{\Hdot^{\si_{c}+\de}}\,ds}_{L^{2}_{T}}.
\end{align*}
By direct calculation and an application of the Lebesgue dominated convergence theorem, we have
\begin{align}\label{est:S1}
     \mathscr{S}_{1}(T)\le C\gam^{-1/2}\nrm{\tht_{0}}_{\Hdot^{\si_c}}, \quad \lim_{T \to 0}\mathscr{S}_{1}(T)=0.
\end{align}
Applying the Hardy-Littlewood-Sobolev inequality followed by Holder's inequality, we obtain
\begin{align}\label{est:S2}
    \mathscr{S}_{2}(T) &\le C\gam^{-1+\de/\kap} \nrm{\Sob{{q}(s)}{\Hdot^{\si_{c}+\frac{\kap}{2}}}\Sob{\tht(s)}{\Hdot^{\si_{c}+\de}}}_{L^{\frac{2\kap}{\kap+2\de}}_{T}}\notag\\
    &\le C\gam^{-1+\de/\kap} \nrm{{q}}_{L^{2}_{T}\Hdot^{\si_{c}+\frac{\kap}{2}}}\nrm{\tht}_{L^{\frac{\kap}{\de}}_{T}\Hdot^{\si_{c}+\de}}.
\end{align}
Upon returning to \eqref{est:L2:endpoint} and applying \eqref{est:S1} and \eqref{est:S2}, we obtain 
    \begin{align}\label{est:L2T:final}
        \Sob{\tht}{L^{2}_{T}\Hdot^{\si_{c}+\frac{\kap}{2}}}
        &\le C\gam^{-1/2}\nrm{\tht_{0}}_{\Hdot^{\si_c}}+C\gam^{-1+\de/\kap}\nrm{{q}}_{L^{2}_{T}\Hdot^{\si_{c}+\frac{\kap}{2}}}\nrm{\tht}_{L^{\frac{\kap}{\de}}_{T}\Hdot^{\si_{c}+\de}}.
    \end{align}

\subsubsection*{$L^\infty_T\dot{H}^{\si_c}$--estimates}
Finally, we obtain an estimate of $\nrm{\tht}_{L^{\infty}_{T}\Hdot^{\si_c}}$. For this, we must return to \eqref{est:I:summary} and \eqref{est:J:summary} and, instead of \eqref{defn:delta:sobolev}, we make the choice $\de=0$. We then choose
    \begin{align*}
        \nu=\frac{\kap}2,\quad \nu'=2-\be+\frac{\kap}2.
    \end{align*}
Then \eqref{est:I:summary} and \eqref{est:J:summary} become
\begin{align}
	   {\frac{1}{2}\frac{d}{dt}\Sob{\Lam^{\si_c}_j{\tht}}{L^2}^2}+{c'{\gam}2^{\kap j} \Sob{\Lam^{{\si_c}}_j{\tht}}{L^2}^2}&\leq c_j2^{\frac{\kap}{2} j}\Sob{ {q}}{\dot{H}^{\si_{c}+\frac{\kap}{2}}}\Sob{\tht}{\dot{H}^{{\si_c}}}\Sob{\Lam^{\si_c}_j\tht}{L^2}\notag\\
	   &\le Cc_j^{2}\Sob{ {q}}{\dot{H}^{\si_{c}+\frac{\kap}{2}}}^{2}\Sob{\tht}{\dot{H}^{{\si_c}}}^{2}+\frac{c'\gam}{2}2^{\kap j}\Sob{\Lam^{\si_c}_j\tht}{L^2}^{2},\notag
	\end{align}
where we used Young's inequality in the last step. We subtract the last term on the right,  sum in $j$ and integrate in time to obtain
\begin{align*}
    \nrm{\tht}_{L^{\infty}_{T}\Hdot^{\si_c}}^{2}\le\nrm{\tht_0}_{\Hdot^{\si_c}}^{2}\exp(C\nrm{{q}}_{L^{2}_{T}\Hdot^{\si_{c}+\frac{\kap}{2}}}^{2}). 
\end{align*}
Upon combining this with (\ref{est:L2:critical}), we deduce
\begin{align}
     \nrm{\tht}_{L^{\infty}_{T}H^{\si_c}}\le\nrm{\tht_0}_{H^{\si_c}}\exp(C\nrm{{q}}_{L^{2}_{T}\Hdot^{\si_{c}+\frac{\kap}{2}}}^{2}). \label{est:Linfty:endpoint}
\end{align}

\subsubsection*{Summary of Sobolev space estimates}

Collecting the estimates \eqref{est:intermediate:final}, \eqref{est:L2T:final}, \eqref{est:Linfty:endpoint}, we arrive at
    \begin{align}\label{est:sobolev:summary}
        \begin{split}
       \Sob{\tht}{L^\frac{\kap}\de_T\dot{H}^{\si_c+\de}}&\leq C\gam^{-\de/\kap}\Sob{\tht_0}{\dot{H}^{\si_c}}+C\gam^{-1+\de/\kap}\Sob{{q}}{L^{\frac{\kap}{\de}}_{T}\Hdot^{\si_{c}+\de}}\Sob{\tht}{L^{\frac{\kap}{\de}}_{
    T}\Hdot^{\si_{c}+\de}},\\
     \Sob{\tht}{L^{2}_{T}\Hdot^{\si_{c}+\frac{\kap}{2}}}
        &\le C\gam^{-1/2}\nrm{\tht_{0}}_{\Hdot^{\si_c}}+C\gam^{-1+\de/\kap}\nrm{{q}}_{L^{2}_{T}\Hdot^{\si_{c}+\frac{\kap}{2}}}\nrm{\tht}_{L^{\frac{\kap}{\de}}_{T}\Hdot^{\si_{c}+\de}},\\
        \nrm{\tht}_{L^{\infty}_{T}H^{\si_c}}&\le\nrm{\tht_0}_{H^{\si_c}}\exp(C\nrm{{q}}_{L^{2}_{T}\Hdot^{\si_{c}+\frac{\kap}{2}}}^{2}).
        \end{split}
    \end{align}

\subsubsection{Gevrey class estimates} In this section, we will obtain an apriori estimate for (\ref{E:mod:claw}) in Gevrey classes. First, for a given measurable function $\Phi:(0,T]\longrightarrow \Hdot^{\si_c+{\de}}$, we define
	\begin{align}\label{def:XT}
	&\nrm{\Phi(\cdot)}_{X_{T}}:= \esssup_{0<t\le T}(\gam t)^{\frac{{\de}}{\kap}} \nrm{\Phi(t)}_{\Gdot^{\lam(t)}_{\al, \si_c+{\de}}}
	\end{align}
    where $\dot{G}^\lam_{\al,\si}$ denotes the homogeneous $(\al,\lam,\si)$--Gevrey class defined in \eqref{def:Gev:class}. Let us also choose $\lam(t)$ to be
 \begin{align*}
   { \lam(t):=\veps\gam^{\al/\kap}t^{\al/\kap},\quad {\veps} >0.}
    \end{align*}
For convenience, we will often drop the dependence on $t$. Suppose that $\al$
satisfies
    \begin{align*}
       0<\al<\kap.
    \end{align*}
Let us assume
    \begin{align*}
   0<\de<\begin{cases}
\min\left\{\frac{\kap+1-\be}{2}, 2(\kap-\al),\al\right\},\quad &\text{if $\be<1+\kap$} \\
\min\left\{\frac{\kap}{2}, 2(\kap-\al),\al\right\},\quad &\text{if $\be \ge 1+\kap$}.
\end{cases}
    \end{align*}

Upon returning to \eqref{est:I:summary} and \eqref{est:J:summary}, we choose
    \begin{align*}
        \nu=\kap-\de,\quad\nu'=2-\be+\kap-\de,\quad \ze=1-\kap+\frac{\de}2+\al.
    \end{align*}
These choices yield
\begin{align*}
	 \frac{1}{2}\frac{d}{dt}&\Sob{\Lam^{\si_c+{\de}}\tiltht_j}{L^2}^2+{\gam} \Sob{\Lam^{\si_c+{\de}+\kap/2}\tiltht_j}{L^2}^2\notag\\
	 \leq& \frac{\al}{\kap}\veps \gam^{\al/\kap}t^{{\al}/{\kap}-1} \Sob{\Lam^{\si_c+{\de}+\al/2}\tiltht_j}{L^2}^2
	 +Cc_j2^{(\kap-\de) j}\Sob{{G}_{\al}^{\lam(t)} {q}}{\dot{H}^{\si_c+\de}}\Sob{{G}_{\al}^{\lam(t)} \tht}{\dot{H}^{{\si_c+{\de}}}}\Sob{\Lam^{\si_c+{\de}}\tiltht_j}{L^2}\\
	 &+C\veps \gam^{\al/\kap}t^{\al/\kap}2^{(\kap-{\de}/2)j}\Sob{{G}_{\al}^{\lam(t)}\lpj \tht}{\Hdot^{\si_c+\de}}\Sob{{E}_\al^{\lam(t)} S_{j-3}{q}}{\dot{H}^{\si_{c}+{\de}/2+\al}}\Sob{\Lam^{\si_c+{\de}}\tiltht_j}{L^2}\\
	 &+C\veps \gam^{\al/\kap}t^{\al/\kap}2^{(\kap-{\de}/2)j}\Sob{{G}_\al^{\lam(t)}\lpj {q}}{\Hdot^{\si_c+\de}}\Sob{{E}_\al^{\lam(t)} S_{j-3}\tht}{\dot{H}^{\si_{c}+{\de}/2+\al}}\Sob{\Lam^{\si_c+{\de}}\tiltht_j}{L^2}.
	\end{align*}
	
{Now observe that from Lemma \ref{lem:interpolate:gevrey},  applied with $\rho=1$, $s_1=\al/2$, $s_2=\kap/2$, $f=\Lam^{\si_{c}+\de}\til{\tht}_{j}$, it follows that}
        \begin{equation}\label{est:alpha:kappa}
          {\frac{\al}{\kap}\veps \gam^{\al/\kap}t^{{\al}/\kap-1}\Sob{\Lam^{\si_c+{\de}+\al/2}\til{\tht}_j}{L^2}^2\le C\veps \gam^{\al/\kap}t^{{\al}/\kap-1}\Sob{\Lam^{\si_c+{\de}}{\tht}_j }{\Hdot^{\al/2}}^2+C\veps^{\kap/\al} \gam\Sob{\Lam^{\si_c+{\de}}\til{\tht}_j}{\dot{H}^{\kap/2}}^2},
        \end{equation}
where $C$ depends on $\al$ and $\kap$. We will fix $\veps$ small enough such that
   \begin{align*}
    C\veps^{\kap/\al}\leq\frac{1}2.
    \end{align*}
On the other hand, recall that $E_\al^\lam$ is defined by \eqref{def:gev:avg}. Using the elementary estimate
\begin{align*}
    x^{a}e^{-yx^{b}}\le Cy^{-a/b},\quad a,x\ge0;\quad b,y>0,
\end{align*}
we first observe that for any function $f\in \dot{G}^{\lam(t)}_{\al,0}$, one has
\begin{align}\label{est:sobolev:de'}
  \nrm{\Lam^{{\al-{\de}/2}}{E}^{\lam(t)}_\al f}_{L^2}^2&=\int\left(\int_0^1\Ax^{\al-\de/2}e^{\lam(t)(\tau^\al-1)|\xi|^\al}d \tau\right)^2|\widehat{{G}^{\lam}_{\al}f}(\xi)|^2d\xi\notag\\
    &\le C\lam(t)^{-2(1-\frac{\de}{2\al})}\left(\int_{0}^{1}\frac{1}{(1-\tau^\al)^{1-\frac{\de}{2\al}}} d\tau \right)^2\nrm{{{G}^{\lam}_{\al}f}}_{L^2}^2\notag\\&\le C\veps^{-2(1-\frac{\de}{2\al})}\gam^{-2(\frac{\al}{\kap}-\frac{\de}{2\kap})}t^{-2(\frac{\al}{\kap}-\frac{\de}{2\kap})}\nrm{{G}^{\lam}_{\al}f}_{L^2}^2.
\end{align}
 
 Upon using the estimate in \eqref{est:alpha:kappa}, \eqref{est:sobolev:de'}  dividing both sides by $\Sob{\Lam^{\si_c+{\de}}\tiltht_j}{L^2}$ and applying Bernstein's inequality, we obtain
\begin{align*}
    &\frac{d}{dt}\Sob{\Lam^{\si_c+{\de}}\tiltht_j}{L^2}+{c'\gam}2^{\kap j} \Sob{\Lam^{\si_c+{\de}}\tiltht_j}{L^2}\notag\\
	 &\leq C\veps\gam^{\al/\kap}t^{{\al}/{\kap}-1} 2^{\al j}\Sob{\Lam^{\si_c+{\de}}\tht_j}{L^2}
	 + Cc_j\left(2^{(\kap-\de) j}+\veps^{\frac{\de}{2\al}}\gam^{\frac{\de}{2\kap}} t^{\frac{\de}{2\kap}}2^{(\kap-{\de}/2)j}\right)\Sob{{G}_\al^{\lam(t)} {q}}{\dot{H}^{\si_c+\de}}\Sob{{G}_\al^{\lam(t)} \tht}{\dot{H}^{{\si_c+{\de}}}},
\end{align*}
for some $\{c_j\}\in\ell^2(\ZZ)$ such that $\Sob{\{c_j\}}{\ell^2}\leq1$. By Gronwall's inequality and Bernstein's inequality, we obtain
\begin{align}\label{est:ET:apriori:intermediate}
    \Sob{\Lam^{\si_c+{\de}}\tiltht_j(t)}{L^2}
    \leq& Ce^{-c'2^{\kap j}\gam t}2^{\de j}\Sob{\lpj \tht(0)}{\dot{H}^{\si_c}}+Cc_j\veps\gam^{\al/\kap}\int_{0}^{t}2^{\al j}e^{-c'2^{\kap j}\gam(t-s)}s^{{\al}/{\kap}-1}\Sob{\tht(s)}{\Hdot^{\si_{c}+\de}}\,ds\notag\\
        &+Cc_j\int_{0}^{t}2^{(\kap-\de) j}e^{-c'2^{\kap j}{\gam}(t-s)}\Sob{{G}_\al^{\lam(s)} {q(s)}}{\dot{H}^{\si_c+\de}}\Sob{{G}_\al^{\lam(s)} \tht(s)}{\dot{H}^{{\si_c+{\de}}}}\,ds\notag\\
        & +Cc_j\veps^{\frac{\de}{2\al}}\gam^{\frac{\de}{2\kap}}\int_{0}^{t}2^{(\kap-{\de}/2) j}e^{-c'2^{\kap j}{\gam}(t-s)}s^{\frac{\de}{2\kap}}\Sob{{G}_\al^{\lam(t)} {q(s)}}{\dot{H}^{\si_c+\de}}\Sob{{G}_\al^{\lam(s)} \tht(s)}{\dot{H}^{{\si_c+{\de}}}}\,ds.
\end{align}
We have
\begin{align*}
    \gam^{\al/\kap}\int_{0}^{t}\left(2^{\al j}e^{-c'2^{\kap j}\gam(t-s)}\right)s^{{\al}/{\kap}-1}\Sob{\tht(s)}{\Hdot^{\si_{c}+\de}}\,ds &\le \gam^{\al/\kap}\int_{0}^{t}\left(\frac{C}{(\gam(t-s))^{\frac{\al}{\kap}}}\right)s^{{\al}/{\kap}-1}\Sob{\tht(s)}{\Hdot^{\si_{c}+\de}}\,ds\\
    &\le C\int_{0}^{1}\left(\frac{s^{\frac{\al}{\kap}-1}}{(t-s)^{\frac{\al}{\kap}}(\gam s)^{\frac{\de}{\kap}}}\right)(\gam s)^{\frac{\de}{\kap}}\Sob{\tht(s)}{\Hdot^{\si_{c}+\de}}\,ds\\
    &\le C(\gam t)^{-\frac{\de}{\kap}}B\left(1-\frac{\al}{\kap},\frac{\al-\de}{\kap}\right)\Sob{\tht(\cdotp)}{X_T},
\end{align*}
where $B(a,b)=\int_0^1(1-x)^{a-1}x^{b-1}dx$. Similarly, we estimate the last two terms in \eqref{est:ET:apriori:intermediate} to obtain
\begin{align*}
    \Sob{\Lam^{\si_c+{\de}}\tiltht_j(t)}{L^2} \leq& C(\gam t)^{-{\de}/\kap}\left((\gam t)^{{\de}/\kap}2^{\de j}e^{-c'2^{\kap j}\gam t}\right)\Sob{\lpj \tht(0)}{\dot{H}^{\si_c}}+ Cc_{j}\veps(\gam t)^{-\de/\kap} B\left(1-\frac{\al}{\kap},\frac{\al-\de}{\kap}\right)\Sob{\tht(\cdotp)}{X_T}
    \\
         &+Cc_{j}(\gam t)^{-\de/\kap}\gam^{-1}\left(B\left(\frac{\de}{\kap},1-\frac{2\de}{\kap}\right)+\veps^{\frac{\de}{2\al}}B\left(\frac{\de}{2\kap},1-\frac{3\de}{2\kap}\right)\right)\Sob{q(\cdotp)}{X_T}\Sob{\tht(\cdotp)}{X_T}.
\end{align*}
 We multiply both sides by $(\gam t)^{{\de}/\kap}$, take the $\ell^2(\ZZ)$-norm on, apply the Minkowski inequality, and take the supremum over $0<t\leq T$, to obtain
    \begin{align*}
       \Sob{\tht(\cdotp)}{X_T}\leq C\mathcal{I}_T(\tht_{0})+C\veps\Sob{\tht(\cdotp)}{X_T}+C\gam^{-1}\Sob{{q}(\cdotp)}{X_T}\Sob{\tht(\cdotp)}{X_T},
    \end{align*}
where
\begin{align*}
    \mathcal{I}_T(\tht_{0}):=&\sup_{0<t\le T}\left(\sum_{j\in\ZZ}(\gam t)^{\frac{2\de}{\kap}}2^{2\de j}e^{-c'2^{\kap j+1}\gam t}\Sob{\lpj\tht_{0}}{\Hdot^{\si_c}}^2\right)^{1/2}
    \le C\left(\sum_{j\in\ZZ}\Sob{\lpj\tht_{0}}{\Hdot^{\si_c}}^2\right)^{1/2} \le C\Sob{\tht_{0}}{\Hdot^{\si_c}}.
\end{align*}
Applying the Lebesgue dominated convergence theorem, we obtain
\begin{align}\label{lim:XT:initialdata}
    \lim_{T\to 0} \mathcal{I}_T(\tht_{0})=0.
    \end{align}
    
Upon possibly taking $\veps$ smaller so that
\[C\veps\leq 1/2,\]
we obtain the following apriori estimate for $\Sob{\tht(\cdot)}{X_T}$:
\begin{align}\label{est:ET:apriori}
       { \Sob{\tht(\cdotp)}{X_T}\leq C\mathcal{I}_T(\tht_{0})+C\gam^{-1}\Sob{{q}(\cdotp)}{X_T}\Sob{\tht(\cdotp)}{X_T}}.
    \end{align}

\section{Existence, uniqueness, and smoothing for $1<\beta<2$: Proof of \cref{thm:main:beta}}\label{sect:proofs:main:beta}

We will now carry out the proof of \cref{thm:main:beta}. We do so by introducing a sequence of approximating equations that will satisfy the dissipative perturbation of the conservation law \eqref{E:mod:claw} at each level of the approximation. In particular, we consider the following approximation scheme. 
\begin{align*}
    \begin{cases}
    &\partial_{t}\tht^{0}+\gam \Lam^{\kap}\tht^{0}=0,\quad (x,t)\in \mathbb{R}^{2}\times \mathbb{R}^{+},\\
    &\tht^{0}(x,0)=\tht_{0}(x)
    \end{cases}
\end{align*}
and
\begin{equation}\label{E:approximation}
    \begin{cases}
    &\partial_{t}\theta^{n+1}+ \Div F_{-\tht^{n}}( \theta^{n+1})=-{\gam} \Lam^{\kap}\tht^{n+1},\quad (x,t)\in \mathbb{R}^{2}\times \mathbb{R}^{+},\\
    & \tht^{n+1}(x,0)=\tht_{0}(x),\quad n\in \mathbb{Z}_{\{\ge 0\}}
    \end{cases}
\end{equation}
where 
\begin{align*}
F_{-\tht^{n}}(\tht^{n+1})=\begin{cases} 
 -(\nabla^{\perp}\Lam^{\be-2}\tht^{n}) \theta^{n+1} &\text{if $\be<1+\kap$},\\
 -(\nabla^{\perp}\Lam^{\be-2}\tht^{n}) \theta^{n+1}-\Lam^{\be-2}(({\nabla}^\perp\tht^{n+1})\tht^{n})\quad &\text{if $\be\geq 1+\kap$}.
\end{cases}
\end{align*}

\subsection{Existence}
First, we establish uniform (in $n$) estimates on $\tht^{n+1}$. Invoking apriori estimates in (\ref{est:T1}), (\ref{est:T2}), (\ref{est:S1}) and (\ref{est:S2}) for (\ref{E:approximation}), we conclude that there exists a bounded function $\mathscr{R}(T)\, (:=\mathscr{T}_{1}(T)+\mathscr{S}_{1}(T))$ such that
\begin{align*}
    \mathscr{R}(T)&\le C(\gam^{-1/2}+\gam^{-\de/\kap})\Sob{\tht_{0}}{\Hdot^{\si_c}}, \quad \lim_{T \to 0}\mathscr{R}(T)=0
    \end{align*}
and
    \begin{align*}
    \Sob{\tht^{0}}{L^{2}_{T}\Hdot^{\si_{c}+\frac{\kap}{2}}\cap L^{\frac{\kap}{\de}}_{T}\Hdot^{\si_{c}+\de}}&\le \mathscr{R}(T),\notag\\
    \Sob{\tht^{n+1}}{L^{2}_{T}\Hdot^{\si_{c}+\frac{\kap}{2}}\cap L^{\frac{\kap}{\de}}_{T}\Hdot^{\si_{c}+\de}}&\le\mathscr{R}(T)+C_{1}\Sob{\tht^{n}}{L^{2}_{T}\Hdot^{\si_{c}+\frac{\kap}{2}}\cap L^{\frac{\kap}{\de}}_{T}\Hdot^{\si_{c}+\de}}\Sob{\tht^{n+1}}{L^{2}_{T}\Hdot^{\si_{c}+\frac{\kap}{2}}\cap L^{\frac{\kap}{\de}}_{T}\Hdot^{\si_{c}+\de}}.
\end{align*}
Here $\de$ is given by (\ref{defn:delta:sobolev}). Let $T_{0}$ be chosen small such that $\mathscr{R}(T)\le 1/(4C_{1})$ for $T\in (0,T_0)$. Note that this condition also holds if $\Sob{\tht_0}{\Hdot^{\si_c}}$ is small enough and $T_0=\infty$. We obtain
\begin{align}\label{est:recursive:estimates3}
     \Sob{\tht^{n}}{L^{2}_{T}\Hdot^{\si_{c}+\frac{\kap}{2}}\cap L^{\frac{\kap}{\de}}_{T}\Hdot^{\si_{c}+\de}}&\le 2\mathscr{R}(T), \quad n=0,1,2,...
\end{align}
Applying \cref{T:transport} recursively and using (\ref{est:recursive:estimates3}), we obtain a unique solution of (\ref{E:approximation}) satisfying, for $T\in (0,T_{0})$,
\begin{align}\label{est:uniform:L2}
    &\Sob{\tht^{n+1}}{L^{2}_{T}\Hdot^{\si_{c}+\frac{\kap}{2}}\cap L^{\frac{\kap}{\de}}_{T}\Hdot^{\si_{c}+\de} }\le 2\mathscr{R}(T),\notag\\
    &\Sob{\tht^{n+1}}{L^{\infty}_{T}H^{\si_c}}\le\Sob{\tht_{0}}{H^{\si_c}}\exp(C_{2}\mathscr{R}(T)^{2}).
\end{align}

Next, we establish that the sequence of solutions $\{\tht^{n}\}$ converges to a solution of (\ref{E:dissipative-beta}). Let us denote by $\bar{\tht}^{n+1}=\tht^{n+1}-\tht^{n}$ and $\bar{\tht}^{0}=\tht^{0}$. Then, we can see that the differences $\bar{\tht}^{n+1}$ satisfy the following equation:
\begin{align}\label{E:diff}
    \begin{cases}
    &\partial_{t}\bar{\theta}^{n+1}+\Div F_{-\tht^{n}}(\bar{\tht}^{n+1})+\Div F_{-\bar{\tht}^{n}}(\tht^{n})=-{\gam} \Lam^{\kap}\bar{\tht}^{n+1},\\
    & \bar{\tht}^{n+1}(x,0)=0,\quad n\in \mathbb{Z}_{\{\ge 0\}}
    \end{cases}
\end{align}

We divide the analysis into two cases: ${\beta<1+\kappa}$ and ${\beta\ge1+\kappa. }$
\subsubsection*{\textbf{Case: $\beta<1+\kappa$} }
{Applying $\triangle_j$ to \eqref{E:diff} and then taking the inner product in $L^2$ with $\triangle_{j}\bar{\tht}^{n+1}$ yields}
	\begin{align}\label{balance:diff2}
	   { \frac{1}{2}\frac{d}{dt}\Sob{\triangle_j\bar{\tht}^{n+1}}{L^2}^2+}{{\gam} \Sob{\Lam^{\kap/2}_j\bar{\tht}^{n+1}}{L^2}^2}=&-\lb\triangle_{j}(u^{n} \cdot \nabla \bar{ \theta}^{n+1}),\triangle_j\bar{ \theta}^{n+1}\rb-\lb \triangle_{j}(\bar{u}^{n} \cdot \nabla  \theta^{n}),\triangle_j\bar{ \theta}^{n+1}\rb\notag\\
	   =&J_{1}'+J_{2}'.
	   \end{align}
Since $\nabla \cdot u^{n}=0$, we have
\begin{align*}
    \lb(u^{n} \cdot \nabla)\triangle_j \bar{ \theta}^{n+1},\triangle_j\bar{ \theta}^{n+1}\rb=0.
\end{align*}
As a result, we can express $J_1'$ in terms of a commutator.
\begin{align*}
    J_{1}'&=-\lb[\triangle_{j},(u^{n})^\ell]\bdy_\ell\bar{\tht}^{n+1} ,\triangle_j\bar{ \theta}^{n+1}\rb.
\end{align*}
By applying \cref{lem:commutator2a} with $\rho_1=2-2\kap/3$, $\rho_2=1-\kap/2$, $g=(u^{n})^\ell$, $f=\bdy_{\ell}\bar{\tht}^{n+1}$ and $h=\triangle_j\bar{ \theta}^{n+1}$, we have
\begin{align}\label{est:J1'}
    |J_{1}'|\le b_{j}2^{-\frac{\kap}{6} j}\Sob{\tht^{n}}{\Hdot^{\si_{c}+ \kap/2}}\Sob{\bar{\tht}^{n+1}}{\Hdot^{\frac{2\kap}{3}}}\Sob{\triangle_j\bar{\tht}^{n+1}}{L^2},
\end{align}
for some $\{b_j\}\in\ell^2(\ZZ)$. Upon using Plancherel's theorem, and applying \cref{lem:commutator1} with $\eps=\be-2\kap/3$, $\si=\si_{c}+\de-1$, $g=\bar{u}^{n}_{\ell}$, $f=\bdy_{\ell}  \theta^{n}$ and $h=\triangle_j\bar{ \theta}^{n+1}$, we obtain
\begin{align}\label{est:J2'}
    |J_{2}'|
    &\le  b_{j}2^{(-\si+\epsilon) j} \Sob{\bdy_{\ell}  \theta^{n}}{\Hdot^{\si_{c}+\de-1}}\Sob{\bar{u}^{n}_{\ell}}{\Hdot^{1+\frac{2\kap}{3}-\be}}\Sob{\triangle_j\bar{ \theta}^{n+1}}{L^2} \notag\\
    &\le b_{j}2^{(\frac{\kap}{3}-\de) j} \Sob{\tht^{n}}{\Hdot^{\si_{c}+ \de}}\Sob{\bar{\tht}^{n}}{\Hdot^{\frac{2\kap}{3}}}\Sob{\triangle_j\bar{\tht}^{n+1}}{L^2},
\end{align}
for some $\{b_j\}\in\ell^2(\ZZ)$.

We use the estimates (\ref{est:J1'}) and (\ref{est:J2'}) in (\ref{balance:diff2}), divide both sides by $\Sob{\triangle_j\bar{\tht}^{n+1}}{L^2}$ and apply Bernstein's inequality to obtain
\begin{align*}
    {\frac{d}{dt}\Sob{\triangle_j\bar{\tht}^{n+1}}{L^2}+}{c{\gam}2^{\kap j} \Sob{\triangle_j\bar{\tht}^{n+1}}{L^2}}\le b_{j}2^{-\frac{\kap}{6} j}\Sob{\tht^{n}}{\Hdot^{\si_{c}+ \frac{\kap}{2}}}\Sob{\bar{\tht}^{n+1}}{\Hdot^{\frac{2\kap}{3}}}+b_{j}2^{(\frac{\kap}{3}-\delta) j} \Sob{\tht^{n}}{\Hdot^{\si_{c}+ \de}}\Sob{\bar{\tht}^{n}}{\Hdot^{\frac{2\kap}{3}}}.
\end{align*}
Multiplying both sides by $2^{\frac{2\kap}{3}j}$ and integrating in time, we have
\begin{align*}
     \Sob{\Lam^{\frac{2\kap}{3}}_j\bar{\tht}^{n+1}}{L^2}\le&\int_0^{t}b_{j}2^{\frac{\kap}{2} j}e^{-c\gam2^{\kap j}(t-s)}\Sob{\tht^{n}(s)}{\Hdot^{\si_{c}+ \frac{\kap}{2}}}\Sob{\bar{\tht}^{n+1}(s)}{\Hdot^{\frac{2\kap}{3}}}\,ds\\
     &+\int_0^{t}b_{j}2^{(\kap-\de) j}e^{-c\gam2^{\kap j}(t-s)} \Sob{\tht^{n}(s)}{\Hdot^{\si_{c}+\de}}\Sob{\bar{\tht}^{n}(s)}{\Hdot^{\frac{2\kap}{3}}}\,ds.
\end{align*}
 Using (\ref{E:elem1:HLS}), and taking the $l_{2}$-norm in $j$ and then {$L^{3}$-norm} in time, obtain
\begin{align*}
        \Sob{\bar{\tht}^{n+1}}{L^{3}_{T}\Hdot^{\frac{2\kap}{3}}}\le C(I+II),
\end{align*}
where
\begin{align*}
    I(T)&=\nrm{\int_0^{t}(\gam(t-s))^{-1/2}\Sob{\tht^{n}(s)}{\Hdot^{\si_{c}+ \frac{\kap}{2}}}\Sob{\bar{\tht}^{n+1}(s)}{\Hdot^{\frac{2\kap}{3}}}\,ds}_{L^{3}_{T}},\\
    II(T)&=\nrm{\int_0^{t}(\gam(t-s))^{-1+\de/\kap} \Sob{\tht^{n}(s)}{\Hdot^{\si_{c}+ \de}}\Sob{\bar{\tht}^{n}(s)}{\Hdot^{\frac{2\kap}{3}}}\,ds}_{L^{3}_T}.
\end{align*}
We treat $I$ and $II$ by applying the Hardy-Littlewood-Sobolev inequality followed by Holder's inequality to obtain
\begin{align*}
    I(T)&\le C\Sob{\tht^{n}}{L^{2}_{T}\Hdot^{\si_{c}+\frac{\kap}{2}}}    \Sob{\bar{\tht}^{n+1}}{L^{3}_{T}\Hdot^{\frac{2\kap}{3}}},\\
    II(T)&\le C\Sob{\tht^{n}}{L^{\frac{\kap}{\de}}_{T}\Hdot^{\si_{c}+\de}}\Sob{\bar{\tht}^{n}}{L^{3}_{T}\Hdot^{\frac{2\kap}{3}}}.
\end{align*}
As a result, we have
\begin{align*}
     \Sob{\bar{\tht}^{n+1}}{L^{3}_{T}\Hdot^{\frac{2\kap}{3}}}&\le C_{4}\Sob{\tht^{n}}{L^{2}_{T}\Hdot^{\si_{c}+\frac{\kap}{2}}\cap L^{\frac{\kap}{\de}}_{T}\Hdot^{\si_{c}+\de}}\left(    \Sob{\bar{\tht}^{n+1}}{L^{3}_{T}\Hdot^{\frac{2\kap}{3}}}+    \Sob{\bar{\tht}^{n}}{L^{3}_{T}\Hdot^{\frac{2\kap}{3}}}\right).
\end{align*}
Let $T_{0}$ additionally satisfy $\mathscr{R}(T_{0})\le 1/(6C_{4})$. This implies
\begin{align}\label{est:cauchy2}
    \Sob{\bar{\tht}^{n+1}}{L^{3}_{T_0}\Hdot^{\frac{2\kap}{3}}}\le \frac{1}{2}\Sob{\bar{\tht}^{n}}{L^{3}_{T_0}\Hdot^{\frac{2\kap}{3}}}.
\end{align}
From (\ref{est:uniform:L2}), (\ref{est:cauchy2}), and interpolation inequality (\ref{E:interpolation}), we conclude that there exists a function $\tht(x,t)$ satisfying
\begin{align*}
    &\tht\in L^{\infty}_{T_{0}}H^{\si_c}\cap L^{2}_{T_{0}}\Hdot^{\si_{c}+\frac{\kap}{2}},\\
    & \tht^{n}\xrightharpoonup {\text{w*}}\tht\quad\text{in}\quad L^{\infty}_{T_{0}}H^{\si_c},\\
    & \tht^{n} \longrightarrow \tht\quad \text{in} \quad L^{3}_{T_0}\Hdot^{\si},\quad \forall \, \si\in[2\kap/3,\si_c+\kap/3). \\
    &u^{n}\cdot \nabla \tht^{n+1} \rightarrow u\cdot \nabla \tht \quad \text{in} \quad   L^{3}_{T_{0}}\Hdot^{\til{\si}},\quad \forall\,\til{\si}\in[0,\si_c-1).
\end{align*}
\subsubsection*{\textbf{Case: $\beta\ge 1+\kappa$} }
{Taking the inner product in $L^2$ of \eqref{E:diff} with $\bar{\tht}^{n+1}$ yields}
	\begin{align}\label{balance:diff}
	 {   \frac{1}{2}\frac{d}{dt}\Sob{\bar{\tht}^{n+1}}{L^2}^2+}{{\gam} \Sob{\Lam^{\kap/2}\bar{\tht}^{n+1}}{L^2}^2}= I_{1}'+I_{2}'+I_{3}',
	   \end{align}
where
\[I_{1}'=\lb(\nabla^{\perp}\Lambda^{\be-2}\tht^{n} \cdot \nabla) \bar{ \theta}^{n+1},\bar{ \theta}^{n+1}\rb,\]
\[I_{2}'=\lb \Lam^{\be-2}\nabla \cdot(({\nabla}^\perp\bar{\tht}^{n+1})\tht^{n}),\bar{ \theta}^{n+1}\rb,\]
\begin{align*}
    I_{3}'&=\lb (\nabla^{\perp}\Lambda^{\be-2}\bar{\tht}^{n} \cdot \nabla  \theta^{n}),\bar{ \theta}^{n+1}\rb+\lb \Lambda^{\be-2}\nabla\cdot( (\nabla^{\perp}\tht^{n})\bar{\tht}^{n}) ,\bar{ \theta}^{n+1}\rb.
\end{align*}
Note that $I_{1}'=0$ and $I_{2}'$ can be written in terms of a commutator just like in (\ref{E:skew-adjoint}). 
\begin{align*}
    I_{2}'=-\frac{1}{2}\lb[A_{\ell},\bdy_{\ell}\tht^{n}]\bar{\tht}^{n+1},\bar{\tht}^{n+1} \rb,
\end{align*}
where $A$ is as in (\ref{def:A:perp})
By Lemma $\ref{lem:commutator2}$ with  $\rho_{1}=\rho_{2}=\kap/2$, and Young's inequality we obtain
\begin{align}\label{est:I2'}
    |I_{2}'|&\le C\Sob{\tht^{n}}{\Hdot^{\si_{c}+\frac{\kap}{2}}}\Sob{\bar{\tht}^{n+1}}{\Hdot^{\frac{\kap}{2}}}\Sob{\bar{\tht}^{n+1}}{L^2}\notag\\
    &\le \frac{C}{\gam}\Sob{\tht^{n}}{\Hdot^{\si_{c}+\frac{\kap}{2}}}^{2}\Sob{\bar{\tht}^{n+1}}{L^2}^{2}+\frac{\gam}{4}\Sob{\bar{\tht}^{n+1}}{\Hdot^{\frac{\kap}{2}}}^{2}.
\end{align}
Note that $I_{3}'$ can be re-written as
\begin{align*}
    I_{3}'&=\lb (\nabla^{\perp}\Lambda^{\be-2}\bar{\tht}^{n} \cdot \nabla  \theta^{n}),\bar{ \theta}^{n+1}\rb-\lb \Lambda^{\be-2}\nabla^{\perp}\cdot( \bar{\tht}^{n}(\nabla \tht^{n})) ,\bar{ \theta}^{n+1}\rb\\&=-\lb[A_{\ell},\bdy_{\ell}\tht^{n}]\bar{\tht}^{n},\bar{\tht}^{n+1} \rb.
\end{align*}
Applying the Cauchy-Schwarz inequality, Lemma $\ref{lem:commutator2}$ with  $\rho_{1}=\kap/2$ and $\rho_{2}=0$, and Young's inequality, we obtain
\begin{align}\label{est:I3'}
    |I_{3}'|&=\lb\Lam^{-\frac{\kap}{2}}[A_{\ell},\bdy_{\ell}\tht^{n}]\bar{\tht}^{n},\Lam^{\frac{\kap}{2}}\bar{\tht}^{n+1} \rb\notag\\&\le C\Sob{\tht^{n}}{\Hdot^{\si_{c}+\frac{\kap}{2}}}\Sob{\bar{\tht}^{n+1}}{\Hdot^{\frac{\kap}{2}}}\Sob{\bar{\tht}^{n}}{L^2}\notag\\
    &\le \frac{C}{\gam}\Sob{\tht^{n}}{\Hdot^{\si_{c}+\frac{\kap}{2}}}^{2}\Sob{\bar{\tht}^{n}}{L^2}^{2}+\frac{\gam}{4}\Sob{\bar{\tht}^{n+1}}{\Hdot^{\frac{\kap}{2}}}^{2}.
\end{align}
From estimates (\ref{est:I2'}), (\ref{est:I3'}), and equation (\ref{balance:diff}), we obtain
\begin{align}\notag
    \frac{1}{2}\frac{d}{dt}\Sob{\bar{\tht}^{n+1}}{L^2}^2+\frac{\gam}{2} \Sob{\bar{\tht}^{n+1}}{\Hdot^{\frac{\kap}{2}}}^2\le C\Sob{\tht^{n}}{\Hdot^{\si_{c}+\frac{\kap}{2}}}^{2}\left(\Sob{\bar{\tht}^{n+1}}{L^2}^{2}+ \Sob{\bar{\tht}^{n}}{L^2}^{2}\right).
\end{align}
We integrate in time and take the $L^{\infty}$ norm with respect to time on both sides to obtain
\begin{align}\notag
    \Sob{\bar{\tht}^{n+1}}{L^{\infty}_{T}L^2}\le C_{3}\Sob{\tht^{n}}{L^{2}_{T}\Hdot^{\si_{c}+\frac{\kap}{2}}}\left(    \Sob{\bar{\tht}^{n+1}}{L^{\infty}_{T}L^2}+    \Sob{\bar{\tht}^{n}}{L^{\infty}_{T}L^2}\right).
\end{align}
Let $T_{0}$ additionally satisfy $\mathscr{R}(T_{0})\le 1/(6C_{3})$. This implies
\begin{align}\label{est:cauchy1}
    \Sob{\bar{\tht}^{n+1}}{L^{\infty}_{T_{0}}L^2}\le \frac{1}{2}\Sob{\bar{\tht}^{n}}{L^{\infty}_{T_{0}}L^2}.
\end{align}
From (\ref{est:uniform:L2}), (\ref{est:cauchy1}), and interpolation inequality (\ref{E:interpolation}) we conclude that there exists a function $\tht(x,t)$ satisfying
\begin{align*}
    &\tht\in L^{\infty}_{T_{0}}H^{\si_c}\cap L^{2}_{T_{0}}\Hdot^{\si_{c}+\frac{\kap}{2}},\\
   & \tht^{n}\xrightharpoonup {\text{w*}}\tht\quad\text{in}\quad L^{\infty}_{T_{0}}H^{\si_c},\\
    & \tht^{n} \longrightarrow \tht\quad \text{in} \quad L^{\infty}_{T_{0}}H^{\si}, \forall \,\si\in[0,\si_c).
\end{align*}
It is now straightforward to check that for any $\til{\si}\in [0,\si_{c}-1)$, we have
\begin{align*}
      \Div F_{-\tht^{n}}(\tht^{n+1}) \rightarrow u\cdot \nabla \tht \quad \text{in} \quad L^{\infty}_{T_{0}}H^{\til{\si}}.
\end{align*}

\subsection{Continuity in time} Let $\tht$ be a solution of (\ref{E:dissipative-beta}) obtained above. We have observed that
\[\tht \in L^{2}(0,T^{*};H^{\be+1-\kap/2}), \, u\in L^{2}(0,T^{*};H^{2-\kap/2}),\]
for any $T^{*}<T_{0}$. Since
\[\partial_{t}\theta =-{\gam} \Lam^{\kap}\tht- u \cdot \nabla \theta,\]
we can observe that $\partial_{t}\theta \in L^{1}(0,T^{*};H^{1-\kap})$. As a result, we obtain
\[\tht \in C([0,T^{*});H^{1-\kap}).\] 
By applying Lemma 1.4 in \cite[pg. 263]{TemamBook2001}, we obtain 
\begin{align}\label{E:weak-cont}
    \tht \in C_{w}([0,T^{*});H^{\be+1-\kap}).
\end{align}
Using (\ref{est:uniform:L2}) and (\ref{E:weak-cont}), we have
\begin{align*}
    \limsup_{t \to 0}\Sob{\tht(t)-\tht_{0}}{H^{\be+1-\kap}}^{2}&=\limsup_{t \to 0}\left \{\Sob{\tht(t)}{H^{\be+1-\kap}}^{2}+\Sob{\tht_{0}}{H^{\be+1-\kap}}^{2}-2\langle \tht(t),\tht_{0}\rangle_{H^{\be+1-\kap}}\right\}\\
    &\le \lim_{t \to 0}\left\{\Sob{\tht_{0}}{H^{\be+1-\kap}}^{2}\exp(C \mathscr{R}(t)^2)+\Sob{\tht_{0}}{H^{\be+1-\kap}}^{2}-2\langle \tht(t),\tht_{0}\rangle_{H^{\be+1-\kap}}\right \}\\
    &=0,
\end{align*}
where we used the fact that $\lim_{t \to 0}\mathscr{R}(t)=0$. This establishes the right continuity of $\tht$ at $t=0$. By a standard bootstrap argument, we obtain
\[\tht \in C([0,T_{0});H^{\be+1-\kap}).\]
\subsection{Uniqueness}
To establish uniqueness, we consider two solutions of (\ref{E:dissipative-beta}), denoted by $\tht^{(1)}$ and $\tht^{(2)}$. Let $\bar{\tht}=\tht^{(1)}-\tht^{(2)}, \, \bar{u}=u^{(1)}-u^{(2)}$. Then, $\bar{\tht}$ satisfies the following equation:
\begin{align}
    \label{E:dissipative-beta:uniq}
	\begin{split}
	{\partial_{t}\bar{\theta} +{\gam}\Lam^{\kap}\bar{\tht}+ u^{(1)} \cdot \nabla \bar{\theta}+\bar{u} \cdot \nabla \theta^{(2)}=0.}
	\end{split}
\end{align}
{Taking the inner product in $L^2$ of (\ref{E:dissipative-beta:uniq}) with $\bar{\tht}$ yields}
\begin{align}\label{balance:beta:uniq}
	   \frac{1}{2}\frac{d}{dt}\Sob{\bar{\tht}}{L^2}^2+{\gam} \Sob{\Lam^{\frac{\kap}{2}}\bar{\tht}}{L^2}^2
	   &=-\lb (u^{(1)} \cdot \nabla) \bar{\theta},\bar{\tht}\rb-\lb (\bar{u} \cdot \nabla) \theta^{(2)},\bar{\tht}\rb\notag\\
	   &=I_{1}''+I_{2}''.
	\end{align}
Observe that since $\nabla \cdot u^{1}=0$, we have $I_{1}''=0$. 
$I_{2}''$ can be written in terms of a commutator just like in (\ref{E:skew-adjoint}). 
\begin{align}
I_{2}''&=-\frac{1}{2}\lb[A_{\ell},\bdy_{\ell}\tht^{(2)}]\bar{\tht},\bar{\tht}\rb.\notag
\end{align}
By \cref{lem:commutator2} with $\rho_1=\rho_2=\kap/2$, and Young's inequality, we obtain
\begin{align}\label{est:I2''}
     |I_{2}''|&\le C\Sob{\bdy_{\ell}\tht^{(2)}}{\Hdot^{\be-\frac{\kap}{2}}}\Sob{\bar{\tht}}{L^2}\Sob{\bar{\tht}}{\Hdot^{\frac{\kap}{2}}}\notag\\
     &\le C\Sob{\tht^{(2)}}{\Hdot^{\si_{c}+\frac{\kap}{2}}}^{2}\Sob{\bar{\tht}}{L^2}^{2}+\frac{\gam}{2}\Sob{\bar{\tht}}{\Hdot^{\frac{\kap}{2}}}^{2}.
\end{align}
From (\ref{balance:beta:uniq}) and (\ref{est:I2''}), we obtain
\begin{align*}
    \frac{d}{dt}\Sob{\bar{\tht}}{L^2}^2\le C\Sob{\tht^{(2)}}{\Hdot^{\si_{c}+\frac{\kap}{2}}}^{2}\Sob{\bar{\tht}}{L^2}^{2}.
\end{align*}
An application of the Gronwall inequality then establishes uniqueness.

\subsection{Gevrey regularity} Invoking the apriori estimate (\ref{est:ET:apriori}) for the approximating equation (\ref{E:approximation}), we obtain
 \begin{align}\label{est:ET:apriori:approx}
        \Sob{\tht^{n+1}(\cdotp)}{X_T}\le {C_5}\mathcal{I}_T(\tht_{0})+{C_6}\gam^{-1}\Sob{\tht^{n}(\cdotp)}{X_T}\Sob{\tht^{n+1}(\cdotp)}{X_T}.
    \end{align}
Assume that $C_{6}\gam^{-1}\Sob{\tht^{n}(\cdotp)}{X_T}\le 1/2$, where the $X_T$--norm is defined in \eqref{def:XT}. From (\ref{est:ET:apriori:approx}), we obtain
\begin{align}\notag
    \Sob{\tht^{n+1}(\cdotp)}{X_T}\leq 2C_{5}\mathcal{I}_T(\tht_{0}).
\end{align}
By \eqref{lim:XT:initialdata}, for arbitrary initial datum $\tht_0$, $T$ can be chosen sufficiently small such that $T\le T_0$ and
\begin{align}\label{est:ET:induction}
    2C_{6}C_{5}\mathcal{I}_T(\tht_{0})\le \gam/2.
\end{align}
This condition also holds if $\Sob{\tht_0}{\Hdot^{\si_c}}$ is small enough and $T=\infty$. By induction, and (\ref{est:ET:induction}), we have the uniform-in-$n$ bound on $\Sob{\tht^{n}(\cdotp)}{X_T}$ given by
\[C_{6}\gam^{-1}\Sob{\tht^{n}(\cdotp)}{X_T}\le1/2.\]
It is now straightforward to check that the solution $\tht$ will also satisfy the above bounds. 

\section{Existence, uniqueness, and smoothing for the endpoint case $\beta=2$: Proof of \cref{thm:main:log}}\label{sect:proofs:main:log}

In this section, we prove our second main result, \cref{thm:main:log}. This is carried out by proving apriori estimates for the norms $\Sob{\tht}{H^{\si}}$ and $\Sob{\tht}{\dot{G}^{\lam}_{\al, \si}}$. The existence of a solution then can be carried out by using a standard artificial viscosity approximation; we refer the reader to \cite{KumarThesis2021} for additional details.

\subsection{Existence}To simplify the treatment, we only consider the case when $\si \in (3-\kap,3)$ and $\al \in (\si-3+\kap,\kap)$. Let $\de_{1}$ be chosen such that
\begin{align*}
    0<\de_{1}<\min\left\{\mu,\frac{ \si-3+\kap}{2},\kap-\alpha\right\}.\\
\end{align*}
Denote by 
\[ \de_{2} =2(\si-3+\kap-\de_{1}).\]
For (\ref{E:dissipative-log}), velocity is given by $ u=-\nabla^{\perp}\ldmu \tht$. Since $\nabla \cdot u=0$, we have $	\frac{1}{2}\frac{d}{dt} \| \tht \|^{2}_{L^2}\le0$, so that
	\begin{align}\label{est:L2:log}
 \nrm{\tht(\cdot,t)}_{L^2}\le\nrm{\tht_{0}}_{L^2}.
	\end{align}
Define
 \begin{align*}
   \lam(t)=\lam t,\quad \lam>0.
    \end{align*}
    
{Upon applying the operator ${G}^{\lam(t)}_{\al}\Lam_j^{\si}$ to (\ref{E:dissipative-log}) and invoking \eqref{identity:dt:Gevrey} with $\ph=\Lam_j^{\si}\tht$, one has}
    \begin{align}\label{E:lambda-log}
    	\begin{split}
	     { \bdy_t({G}_\al^{\lam(t)}\Lam_j^{\si}\tht)-}{\lam'(t){G}_\al^{\lam(t)}\Lam^{{\si}+\al}_j\tht +{\gam} {G}^{\lam(t)}_\al\Lam_j^{{\si}+\kappa}\tht + {G}^{\lam(t)}_\al\Lam_j^{\si} (u \cdot \nabla \tht)=0}.
    	\end{split}
	\end{align}
{Then, taking the inner product in $L^2$ of \eqref{E:lambda-log} with ${G}_\al^{\lam(t)}\Lam_j^{\si}\tht$ yields}
	\begin{align}\label{balance:gevrey:basic1:log}
	   \frac{1}{2}\frac{d}{dt}\Sob{\Lam^{\si}\til{\tht}_j}{L^2}^2+{\gam} \Sob{\Lam^{{\si}+\kap/2}\til{\tht}_j}{L^2}^2
	   =\lam \Sob{\Lam^{{\si}+\al/2}\til{\tht}_j}{L^2}^2-\lb {G}_\al^{\lam(t)}\Lam_j^{\si} (u \cdot \nabla \tht),\Lam^{\si}\til{\tht}_j\rb.
	\end{align}
By interpolation inequality (\ref{E:interpolation}) and Young's inequality, we have 
\begin{align*}
    \lam\Sob{\Lam^{\si+\al/2}\tiltht_{j}}{L^2}^{2}&\le \lam\Sob{\Lam^{\si}\tiltht_{j}}{L^2}^{\frac{2(\kap-\al)}{\kap}}\Sob{\Lam^{\si+\kap/2}\tiltht_{j}}{L^2}^{\frac{2\al}{\kap}}\notag\\
    &\le \frac{\gam}{2}\Sob{\Lam^{\si+\kap/2}\tiltht_{j}}{L^2}^{2}+\frac{C}{\gam^{\frac{\al}{\kap-\al}}}\Sob{\Lam^{\si}\tiltht_{j}}{L^2}^{2}.
\end{align*}
Using this in (\ref{balance:gevrey:basic1:log}), we obtain
\begin{align*}
	   \frac{d}{dt}\Sob{\Lam^{\si}\til{\tht}_j}{L^2}^2+{\gam} \Sob{\Lam^{{\si}+\kap/2}\til{\tht}_j}{L^2}^2
	   \le C\Sob{\Lam^{{\si}}\til{\tht}_j}{L^2}^2-\lb {G}_\al^{\lam(t)}\Lam_j^{\si} (u \cdot \nabla \tht),\Lam^{\si}\til{\tht}_j\rb.
	\end{align*}
The nonlinear term is decomposed as follows:
	\begin{align*}
	\lb {G}_\al^{\lam(t)}\Lam_j^{\si} (u \cdot \nabla \tht),\Lam^{\si}\til{\tht}_j\rb=&\left \langle (\Lam^{\si}\til{u}_{j} \cdot \nabla \tht),\Lam^{\si} \tiltht_j \right \rangle+\left \langle (u\cdot {\nabla} \Lam^{\si} \tiltht_j),\Lam^{\si} \tiltht_j \right \rangle\notag\\
	&+\left \{ \left \langle {G}_\al^{\lam(t)}\Lam_j^{\si} (u \cdot \nabla \tht),\Lam^{\si}\til{\tht}_j\right \rangle- \left \langle (\Lam^{\si}\til{u}_{j} \cdot \nabla \tht),\Lam^{\si} \tiltht_j \right \rangle -\left \langle (u\cdot {\nabla} \Lam^{\si} \tiltht_j),\Lam^{\si} \tiltht_j \right \rangle \right\}\notag\\
	=&K_{1}+K_{2}+K_{3}.
\end{align*}
Since $\nabla\cdot u=0$, we have $K_{2}=0$. We obtain estimates for $K_1$ and $K_3$ below.
\subsubsection*{Bound for ${K_1}$\nopunct}:
	 Letting $A_\ell:=-\bdy_{\ell}^{\perp} \ldmu$, for $\ell=1,2$, we observe, as in \cite{ChaeConstantinCordobaGancedoWu2012}, that $A_\ell$ is a skew-adjoint operator, i.e., $\left \langle A_\ell h,g \right \rangle=-\left \langle h,A_\ell g \right \rangle$. In particular, we have
   	\begin{align}\label{skew:adj:commutator:log}
	\begin{split}
	\left \langle [A_\ell,g]h,h \right \rangle=-2\left \langle (A_\ell h)g,h \right \rangle.
	\end{split}
	\end{align}
    Using this we can express $K_{1}$ in terms of a commutator
	\begin{align*}
	K_{1}&=-\left \langle \nabla^{\perp}\ldmu \Lam^{\si}\tiltht_j\cdot \nabla \tht,\Lam^{\si}\til{\tht}_j \right \rangle\notag\\
	    &=\left \langle (A_{\ell}\Lam^{\si} \til{\tht}_j)\bdy_{\ell}\tht ,\Lam^{\si}\til{\tht}_j\right \rangle\notag\\
	    &=-\frac{1}2\left \langle [A_{\ell},\bdy_{\ell}\tht]\Lam^{\si}\tiltht_j,\Lam^{\si}\tiltht_j\right \rangle. 
\end{align*} 
 By \cref{lem:commutator4} with { $f=h=\Lam^{\si}\til{\tht}_j$, $g=\bdy_l\tht$, and ${\eps}=\kap/2-\de_{1}$, $ \de=\de_{1}$, $\rho=\de_{2}/2$}, it follows that
    \begin{align}\label{est:K1}
        |K_1|&\leq C\Sob{\tht}{{\Hdot}^{\si+\kap/2}}^{\frac{2}{2+\de_{2}}}\Sob{\tht}{{\Hdot}^{2-\kap/2+\de_{1}}}^{\frac{\de_{2}}{2+\de_{2}}}\Sob{\Lam^{\si}\til{\tht}_j}{L^2}\Sob{\Lam^{\si+\kap/2}\til{\tht}_j}{L^2}.
    \end{align}
\subsubsection*{{Bound for} ${K_3}$\nopunct}:
	By applying \eqref{split:Lam:rewrite} and the product rule, we have
	\begin{align*}
	\begin{split}
	K_{3}=&-\left \langle {G}_\al^{\lam(t)} \Lam_j^{\si-2}\bdy_l(\bdy_{l}A_{\ell}\tht\,\bdy_{\ell}\tht),\Lam^{\si} \tiltht_j  \right \rangle
	+ \left \langle  ({G}_\al^{\lam(t)}\Lam_j^{\si-2}\bdy_l(\bdy_{l}A_{\ell}\tht)) \bdy_{\ell} \tht,  \Lam^{\si}\tiltht_j \right \rangle \\
    	&-\left \langle {G}_\al^{\lam(t)} \Lam_j^{\si-2}\bdy_l(A_{\ell}\tht\,\bdy_{l}\bdy_{\ell} \theta),\Lam^{\si} \tiltht_j  \right \rangle
    	+\left \langle A_{\ell}\tht ({G}_\al^{\lam(t)}\Lam_j^{\si-2}\bdy_l(\bdy_{l}\bdy_{\ell} \tht)),\Lam^{\si} \tiltht_j \right \rangle  \notag\\
    	=&-\left\lb [{G}_\al^{\lam(t)} \Lam_j^{\si-2}\bdy_l,\bdy_\ell\tht]\bdy_{l}A_{\ell}\tht,\Lam^{\si}\tiltht_j\right\rb
    	-\left\lb [{G}_\al^{\lam(t)} \Lam_j^{\si-2}\bdy_l,A_{\ell}\tht]\bdy_l\bdy_\ell\tht,\Lam^{\si}\tiltht_j\right\rb\notag\\
    	=&K_3^a+K_3^b.
	\end{split}
	\end{align*}
Applying \cref{lem:commutator3} with {$\si$ replaced by $\si-2$, $\rho=0$, $\nu=(3-\si)/2$, $\zeta=1-\kap+\al+\de_{1}$ and $f=\bdy_{l}A_{\ell}\tht$, $g=\bdy_\ell\tht$, $h=\Lam^{\si}\tiltht_j$}, (\ref{E:log:inequality}) and Bernstein's inequality, we obtain
    \begin{align}\label{est:K3a:int}
         |K_3^a|\leq& c_{j}2^{\frac{(3-\si) j}{2}}\Sob{{G}^{\lam}_{\al}\bdy_{l}A_{\ell}\tht}{\Hdot^{\si-2}}\Sob{{G}^{\lam}_{\al}\bdy_\ell\tht}{\Hdot^{\frac{\si+1}{2}}}\Sob{\Lam^{\si}\tiltht_j}{L^2}\notag\\&+C\lam t2^{(\si-2+\kap-\de_{1})j}\Sob{{E}^{\lam}_{\al}\bdy_{\ell}\tht}{\Hdot^{2-\kap+\al+\de_{1}}}\Sob{\bdy_{l}A_{\ell}\tiltht_{j}}{L^2}\Sob{\Lam^{\si}\tiltht_j}{L^2}\notag\\
         \leq& b_{j}\Sob{{G}^{\lam}_{\al}\tht}{\Hdot^{\si+\de_1}}\Sob{{G}^{\lam}_{\al}\tht}{\Hdot^{\frac{\si+3}{2}}}^{2}+b_{j}\lam t\Sob{{E}^{\lam}_{\al}\tht}{\Hdot^{3-\kap+\al+\de_{1}}}\Sob{{G}^{\lam}_{\al}\tht}{\Hdot^{\si+\frac{\kap}{2}-\frac{\de_{1}}{2}}}\Sob{{G}^{\lam}_{\al}\tht}{\Hdot^{\si+\frac{\kap}{2}}},
    \end{align}
for some $\{b_j\}\in\ell^1(\ZZ)$.
By interpolation inequality (\ref{E:interpolation}), we have
\begin{align*}
    \Sob{{G}^{\lam}_{\al}\tht}{\Hdot^{\si+\de_1}}&\le \Sob{{G}^{\lam}_{\al}\tht}{\Hdot^{\si}}^{1-\frac{2\de_{1}}{\kap}}\Sob{{G}^{\lam}_{\al}\tht}{\Hdot^{\si+\frac{\kap}{2}}}^{\frac{2\de_{1}}{\kap}},\\
    \Sob{{G}^{\lam}_{\al}\tht}{\Hdot^\frac{\si+3}{2}}&\le \Sob{{G}^{\lam}_{\al}\tht}{\Hdot^{\si}}^{1-\frac{3-\si}{\kap}}\Sob{{G}^{\lam}_{\al}\tht}{\Hdot^{\si+\frac{\kap}{2}}}^{\frac{3-\si}{\kap}},\\
    \Sob{{G}^{\lam}_{\al}\tht}{\Hdot^{\si+\frac{\kap}{2}-\frac{\de_1}{2}}}&\le \Sob{{G}^{\lam}_{\al}\tht}{\Hdot^{\si}}^{\frac{\de_{1}}{\kap}}\Sob{{G}^{\lam}_{\al}\tht}{\Hdot^{\si+\frac{\kap}{2}}}^{1-\frac{\de_{1}}{\kap}}.
\end{align*}
Using an argument similar to (\ref{est:sobolev:de'}), we obtain
\[\Sob{{E}^{\lam}_{\al}\tht}{\Hdot^{3-\kap+\al+\de_{1}}}\le C(\lam t)^{-1+\de_{2}/2\al}\Sob{G^{\lam}_{\al}\tht}{\Hdot^{\si}}.\]
Using above bounds in (\ref{est:K3a:int}), we obtain the following estimate for $K_{3}^{a}$:
\begin{align}\label{est:K3a}
     |K_3^a| \leq& b_{j}\Sob{{G}^{\lam}_{\al}\tht}{\Hdot^{\si}}^{1+\frac{\de_{2}}{\kap}}\Sob{{G}^{\lam}_{\al}\tht}{\Hdot^{\si+\frac{\kap}{2}}}^{2-\frac{\de_{2}}{\kap}}+b_{j}(\lam t)^{\frac{\de_{2}}{2\al}}\Sob{{G}^{\lam}_{\al}\tht}{\Hdot^{\si}}^{1+\frac{\de_{1}}{\kap}}\Sob{{G}^{\lam}_{\al}\tht}{\Hdot^{\si+\frac{\kap}{2}}}^{2-\frac{\de_{1}}{\kap}},
\end{align}
for some $\{b_j\}\in\ell^1(\ZZ)$. Similarly, applying \cref{lem:commutator3} with $\si$ replaced by  $\si-2$, $\rho=0$, $\nu=(3-\si)/2$, but with $\zeta=1-\kap+\al+\de_{1}/2$ and $f=\bdy_l\bdy_\ell\tht$, $g=A_{\ell}{\tht}$, $h=\Lam^{\si}\tiltht_j$, (\ref{E:log:inequality}) and Bernstein's inequality, we obtain
\begin{align}\label{est:K3b}
         |K_3^b|\leq& c_{j}2^{\frac{(3-\si) j}{2}}\Sob{G^{\lam}_{\al}A_{\ell}{\tht}}{\Hdot^{\si-1}}\Sob{G^{\lam}_{\al}\bdy_l\bdy_\ell\tht}{\Hdot^{\frac{\si-1}{2}}}\Sob{\Lam^{\si}\tiltht_j}{L^2}\notag\\&+C\lam t 2^{(\si-2+\kap-\frac{\de_{1}}{2})j}\Sob{{E}^{\lam}_{\al}A_{\ell}\tht}{\Hdot^{2-\kap+\al+\frac{\de_{1}}{2}}}\Sob{\bdy_{l}\bdy_{\ell}\tiltht_{j}}{L^2}\Sob{\Lam^{\si}\tiltht_j}{L^2}\notag\\
         \leq&b_{j}\Sob{{G}^{\lam}_{\al}\tht}{\Hdot^{\si}}^{1+\frac{\de_{2}}{\kap}}\Sob{{G}^{\lam}_{\al}\tht}{\Hdot^{\si+\frac{\kap}{2}}}^{2-\frac{\de_{2}}{\kap}}+b_{j}(\lam t)^{\frac{\de_{2}}{2\al}}\Sob{{G}^{\lam}_{\al}\tht}{\Hdot^{\si}}^{1+\frac{\de_{1}}{\kap}}\Sob{{G}^{\lam}_{\al}\tht}{\Hdot^{\si+\frac{\kap}{2}}}^{2-\frac{\de_{1}}{\kap}},
    \end{align}
for some $\{b_j\}\in\ell^1(\ZZ)$.

Collecting the estimates for $K_1$ and $K_3$, we obtain
\begin{align*}
    {\frac{d}{dt}}\Sob{\Lam^{\si}\til{\tht}_j}{L^2}^2+{\gam} \Sob{\Lam^{{\si}+\kap/2}\til{\tht}_j}{L^2}^2
	   \le& C\Sob{\Lam^{{\si}}\til{\tht}_j}{L^2}^2+C\Sob{\tht}{{\Hdot}^{\si+\kap/2}}^{\frac{2}{2+\de_{2}}}\Sob{\tht}{{\Hdot}^{2-\kap/2+\de_{1}}}^{\frac{\de_{2}}{2+\de_{2}}}\Sob{\Lam^{\si}\til{\tht}_j}{L^2}\Sob{\Lam^{\si+\kap/2}\til{\tht}_j}{L^2}\notag\\&+c_{j}\Sob{{G}^{\lam}_{\al}\tht}{\Hdot^{\si}}^{1+\frac{\de_{2}}{\kap}}\Sob{{G}^{\lam}_{\al}\tht}{\Hdot^{\si+\frac{\kap}{2}}}^{2-\frac{\de_{2}}{\kap}}+c_{j}(\lam t)^{\frac{\de_{2}}{2\al}}\Sob{{G}^{\lam}_{\al}\tht}{\Hdot^{\si}}^{1+\frac{\de_{1}}{\kap}}\Sob{{G}^{\lam}_{\al}\tht}{\Hdot^{\si+\frac{\kap}{2}}}^{2-\frac{\de_{1}}{\kap}}.
\end{align*}
Summing in $j$ and applying the Cauchy-Schwarz inequality followed by Young's inequality yields
\begin{align}\label{balance:gevrey:summed:log}
     {\frac{d}{dt}}\Sob{{G}^{\lam}_{\al}\tht}{\Hdot^{\si}}^2+{\gam} \Sob{{G}^{\lam}_{\al}\tht}{\Hdot^{\si+\kap/2}}^2
	   \le& C\Sob{{G}^{\lam}_{\al}\tht}{\Hdot^{\si}}^2+C\Sob{\tht}{{\Hdot}^{2-\kap/2+\de_{1}}}^{\frac{\de_{2}}{2+\de_{2}}}\Sob{{G}^{\lam}_{\al}\tht}{\Hdot^{\si}}\Sob{{G}^{\lam}_{\al}\tht}{{\Hdot}^{\si+\kap/2}}^{\frac{4+\de_{2}}{2+\de_{2}}}\notag\\
	   &+C\Sob{{G}^{\lam}_{\al}\tht}{\Hdot^{\si}}^{1+\frac{\de_{2}}{\kap}}\Sob{{G}^{\lam}_{\al}\tht}{\Hdot^{\si+\frac{\kap}{2}}}^{2-\frac{\de_{2}}{\kap}}+C(\lam t)^{\frac{\de_{2}}{2\al}}\Sob{{G}^{\lam}_{\al}\tht}{\Hdot^{\si}}^{1+\frac{\de_{1}}{\kap}}\Sob{{G}^{\lam}_{\al}\tht}{\Hdot^{\si+\frac{\kap}{2}}}^{2-\frac{\de_{1}}{\kap}}\notag\\
	   \le& C\Sob{{G}^{\lam}_{\al}\tht}{\Hdot^{\si}}^2+C\Sob{\tht}{{\Hdot}^{2-\kap/2+\de_{1}}}^{2}\Sob{{G}^{\lam}_{\al}\tht}{\Hdot^{\si}}^{2+\til{\si}_{1}}+\frac{\gam}{2}\Sob{{G}^{\lam}_{\al}\tht}{\Hdot^{\si+\kap/2}}^{2}\notag\\
	   &+C\Sob{{G}^{\lam}_{\al}\tht}{\Hdot^{\si}}^{2+\til{\si}_{2}}+C(\lam t)^{\frac{\kap\de_{2}}{\al\de_{1}}}\Sob{{G}^{\lam}_{\al}\tht}{\Hdot^{\si}}^{2+\til{\si}_{3}},
\end{align}
where
\begin{align*}
    \til{\si}_{1}=\frac{4}{\de_{2}},\quad
    \til{\si}_{2}=\frac{2\kap}{\de_{2}}, \quad
    \til{\si}_{3}=\frac{2\kap}{\de_{1}}.
\end{align*}
In order to get an apriori estimate for $\Sob{\tht}{H^{\si}}$, we recall (\ref{est:L2:log}) and suppress the Gevrey multiplier in (\ref{balance:gevrey:summed:log}) by taking $\lam=0$. We obtain
\begin{align}\label{balance:gevrey:intermediate:log:H}
    &\frac{d}{dt}\Sob{\tht}{H^{\si}}^2 \le C\left(\Sob{\tht}{{H}^{\si}}^{2}+\Sob{ \tht}{H^{\si}}^{4+\til{\si}_1}+\Sob{\tht}{H^{\si}}^{2+\til{\si}_2}\right).
\end{align}
Denote by \[y(t)=1+\Sob{\tht}{H^{\si}}^2, \] then from (\ref{balance:gevrey:intermediate:log:H}), we obtain
\begin{align*}
     \frac{dy}{dt}\le Cy^{\til{\si}_4},
\end{align*}
where
\[\til{\si}_4=\max\left\{2+\frac{2}{\de_{2}},1+\frac{\kap}{\de_{2}}\right\}.\]
We conclude that there exists a time $T_1=T_{1}(\Sob{\tht_0}{H^{\si}})$ such that $\tht(x,t)$ satisfies
\begin{align*}
    \Sob{\tht}{L^{\infty}_{T_{1}}H^{\si}}\le C(1+\Sob{\tht_{0}}{H^{\si}}).
\end{align*}
The existence of a solution $\tht(x,t)$ now follows, for instance, from a standard argument via artificial viscosity.

Applying the interpolation inequality (\ref{E:interpolation}) in (\ref{balance:gevrey:summed:log}), we obtain
\begin{align}\label{balance:gevrey:log:final}
  {\frac{d}{dt}}&\Sob{{G}^{\lam}_{\al}\tht}{\Hdot^{\si}}^2
	   \le C\Sob{{G}^{\lam}_{\al}\tht}{\Hdot^{\si}}^2+C\Sob{\tht}{{L^2}}^{1-\til{\si}_{5}}\Sob{{G}^{\lam}_{\al}\tht}{\Hdot^{\si}}^{3+\til{\si}_{1}+\til{\si}_{5}}+C\Sob{{G}^{\lam}_{\al}\tht}{\Hdot^{\si}}^{2+\til{\si}_{2}}+C(\lam t)^{\frac{\kap\de_{2}}{\al\de_{1}}}\Sob{{G}^{\lam}_{\al}\tht}{\Hdot^{\si}}^{2+\til{\si}_{3}},
\end{align}
where
\[\til{\si}_{5}=\frac{2\de_{1}-\de_{2}+2}{2\si}.\]
Denote by \[z(t)=1+\Sob{{G}^{\lam(t)}_{\al}\tht}{\Hdot^{\si}}^2, \] then (\ref{balance:gevrey:log:final}) implies
\begin{align*}
     \frac{dz}{dt}\le C(t,\Sob{\tht_{0}}{L^{2}})z^{\til{\si}_{6}},
\end{align*}
where
\[\til{\si}_{6}=\max\left\{\frac{3}{2}+\frac{2}{\de_{2}}+\frac{2\de_{1}-\de_{2}+2}{4\si},1+\frac{\kap}{\de_{2}},1+\frac{\kap}{\de_{1}}\right\}.\]
From the above inequality, we conclude that there exists a time $T=T(\Sob{\tht_0}{H^{\si}})\le T_1$ such that $\tht(x,t)$ satisfies 
\begin{align}\notag
    \esssup_{0\le t\le T} \Sob{\tht(t)}{\Gdot^{\lam t}_{\al, \si}}\le C(1+\Sob{\tht_{0}}{H^{\si}}).
\end{align}
\subsection{Uniqueness}

To establish that the solution obtained above is unique, we consider two solutions of (\ref{E:dissipative-log}), denoted by $\tht^{(1)}$ and $\tht^{(2)}$. Let $\bar{\tht}=\tht^{(1)}-\tht^{(2)}, \, \bar{u}=u^{(1)}-u^{(2)}$. Then, $\bar{\tht}$ satisfies the following equation:
\begin{align}
    \label{E:dissipative-log:diff}
	\begin{split}
	{\partial_{t}\bar{\theta} +{\gam}\Lam^{\kap}\bar{\tht}+ u^{(1)} \cdot \nabla \bar{\theta}+\bar{u} \cdot \nabla \theta^{(2)}=0.}
	\end{split}
\end{align}
{Taking the inner product in $L^2$ of (\ref{E:dissipative-log:diff}) with $\bar{\tht}$ yields}
\begin{align}\label{balance:log:diff}
	   \frac{1}{2}\frac{d}{dt}\Sob{\bar{\tht}}{L^2}^2+{\gam} \Sob{\Lam^{\frac{\kap}{2}}\bar{\tht}}{L^2}^2
	   &=-\lb (u^{(1)} \cdot \nabla) \bar{\theta},\bar{\tht}\rb-\lb (\bar{u} \cdot \nabla) \theta^{(2)},\bar{\tht}\rb\notag\\
	   &=K_{1}'+K_{2}'.
	\end{align}
Observe that since $\nabla \cdot u^{1}=0$, we have $K_{1}'=0$. 
Using (\ref{skew:adj:commutator:log}), we can express $K_{2}'$ in terms of a commutator
\begin{align}
K_{2}'&=-\frac{1}{2}\lb[A_{\ell},\bdy_{\ell}\tht^{(2)}]\bar{\tht},\bar{\tht}\rb.\notag
\end{align}
By \cref{lem:commutator4} with { $f=h=\bar{\tht}$, $g=\bdy_l\tht^{(2)}$, and ${\eps}=\kap/2-\de_{1}$, $ \de=\de_{1}$, $\rho=\de_{2}/2$}, it follows that
\begin{align}\label{est:K2'}
 |K_{2}'|&\le C\Sob{\bdy_{\ell}\tht^{(2)}}{{\Hdot}^{\si+\kap/2-1}}^{\frac{2}{2+\de_{2}}}\Sob{\bdy_{\ell}\tht^{(2)}}{{\Hdot}^{1-\kap/2+\de_{1}}}^{\frac{\de_{2}}{2+\de_{2}}}\Sob{\bar{\tht}}{L^2}\Sob{\bar{\tht}}{\Hdot^{\kap/2}}\notag\\
 &\le C\Sob{\tht^{(2)}}{{\Hdot}^{\si+\kap/2}}^{\frac{4}{2+\de_{2}}}\Sob{\tht^{(2)}}{{H}^{\si}}^{\frac{2\de_{2}}{2+\de_{2}}}\Sob{\bar{\tht}}{L^2}^{2}+\frac{\gam}{2}\Sob{\bar{\tht}}{\Hdot^{\kap/2}}^{2}.
    \end{align}
From (\ref{balance:log:diff}) and (\ref{est:K2'}), we obtain
\begin{align*}
    \frac{d}{dt}\Sob{\bar{\tht}}{L^2}^2\le C\Sob{\tht^{(2)}}{{\Hdot}^{\si+\kap/2}}^{\frac{4}{2+\de_{2}}}\Sob{\tht^{(2)}}{{H}^{\si}}^{\frac{2\de_{2}}{2+\de_{2}}}\Sob{\bar{\tht}}{L^2}^{2}.
\end{align*}
An application of the Gronwall inequality and Holder's inequality then establishes uniqueness.

\appendix
\section{}\label{sect:app}
\subsection*{Proof of Lemma \ref{lem:commutator1}}
		By Bony paraproduct formula
		\begin{align}\notag
		   \mathcal{L}_\si(f,g,h)=\ L_{1}+L_{2}+L_{3},
		\end{align}
		where
		\begin{align*}
		&L_{1}=\sum_{k} \iint\Ax^{\si}{\chi_{k-3}(\xi-\eta)\hat{f}(\xi-\eta) \phi_k(\eta)\hat{g}(\eta)\overline{\hat{h}(\xi)}} d\eta d\xi , \\
		&L_{2}=\sum_{k}\iint\Ax^{\si}{\phi_k(\xi-\eta)\hat{f}(\xi-\eta) \chi_{k-3}(\eta)\hat{g}(\eta)\overline{\hat{h}(\xi)}} d\eta d\xi,  \\  \quad
		&L_{3}=\sum_{k}\iint\Ax^{\si}{\phi_k(\xi-\eta)\hat{f}(\xi-\eta) \widetilde{\phi}_k(\eta)\hat{g}(\eta)\overline{\hat{h}(\xi)}} d\eta d\xi ,
		\end{align*}
		and
		\[\til{\phi}_{k}\mathcal{F}(f)=\sum_{\abs{i-k}\le3}\mathcal{F}({\triangle_{i}}f).\]
		Observe that by the triangle inequality
		    \begin{align*}
		      |\mathcal{L}_\si(f,g,h)|\leq\mathcal{L}_\si(|f|,|g|,|h|).
		    \end{align*}
		We will treat the cases $L_1, L_2$ and $L_3$ separately.
		\subsubsection*{{Estimating} $L_1$\nopunct}:
		{The localizations present in this case imply}
		\begin{align*}
		\xe \in {\Bcal}_{k-3}, \quad \eta \in {\Acal}_k.
		\end{align*}
		Observe that $\chi_{k}(\xe)\phi_{k}(\eta)=0$, for all $\abs{k-j}\geq3$, whenever $\xi\in{\Acal}_j$.  Thus for $\si \in \RR$
		\begin{align}\notag
		|L_{1}| \le C\sum_{\abs{k-j}\le 2}
	\iint\Ae^{\si}\chi_{k-3}(\xi-\eta)|\hat{f}(\xi-\eta)| \phi_{k}(\eta)|\hat{g}(\eta)||\hat{h}(\xi)|d\eta d\xi.
		\end{align}
	    {By} the Cauchy-Schwarz inequality,  Young's convolution inequality, and Plancherel's theorem,
		\begin{align}
		{|L_{1}|}
		&{\le C\sum_{\abs{k-j}\le 2}\nrm{\chi_{k-3}\hat{f}}_{L^{1}}\nrm{{\lpk}g}_{\Hdot^{\si}}\nrm{h}_{L^{2}}\label{E:E1}}.
		\end{align}

		{Suppose  $\eps>0$. Since $f\in L^2$, the Cauchy-Schwarz inequality and Plancherel's theorem implies}
		\begin{align*}
		\nrm{\chi_{k-3}\hat{f}}_{L^{1}}\le \nrm{\mathbbm{1}_{\Ba}|\cdotp|^{\epsilon-1}}_{L^2}\nrm{|\cdotp|^{1-\epsilon}\chi_{k-3}\hat{f}}_{L^2}
		\le \frac{C}{\eps^{1/2}} 2^{\epsilon k}\nrm{{S_{k-3}}f}_{\Hdot^{1-\epsilon}}.
		\end{align*}
		{Upon returning to (\ref{E:E1}), it follows by the Cauchy-Schwarz inequality in $l^{2}$ and Bernstein's inequality that}
		\begin{align}\label{E:L1-estimate}
		\abs{L_{1}}&\le C2^{\eps j}\nrm{{S_{j-1}}f}_{\Hdot^{1-\epsilon}}\sum_{\abs{k-j}\le 2}2^{\si k}\nrm{\triangle_{k}g}_{L^2}\nrm{h}_{L^{2}} \notag\\&\le Ca_{j}2^{\epsilon j}\nrm{{S_{j-1}}f}_{\Hdot^{1-\epsilon}}\nrm{g}_{\Hdot^{\si}}\nrm{h}_{L^{2}},
		\end{align}
	    where
		    \begin{align*}
		    a_{j}({\si}):=\frac{\left(\sum_{\abs{k-j}\le 2}2^{2k{\si}}\nrm{{\lpk}g}_{L^2}^2\right)^{1/2}}{\nrm{g}_{\Hdot^{\si}}}.
		    \end{align*}

		\subsubsection*{{Estimating} $L_2$ \nopunct}:
		{The localizations in this case imply}
		\begin{align*}
		\eta \in {\Bcal}_{k-3}, \quad \xe \in {\Acal}_{k},
		\end{align*}
		Thus
		\begin{align*}
		|L_{2}| \le C\sum_{\abs{k-j}\le 2}
		\iint\abs{\xi-\eta}^{\si}\phi_k(\xi-\eta)|\hat{f}(\xi-\eta)|\chi_{k-3}(\eta)|\hat{g}(\eta)||\hat{h}(\xi)|\, d\eta\, d\xi. 
		\end{align*}
		{Suppose $\si>-1$, then as in (\ref{E:E1}),  the Cauchy-Schwarz inequality and Young's convolution inequality imply}
		\begin{align}\label{E:L2s}
		\abs{L_{2}}\le C\sum_{k}\nrm{\chi_{k-3}\hat{g}}_{L^{\frac{4}{3+{\si}}}}\nrm{\abs{\cdotp}^{\si}\phi_k\hat{f}}_{L^{\frac{4}{3-{\si}}}}\nrm{h}_{L^{2}}.
		\end{align}
		
		Suppose $\si <1$ and $\eps\in\RR$, by H\"older's inequality, we have
		\begin{align*}
		\nrm{ \chi_{k-3}\hat{g}}_{L^{\frac{4}{3+\si}}}
		    &\le \nrm{\mathbbm{1}_{\Ba}|\cdotp|^{-\si}}_{L^{\frac{4}{{\si}+1}}}\nrm{|\cdotp|^{\si}\chi_{k-3}\hat{g}}_{L^{2}}
	    	\le C2^{k(\frac{1}{2}-\frac{\si}{2})}\nrm{{S_{k-3}g}}_{\Hdot^{\si}}\\
		\nrm{\abs{\cdotp}^{\si}\phi_k\hat{f}}_{L^{\frac{4}{3-\si}}}
	    	&\le \nrm{\mathbbm{1}_{\An}|\cdotp|^{\epsilon+{\si}-1}}_{L^{\frac{4}{1-\si}}}\nrm{|\cdotp|^{1-\epsilon}\phi_k\hat{f}}_{L^{2}}
		    \le C2^{k(\epsilon+\frac{\si}{2}-\frac{1}{2})}\nrm{{\lpk}f}_{\Hdot^{1-\epsilon}}. 
		\end{align*}
		{Upon returning to (\ref{E:L2s}), and proceeding as for (\ref{E:L1-estimate}), one obtains}
		\begin{align}\label{E:L2-estimate}
		\abs{L_{2}}\le Cb_{j}2^{\epsilon j}\nrm{f}_{\Hdot^{1-\epsilon}}\nrm{{S_{j-1}g}}_{\Hdot^{\si}}\nrm{h}_{L^{2}},
		\end{align}
		where 
		\[
		b_{j}(\epsilon):=\frac{\left(\sum_{\abs{k-j}\le 2}2^{2k(1-\eps)}\nrm{{\lpk}f}_{L^2}^2\right)^{1/2}}{\nrm{f}_{\Hdot^{1-\epsilon}}}.
		\]
		\subsubsection*{Estimating $L_3$  \nopunct}:	
		{Since $\phi_{k}(\xe)\sum_{\abs{k-l}\le 3}\phi_{l}(\eta)=0$, for all  $j \ge k+6$, whenever $\xi\in{\Acal}_j$, the summation only occurs over the range $k\ge j-5$. The localizations in this case imply}
		\begin{align*}
		\xe \in {\Acal}_{k}, \quad \eta \in {\Acal}_{k-4,k+4},
		\end{align*}
	so that for ${\si} <1$ and $\eps<\si+1$, we have 
\begin{align*}
    \Ax^{\frac{{\si}+1-\eps}{2}}\le C\Axe^{1-\frac{\eps}{2}}\Ae^{\frac{{\si}-1}{2}}.
\end{align*}
This gives us
\begin{align*}
		{|L_{3}| \le C\sum_{k\ge j-5}
		\iint|\xi-\eta|^{1-\frac{\eps}{2}}\phi_k(\xi-\eta)|\hat{f}(\xi-\eta)||\eta|^{\frac{\si}{2}-1} \tilde{\phi}_k(\eta)|\hat{g}(\eta)|\Ax^{\frac{\eps+{\si}-1}{2}}|\hat{h}(\xi)| d\eta d\xi.}
		\end{align*}
The Cauchy-Schwarz inequality, Young's convolution inequality, and Bernstein's inequality imply
		\begin{align}\label{E:L3-estimate}
		{|L_{3}|\le}  & {C\sum_{k \ge j-5} \nrm{|\cdotp|^{1-\frac{\eps}{2}}\phi_k\hat{f}}_{L^2}\nrm{|\cdotp|^{\frac{\si-1}{2}}\tilde{\phi}_k\hat{g}}_{L^2}\nrm{|\cdotp|^{\frac{\eps+{\si}-1}{2}}\hat{h}}_{L^{1}}}\notag \\
		    \le & C\sum_{k \ge j-5}2^{\frac{\epsilon}{2} k}\nrm{{\lpk}\Lam^{1-\eps}f}_{L^2}2^{-(\frac{{\si}+1}{2})k}\nrm{{\tlpk}\Lam^{\si}g}_{L^2}2^{(\frac{\eps+{\si}+1}{2})j}\nrm{h}_{L^{2}} \notag\\
		    \le & Cc_{j}2^{\epsilon j}\nrm{f}_{\Hdot^{1-\epsilon}}\nrm{g}_{\Hdot^{\si}}\nrm{h}_{L^{2}},
		\end{align}
		where
		\[
		{c_{j}({\si},\epsilon)}=\frac{\left(\sum_{k\ge j-5}2^{-2(\frac{{\si}+1-\eps}{2})(k-j)} \nrm{{\lpk}\Lam^{\si}g}_{L^2}^2\nrm{{\lpk}\Lam^{1-\eps}f}_{L^2}^2\right)^{1/2}}{\nrm{g}_{\Hdot^{\si}}\nrm{f}_{\Hdot^{1-\epsilon}}}. 
		\]
Combining \eqref{E:L1-estimate}, \eqref{E:L2-estimate}, \eqref{E:L3-estimate} and the fact that $\mathcal{L}_{\si}(f,g,h)=\mathcal{L}_{\si}(g,f,h)$. completes the proof.
\qed
\subsection*{Proof of Lemma \ref{lem:commutator2a}}

By Plancherel's theorem, we have
\begin{align}\label{E:def:L:functionala}
   \mathcal{L}:= \lb [\triangle_j,g]f ,h\rb=\iint m(\xi,\eta)\hat{g}(\eta)\hat{f}(\xi-\eta)\overline{\hat{h}(\xi)}d\eta d\xi,
\end{align}
where
 \begin{align}\notag
        m(\xi,\eta):=\phi_{j}(\xi)-\phi_{j}(\xe).
    \end{align}
Observe that by the mean value theorem
    \begin{align}\notag
    |m(\xi,\eta)|\le \Ae 2^{-j}\Sob{\nabla \phi_{0}}{L^{\infty}}.
    \end{align}
Using this in (\ref{E:def:L:functionala}) and the fact that $\supp\hat{h}\subset\Acal_j$, we obtain
\begin{align*}
 \abs{\mathcal{L}}\le &C2^{-(1+\si)j}\iint|\xi|^{\si}|\hat{ f}(\xi-\eta)||\widehat{\Lam g}(\eta)||\hat{ h}(\xi)|d\eta d\xi,
\end{align*}
for any $\si\in\RR$. For $\si\in(-1,1)$, $\eps\in(0,2)$ such that $\si>\eps-1$, application of \cref{lem:commutator1} gives us 
\begin{align}\notag
\mathcal{L}\leq Cc_j2^{-(1+\si-\eps) j}\min\left\{\nrm{f}_{\Hdot^{1-\epsilon}}\nrm{\Lam g}_{\Hdot^{\si}},\nrm{\Lam g}_{\Hdot^{1-\epsilon}}\nrm{f}_{\Hdot^{\si}}\right\}\nrm{h}_{L^{2}}.
\end{align}
We set $\rho_1=\eps$ and $\rho_2=\si$ to complete the proof.
\qed

\section{Proof of Theorem \ref{T:transport}}\label{sect:app:thm}

For $\epsilon>0$, we consider the following artificial viscosity regularization of (\ref{E:mod:claw}): 
\begin{align}\label{E:forced-transport:reg}
{\partial_{t}\theta-\epsilon\De \tht+ \Div F_{q}(\tht)=-{\gam} \Lam^{\kap}\tht.}    
\end{align}
For $0\le t\le T$, define
\begin{align*}
    &{F}_{1}(\tht):=\gam\int^{t}_{0}e^{\epsilon\De(t-s)} \Lam^{\kap}\tht(s)\, ds,\\
    &{F}_{2}(\tht;q):=\int^{t}_{0}e^{\epsilon\De(t-s)}\Div F_{q}(\tht)\,ds.
\end{align*}

We have
\begin{align*}
   \Sob{F_1(\tht)(t)}{H^{\si_{c}}}&\le \frac{C}{\epsilon^{\frac{\kap}{2}}}\int_{0}^{t}\frac{1}{(t-s)^{\frac{\kap}{2}}}\Sob{\tht(s)}{H^{\si_{c}}}\,ds\\
   &\le \frac{CT^{1-\frac{\kap}{2}}}{\epsilon^{\frac{\kap}{2}}}\Sob{\tht}{L^{\infty}_{T}H^{\si_{c}}}.
\end{align*}

To estimate $\Sob{{F}_{2}(\tht;q)}{H^{\si_c}}$, we consider the two cases $\beta<1+\kap $ and $\beta\ge 1+\kap$ separately.
\subsubsection*{\textbf{Case: ${\be\ge 1+\kap}$} } By \cref{lem:commutator2} with $\rho_{1}=\kap$ and $\rho_{2}=\be/2$, $A_{\ell}=\Lam^{\be-2}\bdy_{\ell}^{\perp}$, proceeding as in (\ref{est:sobolev:de'}), we obtain
\begin{align*}
    \Sob{F_{2}(\tht;q)(t)}{\Hdot^{\si_c}}&\le \frac{C}{\epsilon^{\frac{\be+2}{4}}}\int_{0}^{t}\frac{1}{(t-s)^{\frac{\be+2}{4}}}\Sob{[A_{\ell},\bdy_{\ell}\tht]{q}}{\Hdot^{\frac{\be}{2}-\kap}}\,ds\\ &\le \frac{C}{\epsilon^{\frac{\be+2}{4}}}T^{\frac{2-\be}{4}}\Sob{\tht}{L^{\infty}_{T}\Hdot^{\si_{c}}}\Sob{{q}}{L^{\infty}_{T}\Hdot^{\frac{\be}{2}}}.
\end{align*}
Similarly, by \cref{lem:commutator2} with $\rho_1=\rho_2=\kap$, we have
\begin{align*}
    \Sob{F_{2}(\tht;q)(t)}{L^2} \le CT\Sob{\tht}{L^{\infty}_{T}\Hdot^{\si_{c}}}\Sob{{q}}{L^{\infty}_{T}\Hdot^{\kap}}.
\end{align*}
\subsubsection*{\textbf{Case: ${\be< 1+\kap}$} } In this case, we have
\begin{align*}
    \Sob{F_{2}(\tht;q)(t)}{H^{\si_c}}&\le \int_{0}^{t}\left\{1+\frac{C}{\epsilon^{\frac{\si_{c}}{2}}}\frac{1}{(t-s)^{\frac{\si_{c}}{2}}}\right\}\Sob{A_{\ell}{q}\,\bdy_{\ell}\tht}{L^2}\,ds\\& \le C\left(T+\frac{T^{\frac{2-\si_c}{2}}}{\epsilon^{\frac{\si_c}{2}}}\right)\Sob{\tht}{L^{\infty}_{T}\Hdot^{\si_{c}}}\Sob{{q}}{L^{\infty}_{T}\Hdot^{\kap}}.
\end{align*}
Using Picard's theorem (cf. \cite{Lemarie-Rieusset2002}), there exists a unique solution $\tht^{\epsilon}$ to (\ref{E:forced-transport:reg}) such that $\tht^{\epsilon}\in L^{\infty}_{T^{\epsilon}}H^{\si_c}$ for some time $T^{\epsilon}>0$. Owing to the uniform estimate in (\ref{est:sobolev:summary}), we can conclude that
\begin{align*}
    T^{\epsilon}=T,\quad \text{for all}\ \epsilon>0.
\end{align*}
Using similar methods as above, it is easy to see that $\Sob{\partial_{t}\theta^{\epsilon}}{L^{\infty}_{T}H^{\si_{c}-2}}$ is bounded uniformly in $\epsilon$. Application of Aubin-Lions theorem (cf. \cite{ConstantinFoiasBook1988}) guarantees the existence of a weak limit $\tht$ in $L^{\infty}_{T}H^{\si_c}$. It is then straightforward to show that $\tht$ is a weak solution of (\ref{E:mod:claw}). This completes the proof. 
\qed

\newcommand{\etalchar}[1]{$^{#1}$}
\providecommand{\bysame}{\leavevmode\hbox to3em{\hrulefill}\thinspace}
\providecommand{\MR}{\relax\ifhmode\unskip\space\fi MR }
\providecommand{\MRhref}[2]{%
  \href{http://www.ams.org/mathscinet-getitem?mr=#1}{#2}
}
\providecommand{\href}[2]{#2}

\vspace{.3in}
\begin{multicols}{2}

\noindent Michael S. Jolly\\ 
{\footnotesize
Department of Mathematics\\
Indiana University-Bloomington\\
Web: \url{https://msjolly.pages.iu.edu/}\\
 Email: \url{msjolly@indiana.edu}}\\[.2cm]

\noindent Anuj Kumar\\ 
{\footnotesize
Department of Mathematics\\
Indiana University-Bloomington\\
Web: \url{https://math.indiana.edu/about/graduate-students}\\
 Email: \url{kumar22@iu.edu}} \\[.2cm]

\columnbreak 

\noindent Vincent R. Martinez\\
{\footnotesize
Department of Mathematics and Statistics\\
CUNY-Hunter College\\
Web: \url{http://math.hunter.cuny.edu/vmartine/}\\
 Email: \url{vrmartinez@hunter.cuny.edu}}

\end{multicols}

\end{document}